\documentclass[a4paper]{amsart}

\usepackage[utf8]{inputenc}
\usepackage[T1]{fontenc}
\usepackage{lmodern, enumerate}
\usepackage{amssymb,amsxtra,amscd,amsmath}
\usepackage{amsthm}
\usepackage{amsfonts}
\usepackage{tikz-cd}
\usepackage{bbm}
\usepackage[all]{xy}
\usepackage{xcolor}
\usepackage{nicefrac,mathtools}
\usepackage[inline]{enumitem}
\usepackage{microtype}
\usepackage{comment}
\usepackage[
  pdftitle={Essential groupoid amenability and nuclearity of groupoid C*-algebras},
  pdfauthor={Alcides Buss, Diego Mart\'{i}nez},
  pdfsubject={Mathematics},
  colorlinks=true, linkcolor=blue, linkbordercolor=blue, citecolor=red, citebordercolor=red, urlcolor=blue, linktocpage=true
]{hyperref}
\usepackage[capitalise, nameinlink, noabbrev, nosort]{cleveref} % see the cleveref documentation to find out what these options do
\usepackage[lite]{amsrefs}

\newtheorem{alphthm}{Theorem}			%letter numbering

\newtheorem{alphcor}[alphthm]{Corollary}           %letter numbering
\newtheorem{alphprop}[alphthm]{Proposition}

\setlist[enumerate]{font=\normalfont}
\crefname{enumi}{}{} % no name for list items
\crefname{enumi}{}{} % no name for nested list items

\numberwithin{equation}{section}
\theoremstyle{plain}
\newtheorem{theorem}[equation]{Theorem}
\newtheorem{lemma}[equation]{Lemma}
\newtheorem{proposition}[equation]{Proposition}
\newtheorem{corollary}[equation]{Corollary}
\newtheorem{claim}[equation]{Claim}
\theoremstyle{definition}
\newtheorem{definition}[equation]{Definition}
\theoremstyle{remark}
\newtheorem{remark}[equation]{Remark}
\newtheorem{notation}[equation]{Notation}
\newtheorem{example}[equation]{Example}

\crefname{theorem}{Theorem}{Theorems}
\crefname{proposition}{Proposition}{Propositions}
\crefname{lemma}{Lemma}{Lemmas}
\crefname{claim}{Claim}{Claims}
\crefname{remark}{Remark}{Remarks}
\crefname{corollary}{Corollary}{Corollaries}
\crefname{question}{Question}{Questions}
\crefname{conjecture}{Conjecture}{Conjectures}
\crefname{fact}{Fact}{Facts}
\crefname{claim}{Claim}{Claims}
\crefname{case}{Case}{Cases}

\DeclareMathOperator{\supp}{\mathrm{supp}}
\DeclareMathOperator{\id}{\mathrm{id}}

\DeclareMathOperator{\ess}{ess}

\renewcommand*{\ess}{\mathrm{ess}}
\newcommand{\sing}{\mathrm{sing}}

\newcommand*{\Star}{\(^*\)\nobreakdash-}
\newcommand*{\Z}{\mathbb Z}
\newcommand*{\N}{\mathbb N}
\newcommand*{\C}{\mathbb C}

\renewcommand*{\L}{\mathcal L}
\renewcommand*{\H}{\mathcal H}

\newcommand*{\Hilb}{\mathcal H} 

\newcommand*{\cont}{C}%continuous functions
\newcommand*{\contz}{\cont_0}%continuous functions vanishing at infinity
\newcommand*{\contc}{\cont_c}%continuous functions with compact support
\newcommand*{\contb}{\cont_b}%continuous bounded functions
\newcommand*{\borelb}{B_b}

\newcommand*{\meager}{M}

\newcommand*{\Mat}{\mathbb M}%matrices
\newcommand*{\M}{\mathcal M} % multiplier algebra

\newcommand*{\NN}{\mathcal N}

\newcommand*{\Id}{\textup{id}}%identity
\newcommand*{\CCC}{\mathbb{C}}% complex numbers

\newcommand*{\braket}[2]{\langle#1\!\mid\!#2\rangle}

\newcommand*{\sbe}{\subseteq} % inclusion

\newcommand*{\cstar}{\texorpdfstring{$C^*$\nobreakdash-\hspace{0pt}}{*-}}
\newcommand*{\into}{\hookrightarrow}
\newcommand*{\onto}{\twoheadrightarrow}
\newcommand*{\red}{\text{red}}

\renewcommand*{\max}{\mathrm{max}}

\newcommand{\B}{\mathbb B}

\newcommand*{\s}{\mathrm d} % source map
\newcommand*{\rg}{\mathrm r}% range map

\newcommand*{\norm}[1]{\left\|#1\right\|}
\newcommand*{\abs}[1]{\left|#1\right|}

\newcommand*{\auxalg}[1]{\mathfrak{A}_c\left(#1\right)}
\newcommand*{\auxalgK}[1]{\mathfrak{A}_K\left(#1\right)}
\newcommand*{\auxalgc}[1]{\mathfrak{A}^\infty_c\left(#1\right)}

\newcommand*{\algalg}[1]{\mathfrak{C}_c\left(#1\right)}

\newcommand*{\essalgalg}[1]{\mathfrak{C}_c^{\rm ess}\left(#1\right)}
\newcommand*{\redalg}[1]{C_{\rm red}^*\left(#1\right)} % reduced \cstar{}algebra of an inverse semigroup
\newcommand*{\fullalg}[1]{C_{\rm max}^*\left(#1\right)} % maximal \cstar{}algebra of an inverse semigroup

\newcommand*{\borelalg}[1]{B_{c}\left(#1\right)}
\newcommand*{\borelessalg}[1]{B^{\rm ess}_{c}\left(#1\right)}
\newcommand*{\fullborel}[1]{B_{\rm max}\left(#1\right)}
\newcommand*{\redborel}[1]{B_{\rm red}\left(#1\right)}
\newcommand*{\essmaxborel}[1]{B_{\rm ess,max}\left(#1\right)}
\newcommand*{\essborel}[1]{B_{\rm ess}\left(#1\right)}

\newcommand*{\essalg}[1]{C_{\rm ess}^*\left(#1\right)} % essential \cstar{}algebra of an inverse semigroup
\newcommand*{\essmaxalg}[1]{C_{\rm ess,max}^*\left(#1\right)} % essential maximal \cstar{}algebra of an inverse semigroup
\begin{document}

%%%%%%%%%%%%%%%%%%%%%%%%%%%%%%%%%%%%%%%%%%%%%%%%%%%%%%%%%%%%%%%%%%%%%%%
% TITLE AND AUTHORS
%%%%%%%%%%%%%%%%%%%%%%%%%%%%%%%%%%%%%%%%%%%%%%%%%%%%%%%%%%%%%%%%%%%%%%%
\title[Essential groupoid amenability and nuclearity]{Essential groupoid amenability and nuclearity of groupoid C*-algebras}

\author[Alcides Buss]{Alcides Buss $^{1}$}
\address{Departamento de Matem\'atica, Universidade Federal de Santa Catarina, 88.040-900 Florian\'opolis-SC, Brazil}
\email{alcides.buss@ufsc.br}

\author[Diego Mart\'{i}nez]{Diego Mart\'{i}nez $^{2}$}
\address{Department of Mathematics, KU Leuven, Celestijnenlaan 200B, 3001 Leuven, Belgium.}
\email{diego.martinez@kuleuven.be}

%%%%%%%%%%%%%%%%%%%%%%%%%%%%%%%%%%%%%%%%%%%%%%%%%%%%%%%%%%%%%%%%%%%%%%%
% ABSTRACT AND META-INFO
%%%%%%%%%%%%%%%%%%%%%%%%%%%%%%%%%%%%%%%%%%%%%%%%%%%%%%%%%%%%%%%%%%%%%%%
\begin{abstract}
We give an alternative construction of the essential \cstar algebra of an \'etale groupoid, along with an ``amenability'' notion for such groupoids that is implied by the nuclearity of this essential \cstar algebra.
In order to do this we first introduce a maximal version of the essential \cstar algebra, and prove that every function with dense co-support can only be supported on the set of ``dangerous'' arrows.
We then introduce an essential amenability condition for a groupoid, which is (strictly) weaker than its (topological) amenability.
As an application, we describe the Bruce-Li algebras arising from algebraic actions of cancellative semigroups as exotic essential \cstar algebras.
\end{abstract}

% 20L05: groupoids
% 46L52: noncommutative function spaces
% 46L55: noncommutative dynamical systems
% 22A22: topological groupoids
% 46L05: general theory of C*-algebras


\thanks{{$^{1}$} Partially supported by CNPq, CAPES -- Brazil \& Humboldt Fundation -- BRA/1146429 and Germany’s Excellence Strategy -- EXC 2044 -- 390685587,  Mathematics Münster -- Dynamics -- Geometry -- Structure, and by Fapesc - SC - Brazil.}

\thanks{{$^{2}$} Supported by project G085020N funded by the Research Foundation Flanders (FWO)}

%%%%%%%%%%%%%%%%%%%%%%%%%%%%%%%%%%%%%%%%%%%%%%%%%%%%%%%%%%%%%%%%%%%%%%%
% TITLE
%%%%%%%%%%%%%%%%%%%%%%%%%%%%%%%%%%%%%%%%%%%%%%%%%%%%%%%%%%%%%%%%%%%%%%%
\maketitle
 \tableofcontents

%%%%%%%%%%%%%%%%%%%%%%%%%%%%%%%%%%%%%%%%%%%%%%%%%%%%%%%%%%%%%%%%%%%%%%%
% INTRODUCTION
%%%%%%%%%%%%%%%%%%%%%%%%%%%%%%%%%%%%%%%%%%%%%%%%%%%%%%%%%%%%%%%%%%%%%%%
\section{Introduction} \label{sec:intro}

\cstar algebras and, in particular, the \emph{nuclear} ones, have seen quite an extraordinary amount of work put into them in the last few decades, see \cites{Elliott2015OnTC,tikuisis-white-winter-2017,Renault2008CartanSI,kumjian-diagonals-1986,white_cartan_2018,winter_nuclear_2010,li-classifiable-cartans-2020,brix-gonzalez-hume-li-2025} for just a taste of the breadth and depth of these studies.
Recall that a \cstar algebra \(A\) is \emph{nuclear} when the identity map \(\id \colon A \to A\) can be approximated in the point-norm topology by maps that factor through matrix algebras, \emph{i.e.}\ there are \emph{contractive} and \emph{completely positive} maps
\[
    \varphi_i \colon A \to \Mat_{k\left(i\right)}(\C) \;\; \text{ and } \;\; \psi_i \colon \Mat_{k\left(i\right)}(\C) \to A
\]
such that \(\psi_i(\varphi_i(a)) \to a\) for all \(a \in A\) (the convergence is in norm).

This concept has strong ties to amenability in the setting of groups and groupoids. In the case of discrete groups, nuclearity of the reduced group \cstar{}algebra $\redalg \Gamma$ is known to be equivalent to the amenability of $\Gamma$ (see \cite{Brown-Ozawa:Approximations}*{Theorem 2.6.8}). In this paper, we extend this relationship to the realm of étale groupoids by introducing a new notion, \emph{essential amenability}, which allows us to characterize the nuclearity of an important class of \cstar{}algebras associated with \emph{non-Hausdorff groupoids.}

A \emph{groupoid} is a small category \(G\) with inverses. For instance, any group \(\Gamma\) is automatically a groupoid with exactly one object (the unit of the group). Likewise, any topological space \(X\) is also a groupoid when seen as having only identity arrows (\emph{i.e.}\ objects).
In this way, groupoids include both groups and their actions on spaces, and thus are particularly interesting objects when it comes to constructing \cstar algebras out of them.
We say that \(G\) is \emph{\'etale} if it is equipped with a topology that turns all operations into continuous maps and, in addition, makes the source and range maps into local homeomorphisms (see \cite{SimsNotes2020} and/or \cref{sec:pre:grpds} below). Every groupoid here considered will be \'etale, but \emph{not} necessarily Hausdorff.
However, its unit space $X \coloneqq G^{(0)}$ will here be always considered to be locally compact and Hausdorff, so that $G$ is also locally compact and locally Hausdorff. This, seemingly subtle, distinction is, in fact, quite crucial, at least when pondering questions related to the \cstar algebras constructed out of \(G\).

Just as is the case for discrete groups or tensor products, one may construct quite a zoo of \cstar algebras out of an \'etale groupoid. 
Nevertheless, there are two special \cstar algebras that stand out.
The \emph{maximal} \cstar algebra of \(G\), denoted by \(\fullalg{G}\), is some universal \cstar algebra of \(G\) (see \cref{sec:pre:algs} for details). Roughly speaking, $\fullalg G$ is the `largest' \cstar{}algebra one can assign to $G$, and it encodes its representation theory.
More precisely, $\fullalg G$ is the enveloping \cstar{}algebra of the convolution algebra $\algalg G$, consisting of linear combinations of functions $G\to \C$ that are compactly supported and continuous on a Hausdorff open subset of $G$. These functions are not necessarily continuous on $G$, which is why we use the different notation $\algalg G$ instead of the more standard $\contc(G)$.

All other \cstar{}algebras assigned to $G$ will be quotients of $\fullalg G$. For instance, the \emph{reduced} \cstar algebra is the quotient of it that one constructs as the image of the \emph{left regular representation}, that we view as a quotient map
\begin{equation} \label{eq:intro:left-reg-rep}
    \Lambda_P \colon \fullalg{G} \twoheadrightarrow \redalg{G} \hookrightarrow \B\left(\oplus_{x \in X} \ell^2\left(Gx\right)\right)
\end{equation}
(cf.\ \cref{sec:pre:algs}).
However, if the groupoid \(G\) is \emph{not} Hausdorff then the story might not end up here.
Exel and Pitts introduced in \cite{ExelPitts} the \emph{essential} \cstar algebra of an \'etale (topologically free) groupoid, which they denoted by \(\essalg{G}\).
Kwa\'{s}niewski and Meyer~\cite{KwasniewskiMeyer-essential-cross-2021} then expanded the definition of the essential \cstar algebra in order to accommodate for non-topologically free groupoids as well.
In this text we follow their point of view.

The \cstar algebra \(\essalg{G}\) is, then, a quotient of \(\redalg{G}\), but it has the ``right ideal structure''. Indeed, even for topologically free groupoids, \(\redalg{G}\) might contain non-zero ideals \(I \subseteq \redalg{G}\) that trivially intersect the sub-algebra \(\contz(X)\) of continuous functions on the unit space \(X \subseteq G\). 
Of course, since nuclearity passes to quotients, it follows that if one wishes to guarantee the nuclearity of \(\essalg{G}\) it suffices to guarantee the nuclearity of \(\redalg{G}\) (see \cite{BussMartinez:Approx-Prop}*{Theorem 6.10} for such a sufficient condition, or \cref{thm-intro} below). The converse, however, fails. In this paper we characterize when \(\essalg{G}\) is nuclear in terms of a certain \emph{essential amenability} of \(G\) (see \cref{def:ess-ame}).
In order to do this we first introduce a \emph{maximal} analogue of \(\essalg{G}\), which we denote by \(\essmaxalg{G}\), and define it as a \cstar algebra enjoying a certain universal property (cf.\ \cref{def:max ess alg}).
It thus follows that \(\essmaxalg{G}\) is a quotient of \(\fullalg{G}\) and quotients onto \(\essalg{G}\). Nonetheless, it does \emph{not} quotient onto (nor is it quotiented onto by) \(\redalg{G}\), in general. 
Under the mild assumption that $G$ is covered by countably many open bisections (\emph{e.g.}\ if $G$ is second countable or $\sigma$-compact), we prove that the regular representation factors through a faithful representation 
$$\Lambda_{E_\mathcal{L}}\colon \essalg G\into \B\left(\oplus_{x \in X \setminus D} \ell^2\left(Gx\right)\right),$$
where $D\subseteq G$ denotes the set of \emph{dangerous} arrows, see \cref{def:dangerous arrow}. These are precisely the arrows that witness the non-Hausdorff nature of $G$. We may also view $\Lambda_{E_\mathcal{L}}$ as a quotient map $\essmaxalg G\onto \essalg G$.
It follows that we now have four \cstar algebras of interest, which fit into the following commutative diagram (cf.\ \cref{cor:diagram}, and compare with \eqref{eq:intro:left-reg-rep}):\footnote{\; The, admittedly rather confusing, notation for the \emph{left regular representations} \(\Lambda^{\rm max}_{E_\mathcal{L},P}\), \(\Lambda_P\), \(\Lambda_{E_\L}\) and \(\Lambda^{\rm red}_{E_\mathcal{L},P}\) in the diagram does follow some logic. Moreover, the set \(D \subseteq G\) denotes the set of \emph{dangerous} arrows, see the last paragraph of the introduction and/or \cref{def:dangerous arrow}.}
\begin{center}
\begin{tikzcd}[scale=50em]
\fullalg{G} \arrow[two heads]{d}{\Lambda^{\rm max}_{E_\mathcal{L},P}} \arrow[two heads]{r}{\Lambda_P} & \redalg{G} \arrow[two heads]{d}{\Lambda^{\rm red}_{E_\mathcal{L},P}} \arrow[hook]{r}{} & \B\left(\oplus_{x \in X} \ell^2\left(Gx\right)\right) \\
\essmaxalg{G} \arrow[two heads]{r}{\Lambda_{E_\mathcal{L}}} & \essalg{G} \arrow[hook]{r}{} & \B\left(\oplus_{x \in X \setminus D} \ell^2\left(Gx\right)\right).
\end{tikzcd}
\end{center}

Heuristically, we see the \cstar algebras on the ``middle-column'' as the ``reduced'' half of the diagram, since they are naturally constructed as concrete \cstar subalgebras acting on some Hilbert spaces. Meanwhile, the \cstar algebras on the left are the ``maximal'' ones. Likewise, the ``top-row'' is the ``weak'' part; whereas the bottom-row is the ``essential'' part of the diagram.
The main objective of the paper is to study when the algebras \(\redalg{G}\) and \(\essalg{G}\) are nuclear, and the relation between these phenomena and the maps \(\Lambda_P\) and \(\Lambda_{E_\L}\) being isomorphisms.
\begin{alphthm}[cf.\ \cref{thm:wk-cont-nuc,thm:wk-cont-nuc-ess}] \label{thm-intro}
    Let \(G\) be an \'etale groupoid with locally compact Hausdorff unit space. Suppose it can be covered by countably many open bisections.
    \begin{enumerate}[label=(\roman*)]
        \item \(\redalg{G}\) is nuclear if and only if \(G\) is amenable. In such case, the regular representation \(\Lambda_P \colon \fullalg{G} \to \redalg{G}\) is an isomorphism.
        \item If \(\essalg{G}\) is nuclear then \(G\) is essentially amenable.
    \end{enumerate}    
\end{alphthm}

There are two new key ideas that allow us to prove \cref{thm-intro}.
The first is the realization that elements \(a \in \algalg{G}\), even though \emph{not} continuous as functions on \(G\), can be understood as continuous functions on the arrows of \(G\) that do not witness the non-Hausdorffness of \(G\). In turn, this follows from the following proposition.

\begin{alphprop}[cf.\ \cref{lemma:support is in d}] \label{prop-intro}
    Let \(G\) be an \'etale groupoid with locally compact Hausdorff unit space. Suppose that \(C \subseteq G\) is dense and \(a \in \algalg{G}\) is such that \(a(c) = 0\) for all \(c \in C\). If \(a(g) \neq 0\) then \(g \in D\).
\end{alphprop}

The above can be understood as a generalization of the well-known fact that a continuous function that is \(0\) on a dense subset must be \(0\) everywhere.
The usefulness of \cref{prop-intro} for us is the fact that it allows to represent \(\essalg{G}\) in a relatively straightforward way (cf.\ \cref{prop:rep red and ess calgs:ess}).

The second key idea in the proof of \cref{thm-intro} is the construction of certain Borel \emph{Herz-Schur multipliers} (see \cref{prop:schur multiplier} below). In particular, the non-Hausdorffness of \(G\) necessitates considering \emph{Borel} functions on \(G\), as the image of these multipliers is always \emph{Borel} (as opposed to lying in \(\algalg{G}\)).\footnote{\, This remains true even if one starts with functions in \(\algalg G\), which was the source of an error in a previous version of this paper; see \cref{ex:trivial-action-ame}.}
This distinction is crucial, as the canonical functions on \(G\) arising from the nuclearity of \(\redalg{G}\) (or \(\essalg{G}\)) are, in general, \emph{not} elements of \(\algalg G\) (see \cref{lemma:taking norm is continuous}). Consequently, to establish \cref{thm-intro}, one must work with Borel functions, which naturally leads to defining a \emph{Borel} (\emph{essential}) \emph{amenability} (cf.\ \cref{def:ame,def:ess-ame}). 

Remarkably, these Borel notions are equivalent to their topological counterparts (cf.\ \cref{thm:wk-cont-nuc,thm:wk-cont-nuc-ess}). In fact, the topological nature of Borel amenability was already observed by Renault in the main result of \cite{renault-2013}. However, the concept of a \emph{topological approximate invariant mean} introduced in \cite{renault-2013} requires refinement; see \cref{rem:mistakes}.

As direct consequences of \cref{thm-intro}, we have the following results for classical crossed products by group actions, which might prove of independent interest. 

\begin{alphcor}[cf.\ \cref{thm:wk-cont-nuc}] \label{thm-intro-borel}
    Let \(\alpha \colon \Gamma \curvearrowright X\) be a discrete group acting on a compact Hausdorff space. The following assertions are equivalent.
    \begin{enumerate}[label=(\roman*)]
        \item \(C(X) \rtimes_{\rm red} \Gamma\) is nuclear.
        \item \(\borelb(X) \rtimes_{\rm red} \Gamma\) is nuclear, where \(\borelb(X)\) denotes the \cstar{}algebra of bounded Borel functions on \(X\).
        \item \(\alpha\) is amenable.
        \item \(\alpha\) is Borel amenable, that is, there are finitely supported functions \(\xi_i \colon \Gamma \to \borelb(X)_+\) such that \(\sup_{x \in X} \sum_{\gamma \in \Gamma} \xi_i(\gamma)(x)^2 \leq 1\) and
        \[
            \sup_{\gamma \in F, \, x \in X} \abs{\sum_{\rho \in \Gamma} \overline{\xi_i\left(\rho^{-1}\right)\left(x\right)} \, \xi_i\left(\rho^{-1} \gamma\right)\left(x\right) - 1} \to 0
        \]
        for all finite sets \(F \subseteq \Gamma\).\footnote{\, The difference between \emph{Borel} amenability and the more common \emph{topological} amenability (cf.\ \cite{brown_c*-algebras_2008}*{Definition 4.3.5}) is the freedom of the functions \(\xi_i\) to be Borel, as opposed to continuous. The required invariance conditions are the same.}
    \end{enumerate}
    In such case, the canonical quotient maps \(C(X) \rtimes_{\rm max} \Gamma \twoheadrightarrow C(X) \rtimes_{\rm red} \Gamma\) and \(\borelb(X) \rtimes_{\rm max} \Gamma \twoheadrightarrow \borelb(X) \rtimes_{\rm red} \Gamma\) are injective.
\end{alphcor}

As an application of our methods, we also derive a condition for coincidence of essential and non-essential \cstar{}algebras of an amenable étale groupoid.

\begin{alphcor}[cf.\ \cref{cor:simple-red-eq-ess}] \label{thm-intro-red-eq-ess}
    Let \(G\) be amenable. Suppose that \(\{a \in \auxalgc{G} \mid \supp(a) \text{ is meager}\} = 0\).
    Then \(\fullalg{G} \cong \essalg{G}\) via the canonical quotient map. In particular, \(\fullalg{G}\) is simple if and only if \(G\) is minimal and topologically free.
\end{alphcor}

\

The paper is organized as follows. In \cref{sec:pre} we recap some (rather brief) preliminaries needed for the rest of the paper. \cref{sec:ess-reps-algs} is dedicated to showing \cref{prop-intro} and its main consequences, particularly a rather straightforward description of a faithful representation of \(\essalg{G}\), cf.\ \cref{prop:rep red and ess calgs:ess}.
\cref{sec:aux-alg} studies some extra Borel \cstar{}algebras needed for the study of amenability, which is carried out in \cref{sec:ess-ap}.
Moreover, in \cref{sec:wk-cont} we state and prove \cref{thm-intro}.
Lastly, \cref{sec:simplicity} proves \cref{thm-intro-red-eq-ess}, whereas \cref{sec:bruce-li} studies how the machinery here developed applies to the \cstar algebras studied in \cite{bruce-li:2024:alg-actions}.

\

\textbf{Conventions:} \(G\) denotes an \'etale groupoid with locally compact Hausdorff unit space \(X\coloneqq G^{(0)} \subseteq G\). The domain and range maps will be denoted by \(\s,\rg\colon G\to X\).
Moreover, $\B(\Hilb{})$ denotes the set of linear bounded operators on the Hilbert space $\Hilb{}$. When we write that $K$ is compact we mean that $K$ is a possibly non-Hausdorff compact topological space, that is, every open cover of $K$ admits a finite subcover. 
\vskip 0,5pc
\textbf{Acknowledgements.} We are grateful to Chris Bruce for insightful discussions related to \cite{bruce-li:2024:alg-actions}. We also thank Sergey Neshveyev and Bartosz Kwa\'sniewski for drawing our attention to \cites{neshveyevetal2023,bkmk-2024-twited-grpds}, respectively, and for their helpful comments.
We are further indebted to Ralf Meyer for valuable discussions concerning essential groupoid \(\mathrm{C}^*\)-algebras, and to Rufus Willett for pointing out an example involving crossed products (cf.\ \cref{exa:crossed-product-l-infty}) that demonstrates the limitations of extending \cref{thm-intro-borel}. 
Finally, we express our deepest gratitude to Xin Li for identifying errors in an earlier version of this paper. Interestingly, these turned out to be closely related to mistakes we later discovered in \cite{renault-2013}.
\vskip 0,5pc
\textbf{Note.} During the preparation of different versions of this paper, the independent work \cite{brix-gonzalez-hume-li-2025} was completed. While there is some overlap in results -- particularly concerning \cref{thm:wk-cont-nuc} -- the approaches and the remaining conclusions are substantially different.

%%%%%%%%%%%%%%%%%%%%%%%%%%%%%%%%%%%%%%%%%%%%%%%%%%%%%%%%%%%%%%%%%%%%%%%
% PRELIMINARIES
%%%%%%%%%%%%%%%%%%%%%%%%%%%%%%%%%%%%%%%%%%%%%%%%%%%%%%%%%%%%%%%%%%%%%%%
\section{Preliminaries}\label{sec:pre}
In this section we quickly cover all the necessary background for the paper. We are, thus, particularly interested in groupoids (cf.\ \cref{sec:pre:grpds}) and in their \cstar algebras (cf.\ \cref{sec:pre:algs}).

\subsection{\'Etale groupoids} \label{sec:pre:grpds}
Groupoid literature is vast, but we refer the reader to \cite{SimsNotes2020} for a (now rather canonical) resource. Recall that a \emph{groupoid} is a small category with inverses, that is, \(G\) is a set equipped with subsets \(G^{(0)} \subseteq G\) and \(G^{(2)} \subseteq G \times G\) of \emph{units} and \emph{composable pairs} respectively, along with an associative \emph{multiplication map}
\[
    \cdot \colon G^{(2)} \to G, \;\; \left(g, h\right) \mapsto gh
\]
and a bijective \emph{inverse map}
\[
    ^{-1} \colon G \to G, \;\; g \mapsto g^{-1}
\]
such that
\begin{enumerate}[label=(\roman*)]
    \item \(gg^{-1}g = g\) and \((gh)^{-1} = h^{-1} g^{-1}\) for all \((g,h) \in G^{(2)}\); and
    \item \((g,h) \in G^{(2)}\) if and only if \(g^{-1}g = hh^{-1}\).
\end{enumerate}
In particular it follows that \((x,x) \in G^{(2)}\) and \(xx = x\) for every unit \(x \in G^{(0)}\). For the sake of readability, we will usually denote the set of units by \(X \coloneqq G^{(0)} \subseteq G\).
The \emph{domain} and \emph{range} maps \(\s, \rg \colon G \to X\) are defined to be
\[
    \s\left(g\right) \coloneqq g^{-1} g \; \text{ and } \; \rg\left(g\right) \coloneqq gg^{-1}
\]
respectively.
We say that \(G\) is a \emph{topological} groupoid if it is endowed with a topology such that the multiplication, inverse, domain and range maps are all continuous. We say that a subset \(s \subseteq G\) is a \emph{bisection} if both the domain and range maps restrict to homeomorphisms from \(s\) onto certain subsets of \(X\). Likewise, we say that \(G\) is \'etale if its topology can be generated by open bisections.
Morally speaking, \'etale groupoids are the equivalent of discrete groups, or, rather, discrete group \emph{actions} on topological spaces \(X\).
Lastly, we say that \(G\) is \emph{locally compact} if it is locally compact as a topological space.
We do \emph{not} require \(G\) to be globally a Hausdorff space, but only locally so. Indeed, the unit space \(X\) will usually be a locally compact Hausdorff space. Thus, any open bisection \(s \subseteq G\) is homeomorphic to some open subset of \(X\) (actually, the subset is \(s^*s\), see \eqref{eq:convention about AB}).
Throughout the paper, given \(A, B \subseteq G\) we will let
\begin{equation} \label{eq:convention about AB}
    A B \coloneqq \left\{ab \mid a \in A, b \in B \text{ and } (a, b) \in G^{(2)}\right\} \; \text{ and } \; A^* \coloneqq \left\{a^{-1} \mid a \in A\right\}.
\end{equation}
Likewise, if \(A = \{a\}\) is a singleton then \(a B \coloneqq \{a\} B\) and, similarly, \(B a \coloneqq B \{a\}\). In particular, \(Gx\) denotes all the arrows in \(G\) starting at the unit \(x \in X \subseteq G\), and \(xGx\) is the (necessarily discrete) isotropy group at \(x\).\footnote{\, In groupoid literature it is usually the case that \(Gx\) and \(xG\) are denoted by \(G_x\) and \(G^x\) respectively. We find the former notation to be both more coherent and readable.}
Since \(G\) will be \'etale throughout, it follows that the sets \(Gx, xG\) and \(xGx\) are all discrete (when equipped with the subspace topology).
It also follows from the definitions that if \(s, t \subseteq G\) are open bisections then so are \(s^*\) and \(st\).
In particular, it follows that the set of open bisections of \(G\) forms an \emph{inverse semigroup}. This duality has been used extensively, see \cites{Renault2008CartanSI,SimsNotes2020,Buss-Exel-Meyer:Reduced,BussMartinez:Approx-Prop,ExelPitts,Exel:Inverse_combinatorial}, and we will use it here as well (see \cref{notation:cont functions as sums}).
We end the groupoid preliminaries with the following, which we only state for future use. The proof is well known, and trivial to check.
\begin{lemma} \label{lemma:compact set is covered by fin many bisections}
    Let \(G\) be an \'etale groupoid. If \(K\sbe G\) is compact then there are finitely many bisections \(s_1, \dots, s_p \subseteq G\) such that \(K \subseteq s_1 \cup \dots \cup s_p\).
\end{lemma}

\subsection{Maximal, reduced and essential groupoid algebras} \label{sec:pre:algs}
In this subsection we recall the construction of several different canonical \cstar algebras associated to an \'etale groupoid \(G\). Standard resources for this include \cites{SimsNotes2020,KwasniewskiMeyer-essential-cross-2021,ExelPitts}, among many others.
\begin{definition} \label{def:grpd-alg}
    Let \(G\) be an \'etale groupoid with locally compact Hausdorff unit space \(X \subseteq G\).
    \begin{enumerate}[label=(\roman*)]
        \item \label{def:grpd-alg:alg} \(\algalg{G}\) is the vector space of functions \(a \colon G \to \CCC\) that are sums \(a = a_1 + \dots + a_k\), where each \(a_i \colon G \to \CCC\) is a function that is compactly supported and continuous on an open bisection (but not necessarily continuous on $G$). We view $\algalg{G}$ as a \Star{}algebra with the usual operations of convolution and involution:
        \[(a*b)(g)\coloneqq\sum_{hk=g}a(h)b(k),\quad a^*(g)\coloneqq\overline{a(g^{-1})}.\]
        \item \label{def:grpd-alg:p} \(P \colon \algalg{G} \to \borelb(X)\) is the canonical \emph{weak conditional expectation}, that is, the map given by restriction to the open subspace \(X \subseteq G\), \emph{i.e.}\
            \[
                P \colon \algalg{G} \to \borelb(X), \;\; a \mapsto a|_X.
            \]
            Here $\borelb(X)$ denotes the \cstar algebra of all bounded Borel functions $X\to \C$.
        \item \label{def:grpd-alg:el} \(E_\L \colon \algalg{G} \to \borelb(X)/\meager(X)\) is the canonical \emph{essential conditional expectation}, that is, the map given by \(E_\L \coloneqq \pi_{\rm ess} \circ P\), where \(\pi_{\rm ess} \colon \borelb(X) \to \borelb(X)/\meager(X)\) is the usual quotient map, with $\meager(X)$ denoting the ideal of functions in $\borelb(X)$ vanishing off a meager set.
    \end{enumerate}
\end{definition}

\begin{remark}
    The nomenclature \(\algalg{G}\) in \cref{def:grpd-alg} \emph{could} be confusing. Indeed, the elements \(a \in \algalg{G}\) are \emph{not} necessarily continuous functions (at least in the non-Hausdorff case, cf.\ \cref{lemma:hausdorff}). We are, nevertheless, going to use this notation, since \(\algalg{G}\) is the only reasonable algebra one could define out of an \'etale groupoid that `separates' points in \(G\).
    Moreover, should \(G\) be Hausdorff then \(\algalg{G}\) in \cref{def:grpd-alg} really is the algebra of functions \(a \colon G \to \CCC\) that are globally continuous and compactly supported.
\end{remark}

In the text we will use the following nomenclature, which is not customary but very useful when talking about non-Hausdorff groupoids, as we will be doing.
\begin{notation} \label{notation:cont functions as sums}
    We will denote an arbitrary function \(a \in \algalg{G}\) as
    \[
        a = a_1 v_{s_1} + \dots + a_k v_{s_k},
    \]
    where \((s_i)_{i = 1}^k\) are open bisections of \(G\) and \(a_i\) is a continuous function supported and vanishing at infinity on \(s_is_i^*=\rg(s_i)\sbe X\), \emph{i.e.}\ \(a_i \in \contz(s_is_i^*) \subseteq \contz(X)\). The notation $a_iv_{s_i}$ then just means the composition $a_i\circ r$ when restricted to $s_i$, and is zero off $s_i$. 
    In particular,
    \[
        a\left(g\right) = \sum_{\substack{ i = 1, \dots, k \\ g \in s_i }} a_i\left(gg^{-1}\right)=\sum_{\substack{ i = 1, \dots, k \\ g \in s_i }} a_i\left(\rg(g)\right)
    \]
    for all \(g \in G\).
\end{notation}

The usefulness of \cref{notation:cont functions as sums} derives from the fact that we are considering both the functions \((a_i)_{i = 1}^k\) and the open bisections \((s_i)_{i = 1}^k\) at the same time. 
The following is well known, but also quite relevant in the context of the paper.
\begin{lemma} \label{lemma:hausdorff}
    Let \(G\) be an \'etale groupoid with locally compact Hausdorff unit space \(X \subseteq G\). Consider the following conditions.
    \begin{enumerate}[label=(\roman*)]
        \item \label{lemma:hausdorff:hausdorff} \(G\) is Hausdorff.
        \item \label{lemma:hausdorff:cont} Every \(a \in \algalg{G}\) is continuous.
        \item \label{lemma:hausdorff:p} \(P\) lands in \(\contz(X) \subseteq \borelb(X)\).
        \item \label{lemma:hausdorff:el} \(E_\L\) lands in \(\contz(X) \subseteq \borelb(X)/\meager(X)\).
    \end{enumerate}
    Then \cref{lemma:hausdorff:hausdorff} \(\Leftrightarrow\) \cref{lemma:hausdorff:cont} \(\Leftrightarrow\) \cref{lemma:hausdorff:p} \(\Rightarrow\) \cref{lemma:hausdorff:el}.
\end{lemma}

\begin{remark}
    Items \cref{lemma:hausdorff:hausdorff,lemma:hausdorff:el} are \emph{not} equivalent. For instance, the essential conditional expectation \(E_\L\) for the groupoid in \cref{ex:trivial-action} lands in \(\contz(X) \subseteq \borelb(X)/\meager(X)\).
\end{remark}

With the tools of \cref{def:grpd-alg} we may now define three different \cstar algebras associated to \(G\).
\begin{definition} \label{def:grpd-c-algs}
    Let \(G\) be an \'etale groupoid with Hausdorff unit space \(X \subseteq G\).
    \begin{enumerate}[label=(\roman*)]
        \item \(\fullalg{G}\) is the universal (or enveloping) \cstar algebra of \(\algalg{G}\), that is, $\fullalg{G}$ is the completion of $\algalg{G}$ with respect to the largest \cstar{}norm.
        \item \(\redalg{G}\) is a completion of \(\algalg{G}\) such that \(P \colon \algalg{G} \to \borelb(X)\) extends to a faithful positive map \(P \colon \redalg{G} \to \borelb(X)\) (that we also denote by \(P\) by abuse of notation).
        \item \(\essalg{G}\) is a (Hausdorff) completion of \(\algalg{G}\) such that \(E_\L \colon \algalg{G} \to \borelb(X)/\meager(X)\) induces a faithful positive map \(E_\L \colon \essalg{G} \to \borelb(X)/\meager(X)\) (that we also denote by \(E_\L\) by abuse of notation).
    \end{enumerate}
\end{definition}

Note that the approach to \(\redalg{G}\) and \(\essalg{G}\) taken in \cref{def:grpd-c-algs} does not define the algebras, as it is not clear whether they need to be necessarily unique.
Nevertheless, we claim that both $\redalg{G}$ and $\essalg{G}$ are the unique quotients of $\fullalg{G}$ with the properties stated in \cref{def:grpd-c-algs}. More precisely, if  $\rho\colon \fullalg{G} \to B$ is a quotient map and $Q \colon B \to C_0(X)''$ is a faithful weak expectation that factors the standard expectation $P$, meaning that $Q\circ\rho=P$, then $B\cong \redalg{G}$. A similar result holds for $\essalg{G}$.
In fact, both cases follow from the following.
\begin{proposition} \label{prop: generalized cond exp well defined algs}
    Suppose $A$ is a \cstar{}algebra containing a \cstar{}subalgebra $A_0\sbe A$ with a generalized conditional expectation $E\colon A\to \tilde A_0\supseteq A_0$, where $\tilde A_0$ is some \cstar{}algebra containing $A_0$ as a \cstar{}subalgebra. Then there is at most one quotient $A_{\rm red}$ of $A$ for which $E$ factors through a faithful generalized conditional expectation $E_{\rm red} \colon A_{\rm red} \to \tilde A_0\supseteq A_0$. More precisely, if such $A_{\rm red}$ exists and $\lambda\colon A\to A_{\rm red}$ denotes the quotient homomorphism, then for any other \cstar{}algebra $B$ that is a quotient of $A$ via a quotient homomorphism $\rho\colon A\to B$ and carries a faithful generalized expectation $Q\colon B\to \tilde A_0\supset A_0$ with $Q\circ\rho=E=E_{\rm red}\circ\lambda$, we must have $B\cong A_{\rm red}$ via an isomorphism that permutes the expectations.
\end{proposition}
\begin{proof}
    Since both $E_{\rm red}$ and $Q$ are faithful, and $Q\circ \rho=E_{\rm red} \circ \lambda$, we have
    $$\ker(\rho)=\{a\in A: \rho(a^*a)=0\}=\{a\in A: Q(\rho(a^*a))=0\}$$
    $$=\{a\in A: E_{\rm red}(\lambda(a^*a))=0\}=\ker(\lambda),$$
    so that 
    $$B\cong A/\ker(\rho)=A/\ker(\lambda)\cong A_{\rm red}.$$
\end{proof}

An immediate consequence of \cref{prop: generalized cond exp well defined algs} is that \cref{def:grpd-c-algs} does define \(\redalg{G}\) and \(\essalg{G}\).
In fact, as a corollary of the above approach we get the following commutative diagram.
\begin{corollary} \label{cor:diagram:basic}
    Let \(G\) be an \'etale groupoid with l.c.\ Hausdorff unit space \(X \subseteq G\). The identity map \(\id \colon \algalg{G} \to \algalg{G}\) induces quotient maps making
    \begin{center}
        \begin{tikzcd}[scale=50em]
            \fullalg{G} \arrow[two heads]{dr}{\Lambda^{\rm ess}} \arrow[two heads]{r}{\Lambda} & \redalg{G} \arrow[two heads]{d}{\Lambda^{\rm red}_{E_\L, P}} \arrow{r}{P} & \borelb\left(X\right)\arrow[two heads]{d}{\pi_{\text{ess}}}  \\
            & \essalg{G} \arrow{r}{E_\L} & \borelb\left(X\right)/\meager\left(X\right)
        \end{tikzcd}
    \end{center}
    commute. Moreover, \(P\) and \(E_\L\) are faithful.
\end{corollary}

One of the main contributions of the text is an enrichment of the above diagram (cf.\ \cref{cor:diagram}). We end the preliminary section recording the following standard result.
\begin{lemma} \label{lemma:norms red infty max}
    Let \(a \in \algalg{G}\), where \(G\) is some \'etale groupoid. Then, viewing $a$ as an element of the \cstar{}algebras $\essalg G$, $\redalg{G}$ and $\fullalg G$, we have
    \[
        \norm{a}_{\rm ess} \leq \norm{a}_{\rm red} \leq \norm{a}_{\rm max},
    \]
    and all of the above are equal to $\norm{a}_{\infty}\coloneqq\sup_{g\in G}|a(g)|$ when \(a\) is supported on a single bisection.
\end{lemma}

%%%%%%%%%%%%%%%%%%%%%%%%%%%%%%%%%%%%%%%%%%%%%%%%%%%%%%%%%%%%%%%%%%%%%%%
% ESSENTIAL REPRESENTATIONS AND ESSENTIAL ALGEBRAS
%%%%%%%%%%%%%%%%%%%%%%%%%%%%%%%%%%%%%%%%%%%%%%%%%%%%%%%%%%%%%%%%%%%%%%%
\section{Essential representations and essential algebras} \label{sec:ess-reps-algs}

In this section we produce yet another groupoid \cstar algebra (the \emph{maximal} version of \(\essalg{G}\), which we denote by \(\essmaxalg{G}\), cf.\ \cref{def:max ess alg}) and yet another approach to \(\essalg{G}\) (cf.\ \cref{prop:rep red and ess calgs:ess}). In order to do this, however, we need a certain amount of preparatory work.
\begin{definition} \label{def:dangerous arrow}
    Let \(G\) be an \'etale groupoid.
    \begin{enumerate}[label=(\roman*)]
        \item An arrow \(g \in G\) is \emph{dangerous} if there are \((g_i)_{i} \subseteq G\) and \(h \in G \setminus \{g\}\) such that \(g_i \to g, h\).
        \item We let \(D \subseteq G\) be the set of dangerous arrows.
    \end{enumerate}
\end{definition}

\begin{notation}
    From now on, \(D \subseteq G\) will always denote the set of dangerous arrows in the sense of \cref{def:dangerous arrow}, even if not explicitly mentioned.
\end{notation}

Roughly speaking, an arrow \(g \in G\) is dangerous whenever it witnesses the ``non-Hausdorffness'' of \(G\). More precisely, $g\in D$ if and only if there exists $h\not=g$ in $G$ that cannot be topologically separated from $g$ in the sense that for all neighborhoods $U$ and $V$ of $g$ and $h$, respectively, we have $U\cap V\not=\emptyset$. This speaks to, in fact, the nomenclature used in \cite{bkmk-2024-twited-grpds}, for instance, where an arrow is \emph{Hausdorff} if, in our notation, it is not dangerous.

\begin{lemma} \label{lemma:convergence in a bisection}
    Suppose that \(x_i \to \s(g)\) for some \((x_i)_i \subseteq X\) and \(g \in G\). Then, by passing to a subnet if necessary, there are \((g_i)_i \subseteq G\) such that \(\s(g_i) = x_i\) and \(g_i \to g\).
\end{lemma}
\begin{proof}
    As \(G\) is \'etale, let \(\mathcal{U}\) be a basis of open neighborhoods of \(g\) formed by bisections.
    For any \(u \in \mathcal{U}\) it follows that \(x_i \in u^*u=\s(u)\) for all \(i \geq i(u)\) (for some large enough index \(i(u)\)).
    In particular, since \(u\) is a bisection, there is some arrow \(g_{i, u} \in u\) such that \(\s(g_{i,u}) = x_i\). It then follows that the subnet \((g_{i(u),u})_{u} \subseteq G\) converges to \(g\).
\end{proof}

\begin{lemma} \label{lemma:dangerous arrows:convergence}
Let \((g_i)_{i} \subseteq G \ni g, h\) be such that \(g_i \to g, h\). Then \(\s(g_i) \to \s(g) = \s(h)\) and \(\rg(g_i) \to \rg(g) = \rg(h)\). In particular, \(g^{-1}h \in xGx\), where \(x \coloneqq \s(h)\), and \(\s(g_i) \to x, g^{-1}h\).
\end{lemma}

\begin{proof}
The first statements follow from the fact that $\s$ and $\rg$ are continuous maps. Since $X=G^{(0)}$ is Hausdorff, if \(g_i \to g, h\) then \(\s(g) = \s(h)\) and $\rg(g)=\rg(h)$. In particular, \(g^{-1}h \in xGx\), where \(x \coloneqq h^{-1}h\). Since the multiplication on $G$ is continuous, 
we get \(\s(g_i)=g_i^{-1} g_i \to g^{-1}h\), as desired.
\end{proof}

The following description of the dangerous arrows \(D \subseteq G\) will be useful. The following should be compared to \cite{bkmk-2024-twited-grpds}*{Lemma 4.3}, which was found independently.
\begin{lemma} \label{lemma:dangerous arrows}
    Let \(G\) be an \'etale groupoid with Hausdorff unit space \(X \subseteq G\), and let \(D \subseteq G\) be the set of dangerous arrows. The following assertions hold.
    \begin{enumerate}[label=(\roman*)] 
        \item \label{lemma:dangerous arrows:hausdorff} \(D\) is non empty if and only if \(X\) is not closed (in \(G\)) if and only if \(G\) is not Hausdorff. In such case, \(G\) is not principal.
        \item \label{lemma:dangerous arrows:subgroupoid} \(D\) is an \emph{ideal subgroupoid} of \(G\) in the sense that \(gh, kg \in D\)  whenever \(g \in D\) and \(h, k \in G\) are composable with \(g\).
    \end{enumerate}
    In particular, \(g \in D\) if and only if \(\s(g) \in D\), if and only if \(\rg(g) \in D\). Moreover, \(G\) decomposes into the sub-groupoids \(G = D \sqcup (G \setminus D)\) (although this decomposition may not be into closed sets).
\end{lemma}
\begin{proof}
    Statement \cref{lemma:dangerous arrows:hausdorff} trivially follows from \cref{lemma:dangerous arrows:convergence}. Indeed, if \(g_i \to g, h\) and \(g \neq h\) then, by \cref{lemma:dangerous arrows:convergence}, \(g^{-1} h\) is some isotropy arrow, and it is non trivial since \(g \neq h\). In this case, \(G\) is not principal. Likewise, if \(X\) is not closed then, by definition, there is some \((x_i)_i \subseteq X\) converging to \(g \in G \setminus X\), but, again by \cref{lemma:dangerous arrows:convergence}, it follows that \(x_i \to \s(g), g, \rg(g)\), \emph{i.e.}\ \(D\) is non empty.

    We now turn our attention to \cref{lemma:dangerous arrows:subgroupoid}.
    We first show that \(hg \in D\) whenever \(g \in D\) and \((h, g) \in G^{(2)}\).
    Suppose that \(g_i \to g, \ell\), where \(g \neq \ell\). By \cref{lemma:dangerous arrows:convergence} it follows that \(\rg(g_i) \to \rg(g) = \rg(h) = \rg(\ell)\) (and similarly for the domains).
    By passing to a subnet if necessary and using \cref{lemma:convergence in a bisection}, we may assume, without loss of generality, that there is some \((h_i)_i \subseteq G\) such that \(h_i \to h\) and \(\s(h_i) = \s(g_i)\).
    It then follows from the fact that multiplication in \(G\) is continuous
    that \(h_ig_i \to hg, h\ell\). Since \(g \neq \ell\), this proves that \(hg \in D\). The proof that \(gk \in D\) whenever \(g \in D\) and \((g,k) \in G^{(2)}\) is similar.
    In order to finish the proof of \cref{lemma:dangerous arrows:subgroupoid} we just need to show that \(g^{-1} \in D\) when \(g \in D\), but this follows from the fact that the map \(g \mapsto g^{-1}\) is continuous. Lastly, the ``in particular'' statement follows from \(D\) being a subgroupoid.
\end{proof}

\begin{remark} \label{rem:d not loc comp}
    It may seem that \cref{lemma:dangerous arrows} describes \(D \subseteq G\) rather nicely, but this is far from the truth.
    \begin{enumerate}[label=(\roman*)]
        \item Algebraically the set \(D \subseteq G\) is indeed well-behaved, since it is an ideal subgroupoid, that is, it ``absorbs'' arrows (cf.\ \cref{lemma:dangerous arrows}~\cref{lemma:dangerous arrows:subgroupoid}).
        \item Topologically, if \(G\) is minimal,\footnote{\, A groupoid is \emph{minimal} when the orbit of every unit \(x \in X \subseteq G\) is dense, \textit{i.e.}\ \(\{gxg^{-1} \mid g \in Gx\}\) is dense in \(X\).} then the set \(D^{(0)}=X \cap D\) -- the unit space of $D$ as a groupoid -- is dense in \(X\) whenever $G$ is not Hausdorff. Unless $D^{(0)}=X$, \(D\) is not closed in \(G\), and it may even not be locally compact.
        \item \label{rem:d not loc comp:3} It is also worth noting that \(xGx \cap \overline{X}\) for $x\in X$ may \emph{not} be a group, as one could have hoped for. In particular, it is \emph{not} true that \(xGx \cap \overline{X} = xGx \cap D\), as the latter is a group. Indeed, let \(S\) be the inverse semigroup generated (as an inverse semigroup) by a copy of \(\Z^2\), together with \(3\) central idempotents \(\{0, e, f\}\), subject to \(0 \leq e \leq (1, 0)\), \(0 \leq f \leq (0, 1)\) and \(ef = 0\), where \(\Z^2 = \langle (1,0), (0,1) \rangle\). Let \(S\) act trivially on \([0, 2]\), where the domain of \(0\) is empty; the domain of \(e\) is \([0, 1)\) and the domain of \(f\) is \((1, 2]\). In such case, the unit \(x \coloneqq 1 \in [0, 2] \subseteq [0, 2] \rtimes S \eqqcolon G\) is dangerous, and \(xGx \cap \overline{X}\) contains the (non-unit) arrows \([(1,0), 1]\) and \([(0, 1), 1]\). Nevertheless, it does \emph{not} contain \([(1, 1), 1]\).
    \end{enumerate}
\end{remark}

The following nomenclature may not be customary, since some people denote by \(\supp(a)\) the closure of the set we are interested in here. What we denote by $\supp(a)$ is sometimes called the \emph{open support} of $a$. Notice, however, that $\supp(a)$ is also not open, in general, as $a$ need not be continuous.
\begin{notation}
    Given a function \(a \in \algalg{G}\), we denote by \(\supp(a)\) the set of arrows \(g \in G\) such that \(a(g) \neq 0\).
\end{notation}

\begin{lemma} \label{lemma:supp is compact}
    Each function \(a \in \algalg G\) is compactly supported, meaning that \(\supp(a)\) is contained in some compact set \(K \subseteq G\).
\end{lemma}
\begin{proof}
   Write $a=a_1+\ldots+a_n$ with $a_i\in \contc(s_i)$ for some open bisections $s_i\sbe G$. In particular $\supp(a_i)\sbe K_i$ for some compact subset $K_i\sbe s_i\sbe G$. Then $\supp(a)$ is contained in the compact subset $K \coloneqq K_1\cup\cdots\cup K_n$. 
\end{proof}

For the purposes of this text, the following result is crucial. Recall that a \emph{continuous} function \(a \colon X \to \CCC\) on some topological space \(X\) that happens to be \(0\) on a dense subset must always be \(0\). The following result states the analogous result in the non-commutative non-Hausdorff setting and gives some interesting characterizations of the dangerous arrows $D\sbe G$.

\begin{proposition}\label{pro:dangerous-points-discontinuity}
Let $G$ be a locally compact étale groupoid with Hausdorff unit space $X$. Let $D\sbe G$ be the set of dangerous arrows. Then the following are equivalent for an arrow $g\in G$.
\begin{enumerate}[label=(\roman*)]
    \item \label{pro:dangerous-points-discontinuity:1} $g$ is a dangerous arrow, \emph{i.e.}\ $g\in D$;
    \item \label{pro:dangerous-points-discontinuity:2} there exists a compact (neighborhood bisection) $K\sbe G$ with $g\in \overline K\backslash K$;
    \item \label{pro:dangerous-points-discontinuity:3} there exists $a\in \algalg G$ and a net $(g_n)_n \sbe G$ with \(g_n \to g\) and  
        $$a(g_n)=0\mbox{ for all }n\mbox{ and }a(g)\not=0;$$
    \item \label{pro:dangerous-points-discontinuity:4} there exists $a\in \algalg G$ that is discontinuous at $g$.
\end{enumerate} 
In other words,
    $$D=\bigcup_{a\in \algalg G} D(a),$$
    where $D(a)$ denotes the set of discontinuity points of $a$. Moreover, given $g\in D$, we may always choose $a\in \algalg G$ that is compactly supported and continuous on an open bisection of $G$ and such that $g\in D(a)$.  
\end{proposition}
\begin{proof}
    \cref{pro:dangerous-points-discontinuity:1} $\Rightarrow$ \cref{pro:dangerous-points-discontinuity:2}. If $g\in D$, take a net $(g_n)_n \sbe G$ with $g_n\to g,h$ for some $g\not=h\in G$. Since $G$ is locally compact and \'etale, there exists a compact neighborhood bisection $K$ containing $h$. Since $\s(g)=\s(h)$ and $K$ is a bisection, we must have $g\notin K$. As $g_n\to h$, we must eventually have $(g_n)_n \sbe K$. It follows that $g\in \overline K \backslash K$.

    \cref{pro:dangerous-points-discontinuity:2} $\Rightarrow$ \cref{pro:dangerous-points-discontinuity:1}.
    If $g\in \overline K\backslash K$, then there is a net $(g_n)_n \sbe K$ with $g_n\to g$.
    Since $K$ is compact, after passing to a subnet, we may assume that this net converges to some $h\in K$, and since $g\notin K$, we have $g\not=h$, so that $g\in D$.

    \cref{pro:dangerous-points-discontinuity:1} $\Rightarrow$ \cref{pro:dangerous-points-discontinuity:3}. Suppose $g\in D$, and take a net $(g_n)_n \sbe G$ and $h\not= g$ with $g_n\to g,h$. In particular, we have $\s(g)=\s(h)$ and $\rg(g)=\rg(h)$ (cf.\ \cref{lemma:dangerous arrows:convergence}). Let $s,t \subseteq G$ be open bisections with $g\in s$ and $h\in t$. Then $g\notin t$ and $h\notin s$ (since both \(s\) and \(t\) are bisections). After passing to a subnet if necessary, we may assume that $(g_n)_n \sbe s\cap t$. Pick $b\in \contc(X)$ with $\supp(b)\sbe s^*s=\s(s)$ and $b(\s(g))=1$, and consider the function $a \coloneqq bv_s -bv_t$ (recall \cref{notation:cont functions as sums}). Then we have $a(g)=b(\s(g))=1$ as $g\notin t$, and $a(g_n)=b(\s(g_n))-b(\s(g_n))=0$ for all $n$, since $(g_n)_n \sbe s\cap t$.

    \cref{pro:dangerous-points-discontinuity:3} $\Rightarrow$ \cref{pro:dangerous-points-discontinuity:4}. This is obvious as the net in \cref{pro:dangerous-points-discontinuity:3} witnesses the discontinuity of $a$ at the point $g$.

    \cref{pro:dangerous-points-discontinuity:4} $\Rightarrow$ \cref{pro:dangerous-points-discontinuity:1}. Suppose $a\in \algalg G$ and $g\in D(a)$. Then there exists a net $(g_n)_n\sbe G$ with $g_n\to g$ but no subnet of $(a(g_n))_n$ converges to $a(g)$. Write $a=a_1+a_2+\ldots+a_k$, with $a_i$ continuous and supported on an open bisection $s_i$, \emph{i.e.}\ $a_i\in \contc(s_i)$. Since we have a finite number of indices $i\in \{1,\ldots,k\}$, after passing to a subnet if necessary, we may assume that there exists some fixed $i\in \{1,\ldots,k\}$ such that no subnet of $(a_i(g_n))_n$ converges to $a_i(g)$. Moreover, after passing to a subnet again, we may assume that $(g_n)_n \sbe \supp(a_i)$. Indeed, otherwise there is a subnet of $(g_n)_n$ with $a_i(g_n)=0$ for all $n$. In this case, we must have $a_i(g)\not=0$, as otherwise we would have a subnet of $(g_n)_n$ with $a(g_n)=0\to 0=a_i(g)$. Therefore we may assume $(g_n)_n\sbe \supp(a_i)\sbe s_i$. As $a_i$ is continuous on $s_i$, we have $g\notin s_i$. Now, the closure of $\supp(a_i)$ relative to $s_i$ is a compact subset $K\sbe s_i$, so there exists some subnet $(h_m)_m$ of $(g_n)_n$ converging to some $h\in K$, and since $a_i$ is continuous on $s_i\supseteq K$, we have $a_i(h_m)\to a_i(h)$. But $a_i(h_m)\not\to a_i(g)$, so that $g\not= h$ and therefore $g\in D$. At the same time, $g\in D(a_i)$ with $a_i$ continuous and compactly supported on the open bisection $s_i$. This gives the final assertion in the statement.
\end{proof}    

\begin{corollary}\label{cor:restriction-map}
    If $a\in \algalg G$, then $a$ is continuous on every arrow $g\in G\backslash D$, that is, we have a well-defined linear map $\algalg G\to \contb(G\backslash D)$, $a\mapsto a|_{G\backslash D}$.
\end{corollary}

\begin{corollary} \label{lemma:support is in d}
    Let \(G\) be an \'etale groupoid with locally compact Hausdorff unit space \(X \subseteq G\).
    If $a\in \algalg G$ vanishes in a dense set, then $\supp(a)\sbe D$. More generally, let \(K \subseteq G\) be given, and let \(C \subseteq K\) be dense (in \(K\)) and \(a \in \algalg{G}\) be such that \(a(c) = 0\) for all \(c \in C\). Then \(\supp(a|_K) \subseteq D \cap K\).
\end{corollary}
\begin{proof}
Take $g\in \supp(a|_K)$ and a net $(g_n)\sbe C$ with $g_n\to g$ and $a(g_n)=0$.
By \cref{pro:dangerous-points-discontinuity} we must have $g\in D$. Therefore $\supp(a|_K)\sbe D\cap K$.
\end{proof}

\begin{remark} \label{rem:support is in d is sharp}
    \cref{lemma:support is in d} is sharp, in the sense that \(\supp(a|_K)\) may actually coincide with \(D \cap K\). Indeed, this already happens with very simple groupoids and very large sets \(K\). For instance, let \(G\) be the groupoid of germs of the trivial action of the inverse semigroup \(S \coloneqq \Z_2 \sqcup \{0\}\) on \([0,1]\), where the domain of \(0 \in S\) is defined to be \([0, 1)\). This groupoid can be described by
    \[
        G = \left[0, 1\right) \sqcup \left\{1\right\} \times \Z_2,
    \]
    with the obvious groupoid structure and the unique topology that restricts to the usual topology on \([0,1]\) and such that \(x_k \to (1, \gamma) \in G\) for all \(\gamma \in \Z_2\) and net \((x_k)_{k} \subseteq [0,1)\) converging to \(1\) (in the usual sense). With this structure, $G$ is an étale compact groupoid, in particular we can take $K\coloneqq G$.
    Letting \(\Z_2 = \{e, \gamma\}\), where \(e = \gamma^2\), the function
    \[
        a \coloneqq 1_{[0, 1]} v_{e} - 1_{[0, 1]} v_{\gamma}
    \]
    clearly is contained in \(\algalg{G}\). Moreover, \(a(x) = 1 - 1 = 0\) for all \(x \in [0,1) \subseteq G\). Lastly, \(a([1, e]) = 1\), while \(a([1, \gamma]) = -1\). Therefore $\supp(a)=D=D\cap K$.
\end{remark}

Some immediate consequences of \cref{lemma:support is in d} are the following.
\begin{corollary} \label{cor:meager supp is in d}
    Let \(G\) be an \'etale groupoid with locally compact Hausdorff unit space \(X \subseteq G\). The following assertions hold.
    \begin{enumerate}[label=(\roman*)]
        \item If \(a \in \algalg{G}\) is such that \(\supp(a|_X)\) is meager in $X$, \emph{i.e.}\ \(E_\L(a) = 0\), then \(\supp(a|_X) \subseteq D\cap X\).
        \item If \(a \in \algalg{G}\) is such that \(\supp(a)\) is meager, then \(\supp(a) \subseteq D\).
    \end{enumerate}
\end{corollary}
\begin{proof}
    The proof of both assertions follow the same ideas. As \(G\) and $X$ are locally compact and regular, the Baire Category Theorem implies that any co-meager set is dense. The claims then follow from an application of \cref{lemma:support is in d}.
\end{proof}

The following lemma starts the trend of considering groupoids that are ``small'' enough, \emph{i.e.}\ only those that can be covered by countably many open bisections.
\begin{lemma}\label{lem:D-meager}
    Let \(G\) be an \'etale groupoid with locally compact Hausdorff unit space. If \(G\) can be covered by countably many open bisections, then the set $D\sbe G$ of dangerous arrows is a meager set, that is, a countable union of nowhere dense sets.
\end{lemma}
\begin{proof}
    Suppose that \(\{s_n\}_{n \in \N}\) are countably many open bisections that cover \(G\). Using \cref{pro:dangerous-points-discontinuity}, we get
    \[
        D \subseteq \bigcup_{n = 1}^\infty \overline{s_n} \setminus s_n
    \]
    which is, by construction, a countable union of closed nowhere dense sets, \emph{i.e.}\ a meager set. Subsets of meager sets are also meager, and thus so is \(D\).
\end{proof}

Observe that the left hand side in the following is the ideal \(\NN_{E_\L} \cap \algalg{G}\), by which we will quotient \(\algalg{G}\) to create \(\essalgalg{G} \subseteq \essalg{G}\) (cf.\ \cref{def:ess-algalg}).
\begin{corollary} [cf.\ \cite{bkmk-2024-twited-grpds}*{Proposition 4.5}] \label{cor:nucleus is in d}
    Let \(G\) be an \'etale groupoid with locally compact Hausdorff unit space.
    If the set $D$ of dangerous arrows is meager, then
    \begin{align*}
        \NN_{E_\L} \cap \algalg{G}&=\left\{a \in \algalg{G} \mid E_\L\left(a^*a\right) = 0\right\} \\ &= \left\{ a \in \algalg{G} \mid \supp\left(a^*a|_X\right) \subseteq D\cap X \right\}
        \\&=\left\{ a \in \algalg{G} \mid \supp\left(a\right) \subseteq D \right\}.
    \end{align*}
    In particular, the above holds when \(G\) can be covered by countably many open bisections, and this, in turn, is the case if $G$ is \(\sigma\)-compact or second countable. Moreover, if $G=\cup_{n=1}^\infty s_n$ is a cover of $G$ by countably many open bisection $s_n\sbe G$, then
    $$\NN_{E_\L} \cap \algalg{G}=\left\{ a \in \algalg{G} \mid \supp\left(a\right) \subseteq \tilde D \right\},$$
    where $\tilde D \coloneqq \cup_{n=1}^\infty \overline s_n\backslash s_n$.
\end{corollary}
\begin{proof}
    The inclusion \(\subseteq\) follows from \cref{cor:meager supp is in d}.
    For the reverse inclusion, we use \cref{lem:D-meager}, which implies that $D$, and hence also $D\cap X$ is meager.
    Hence, if \(\supp(a^*a|_X) \subseteq D\cap X\) then \(E_\L(a^*a) = 0\), as desired.

    Notice that for $x\in X \coloneqq G^{(0)}$, we have
    $$(a^*a)(x)=\sum_{g\in Gx} |a(g)|^2.$$
    Since $g\in D$ if and only if $\s(g)\in D\cap X$ (cf.\ \cref{lemma:dangerous arrows}) it follows that $\supp(a^*a|_X)\sbe D\cap X$ if and only if $\supp(a)\sbe D$.

    The ``in particular'' statement holds by \cref{lem:D-meager} and because any \(\sigma\)-compact groupoid can be covered by countably many open bisections. The same is true for a second countable \'etale groupoid. 

    Finally, to prove the last assertion of the statement, recall from the proof of \cref{lem:D-meager} that $D\sbe \tilde D$. Hence, if $a\in \algalg G$ and $\supp(a)\sbe D$, then also $\supp(a)\sbe \tilde D$. On the other hand, if $\supp(a)\sbe \tilde D$, then $\supp(a)$ is meager as $\tilde D$ is meager, and therefore $\supp(a)\sbe D$ by \cref{cor:meager supp is in d}.
\end{proof}

\begin{remark} \label{rem:covered by cnt bisec}
    The condition that \(G\) be covered by countably many open bisections in \cref{cor:nucleus is in d} is both quite mild and necessary.
    \begin{itemize}
        \item It is quite mild since any second countable groupoid satisfies it, \emph{i.e.}\ any quotient of any transformation groupoid will satisfy it as long as the acting group is countable. Moreover, some non-second countable groupoids such as \(\beta \Gamma \rtimes \Gamma\) (where \(\Gamma\) is any countable discrete group) are covered by countably many bisections (namely \((\beta \Gamma \times \{\gamma\})_{\gamma \in \Gamma}\)). Lastly, any locally compact Hausdorff space \(X \eqqcolon G\) that is \emph{not} \(\sigma\)-compact also satisfies it, as it is covered by \(1\) bisection, namely \(X\) itself. More generally, a groupoid $G$ admits a countable cover by bisections if and only if it can be written as a groupoid of germs $G\cong X\rtimes S$, for a quasi-countable inverse semigroup (cf.\ \cite{chung-martinez-szakacs-2022}) $S$ acting on $X$. This covers, for instance, all groupoids associated to uniform Roe algebras of uniformly locally finite metric spaces.
        \item It is also necessary since otherwise \(D \subseteq G\) may well not be meager. In fact, if \(G\) can only be covered by uncountably many bisections then it could happen that \(D = G\) (see \cite{KwasniewskiMeyer-essential-cross-2021}*{Example 7.16}).
    \end{itemize}
\end{remark}

\subsection{Essential representations and the maximal essential algebra} \label{subsec:ess-reps}
The following is one of the central notions we investigate in the paper.
\begin{definition} \label{def:ess-rep}
Let \(G\) be an \'etale groupoid with Hausdorff unit space. We say a representation \(\pi \colon \fullalg{G} \to \B(\H)\) is \emph{essential} if \(\pi(a) = 0\) for every \(a \in \algalg{G} \subseteq \textbf{}\fullalg{G}\) such that \(E_\L(a^*a) = 0\).
\end{definition}

We refer the reader to \cref{sec:class of ess reps} for a description of a (rather natural) class of essential representations (in the sense of \cref{def:ess-rep}). The following is a semantic, and yet relevant, remark.
\begin{remark}
    In the literature, the naming \emph{essential} representation usually means a representation  \(\pi \colon A \to \B(\H)\) of a \cstar algebra $A$ whose image does not contain any compact operator -- see \cite{Brown-Ozawa:Approximations}*{Definition~1.7.4}. These are not to be confused with the notion in \cref{def:ess-rep}. The naming \emph{essential} in \cref{def:ess-rep} comes from the fact that these have to do with the essential \cstar algebra \(\essalg{G}\).
\end{remark}

\begin{remark}\label{rem:essential-rep-cntble-bis}
    In the case when \(G\) is covered by countably many open bisections $s_1, s_2, s_3,\ldots$ (which is quite often, cf.\ \cref{rem:covered by cnt bisec}), it follows from \cref{cor:nucleus is in d} that a representation \(\pi\) is essential in the sense of \cref{def:ess-rep} if and only if \(\pi(a) = 0\) for all \(a \in \algalg{G}\) with \(\supp(a^*a) \subseteq D\cap X\), or equivalently, $\supp(a)\sbe D$, or also $\supp(a)\sbe \tilde D \coloneqq \cup_n \overline s_n\backslash s_n$.
\end{remark}

With the help of \cref{def:ess-rep} we may now define the following.
\begin{definition} \label{def:max ess alg}
    Let \(G\) be an \'etale groupoid with Hausdorff unit space \(X \subseteq G\). The \emph{maximal essential \cstar algebra of \(G\)} is the universal \cstar algebra of \(\algalg{G}\) with respect to essential representations, that is, it is the (necessarily unique) \cstar algebra \(\essmaxalg{G}\) densely containing a quotient of \(\algalg{G}\) and such that any essential representation \(\pi \colon \algalg{G} \to \B(\H)\) induces a continuous map \(\pi \colon \essmaxalg{G} \to \B(\H)\). More precisely, $\essmaxalg G$ is the Hausdorff completion of $\algalg G$ with respect to the \cstar{}seminorm
    $$\|a\|_{\ess,\max}:=\sup\{\|\pi(a)\|: \pi\mbox{ is an essential representation of }\algalg G\}.$$
\end{definition}

The following basic proposition just shows that \(\essmaxalg{G}\) coincides with the usual \(\fullalg{G}\) whenever \(G\) happens to be Hausdorff (cf.\ \cref{lemma:hausdorff}).
\begin{proposition}
    If \(G\) is Hausdorff then every representation \(\pi \colon \algalg{G} \to \B(\H)\) is essential (in the sense of \cref{def:ess-rep}), and therefore the identity map \(\id \colon \algalg{G} \to \algalg{G}\) extends to an isomorphism \(\fullalg{G} \cong \essmaxalg{G}\).
\end{proposition}
\begin{proof}
    This follows from \cref{cor:meager supp is in d} or \cref{cor:nucleus is in d}. Indeed, if \(G\) is Hausdorff then \(D\) is empty and therefore the only element \(a \in \algalg{G}\) such that \(\supp(a^*a) \subseteq D\) is \(a = 0\), which trivially gets mapped to \(0\) under any representation.
\end{proof}

As usual, it is (a priori) not clear whether \(\essmaxalg{G}\) is non-trivial. Nevertheless, this will follow since it projects onto \(\essalg{G}\) (cf.\ \cref{cor:diagram}).
Furthermore, we briefly observe in passing that, unless \(G\) is Hausdorff (see \cref{lemma:hausdorff}), \(\essmaxalg{G}\) needs only densely contain a proper \emph{quotient} of \(\algalg{G}\).
Indeed, let \(\Lambda^{\ess} \colon \fullalg{G} \to \essalg{G}\) be the usual quotient map (cf.\ \cref{cor:diagram:basic}). Then \(\essmaxalg{G}\) will densely contain \(\Lambda^{\ess}(\algalg{G})\). For this reason we give the following.
\begin{definition} \label{def:ess-algalg}
    We let \(\essalgalg{G} \coloneqq \Lambda^\ess(\algalg{G})\).
\end{definition}

Using \cref{cor:meager supp is in d} we may give several useful characterizations of the algebra \(\essalgalg{G}\).
\begin{lemma} \label{lemma:ess-alg-alg}
    Let \(G\) be an \'etale groupoid with Hausdorff unit space. If the set \(D\sbe G\) of dangerous arrows is meager (e.g.\ if $G$ is covered by countably many bisections), then
    \[
      \essalgalg{G} \cong \algalg{G}/\NN_{E_\L}^c \cong \algalg{G}/J_{\sing}^c,
    \]
    where $\NN_{E_\L}^c\coloneqq\NN_{E_\L} \cap \algalg{G}$ and \(J_{\sing}^c \subseteq \algalg{G}\) is the ideal of functions such that \(a^*a\) is supported on \(D\); and \(\NN_{E_\L}\) is the ideal of elements \(a \in \fullalg{G}\) such that \(E_\L(a^*a) = 0\).
\end{lemma}

It is clear that $\essmaxalg G$ is a quotient of $\fullalg G$. The following result provides a more explicit description of this quotient.

\begin{proposition}\label{prop:description-ess-max-algebra-quotient}
    We have $\essmaxalg G = \fullalg G / \NN_{E_\L}^\max$, where $\NN_{E_\L}^\max$ denotes the closure of $\NN_{E_\L}^c$ in $\fullalg G$, that is, the closure with respect to the maximal \cstar{}norm.
\end{proposition}

\begin{proof}
    We show that $\fullalg G / \NN_{E_\L}^\max$ satisfies the desired universal property for essential representations of $\algalg G$. 
    
    Let $\pi \colon \algalg G \to \B(\Hilb)$ be an essential representation. By definition, we have $\pi(\NN_{E_\L}) = 0$. Let us also denote by $\pi$ the (unique) extension to a representation $\pi \colon \fullalg G \to \B(\Hilb)$. 
    
    Since $\pi$ is contractive, it is continuous with respect to the maximal \cstar{}norm. Therefore, it vanishes on the closure $\NN_{E_\L}^\max$, that is, $\pi(\NN_{E_\L}^\max) = 0$. It follows that $\pi$ factors through the quotient $\fullalg G / \NN_{E_\L}^\max$, which establishes the claim.
\end{proof}

For the results that follow, we need to recall the definition of the regular representation.

\begin{definition} \label{def:lambda x}
    Let \(G\) be an \'etale groupoid with Hausdorff unit space \(X \subseteq G\). Given \(x \in X \subseteq G\), we let
    \[
        \lambda_x \colon \algalg{G} \to \B\left(\ell^2\left(Gx\right)\right), \;\; \lambda_x\left(a\right)\left(\delta_h\right) \coloneqq \sum_{\s(g)=\rg(h)} a\left(g\right) \delta_{gh}.
    \]
\end{definition}

Straightforward computations shows that each \(\lambda_x\) in \cref{def:lambda x} is a representation of $\algalg G$, so that it extends to a representation $\lambda_x\colon \fullalg G\to \B(\ell^2(Gx))$ (still denoted the same way by abuse of notation).
The following result is well-known, hence we omit the details. Indeed, usually the representation $\Lambda$ obtained as the direct sum of the $\lambda_x$ above for $x\in X$ is used to \emph{define} $\redalg{G}$.

\begin{proposition}\label{prop:rep red and ess calgs:red}
 The representation
        \[
            \Lambda\coloneqq\bigoplus_{x \in X} \lambda_x \colon \fullalg{G} \to \B\left(\bigoplus_{x \in X} \ell^2\left(Gx\right)\right)
        \]
        induces a faithful representation of \(\redalg{G} \hookrightarrow \B(\oplus_{x \in X} \ell^2(Gx))\).    
\end{proposition}

Next we recall the well-known $j$-map (see \cite{Renault80}*{II.4.2}) that allows us to view elements of $\redalg G$ as Borel bounded functions on $G$.
    
\begin{lemma} \label{rem:j-map-red}
Given an étale groupoid $G$ with locally compact Hausdorff unit space $X$, we define the \emph{$j$-map} as 
    $$j\colon \redalg{G}\to \ell^\infty(G),\quad j(a)(g)\coloneqq\braket{\delta_{g}}{\lambda_{\s(g)}(a)\delta_{\s(g)}},$$
    where we use the convention that inner products are linear in the second variable. Then $j$ is an injective linear contraction, that is, $\|j(a)\|_\infty\leq \|a\|_{\rm red}$, and it extends the inclusion map $\algalg G\into \redalg G$. The image of $j$ consists of bounded Borel measurable functions on $G$ that are continuous on $G\backslash D$, so that $j$ may be also viewed as a contractive linear embedding from $\redalg G$ into $\borelb(G)$, and also into $\contb(G\backslash D)\oplus\ell^\infty(D)\sbe \ell^\infty(G)$. 
\end{lemma}
\begin{proof}
From its formula, it is clear that $j$ is a linear contraction. The fact that $j$ is injective is a standard fact: if $j(a)=0$, then $\braket{\delta_g}{\lambda_{\s(g)}(a)\delta_{\s(g)}}=0$ for all $g\in G$. Using the unitaries $U_h\colon \ell^2(G\s(h))\to \ell^2(G\rg(h))$ given by $U_h(\delta_g)\coloneqq\delta_{gh^{-1}}$ and the relation $U_g\lambda_{\s(g)}(a)=\lambda_{\rg(g)}(a)U_g$, one gets from 
$$\braket{U_h(\delta_g)}{U_h\lambda_{\s(g)}(a)\delta_{\s(g)}}=\braket{\delta_g}{\lambda_{\s(g)}(a)\delta_{\s(g)}}=0$$
that $\lambda_x(a)=0$ for all $x\in X$, so that $a=0$.
        
        The image of $j$ is in the closure of $\algalg G$ within $\ell^\infty(G)$ with respect to the sup norm $\|\cdot\|_\infty$. In particular all functions in the image of $j$ are Borel bounded functions.
        To see that $j(a)$ is continuous on $G\backslash D$, write $a=\lim a_n$ for a net $(a_n)_n \sbe \algalg G$, where the limit is taken with respect to the reduced norm. Since $j$ is contractive, $j(a)=\lim j(a_n)$ with respect to $\|\cdot\|_\infty$. Since each $a_n$ is continuous on $G\backslash D$ by \cref{pro:dangerous-points-discontinuity}, the same holds for $j(a)$. 
\end{proof}

The following result is an extension of \cref{lemma:ess-alg-alg} to the reduced \cstar{}algebra.
An alternative proof of the following can also be found in \cite{neshveyevetal2023}*{Proposition 1.12}.
\begin{theorem}[cf.\ \cite{KwasniewskiMeyer-essential-cross-2021}*{Theorem 7.18}]\label{the:KM-ess-J-sing}
Given an étale groupoid $G$ with locally compact Hausdorff unit space $X$,
let $J_\sing\coloneqq\ker(\redalg G \twoheadrightarrow \essalg G)$. Then
$$
J_\sing= \{a\in \redalg G: \supp(j(a^*a)|_X)\mbox{ is meager in } X\},$$ 
and if $G$ is covered by countably many bisections, then 
\begin{align*}
J_\sing&= \{a\in \redalg G: \supp(j(a))\mbox{ is meager in } G\}
\\&=\{a\in \redalg G: \supp(j(a))\sbe D\}
\\&=\{a\in \redalg G: \supp(j(a^*a)|_X)\sbe X\cap D\}.    
\end{align*}
\end{theorem}

    The above result has the interesting consequence that we may view elements of $\essalg G$ as continuous functions on $G\backslash D$.

\begin{corollary}\label{cor:j-ess}
Let $G$ be an étale groupoid with locally compact Hausdorff unit space $X$ and let $D\sbe G$ be the set of dangerous arrows. Assume $G$ is covered by countably many bisections. Then the restriction map $\algalg G\to \contb(G\backslash D)$ from \cref{cor:restriction-map} induces an injective and norm-contractive linear map $j_\ess\colon\essalg{G}\to \contb(G\backslash D)$, where $\contb(G\backslash D)$ is the Banach space of bounded continuous functions $G\backslash D\to \C$ endowed with the sup-norm.
\end{corollary}
\begin{proof}
    Let $R\colon \ell^\infty(G)\to \ell^\infty(G\backslash D)$ be the restriction map; this is a contractive linear map. Composing it with $j\colon \redalg G\to \ell^\infty(G)$ from \cref{rem:j-map-red}, we get a linear contraction
    $$R\circ j\colon \redalg{G}\to \ell^\infty(G\backslash D)$$
    whose image lies within $\contb(G\backslash D)\sbe \ell^\infty(G\backslash D)$, again by \cref{rem:j-map-red}. Moreover, by \cref{the:KM-ess-J-sing}, $\ker(R\circ j)=J_\sing$, so that $R\circ j$ factors through an injective linear contraction $\essalg G\to \contb(G\backslash D)$, as desired.
\end{proof}

\begin{corollary}\label{cor:expectation-ess-cb}
Under the same assumptions as above, there exists a (unique) faithful generalized conditional expectation 
    $$E\colon \essalg G\to \contb(X\backslash D)$$ 
    induced by the restriction map $\algalg G\to \contb(X\backslash D)$.
\end{corollary}
\begin{proof}
    Consider the standard generalized conditional expectation $P\colon \redalg G\to \borelb(X)\sbe \ell^\infty (X)$. Using the $j$-map, this can be more concretely realized as $P(a)=j(a)|_X$. Notice that by \cref{rem:j-map-red}, $P(a)$ is continuous on $X\backslash D$. And by \cref{the:KM-ess-J-sing}, $a\in J_\sing$ if and only if $\supp(P(a^*a))\sbe D$. This means that the composition $R\circ P$ of $P$ with the restriction map $R\colon \ell^\infty(X)\to \ell^\infty(X\backslash D)$ is a generalized conditional expectation with left kernel $\NN_{R\circ P}=J_\sing$, from which the result follows. Notice that $\contz(X)$ embeds into $\contb(X\backslash D)$ because $D$ is meager and hence $X\backslash D$ is dense in $X$.
\end{proof}

\begin{remark}
    The above expectation $\essalg G\to \contb(X\backslash D)$ is basically the essential expectation $E_\L$ when viewed as a map from $\essalg G$. More precisely, recall that $E_\L$ is a map $\fullalg G\to \borelb(X)/\meager(X)$.
    Notice that both $\contb(X\backslash D)$ and $\borelb(X)/\meager(X)$ embed into $\ell^\infty(X)/\meager^\infty(X)$, where $\meager^\infty(X)$ is the ideal of functions $f\in \ell^\infty(X)$ with meager support.
    Using these embeddings, we may write that $E_\L(a)=E(\Lambda^\ess(a))=j_\ess(\Lambda^\ess(a))|_X$ for $a\in \fullalg G$, where $\Lambda^\ess\colon \fullalg{G}\to\essalg{G}$ denotes the quotient map.  
\end{remark}

\subsection{The essential C*-algebra as a reduced C*-algebra} \label{sec:ess-red}
In this subsection we present an alternative approach to \(\essalg{G}\) than that of \cref{def:grpd-c-algs}, as we now define \(\essalg{G}\) as a \emph{reduced} \cstar algebra in the sense that it will be concretely represented via an `essential regular representation'. This can be a more appropriate approach in certain cases, and we apply this later when studying essential amenability.

Recall from \cref{def:lambda x} and \cref{prop:rep red and ess calgs:red} the definition of the regular representation $\Lambda$.
The following result showcases the quotient map \(\Lambda^{\rm red}_{E_\L, P} \colon \redalg{G} \to \essalg{G}\) (cf.\ \cref{cor:diagram:basic}).
In fact, it is also proven in \cite{neshveyevetal2023}*{Proposition 1.12} that the same result below also applies when substituting the set \(D\) of dangerous arrows by the set of \emph{extremely dangerous} arrows. As a consequence, we see that the set of functions \(a \in \algalg{G}\) supported on either of these sets are the same.
\begin{proposition}[cf.\ \cite{neshveyevetal2023}*{Proposition 1.12} and \cite{bkmk-2024-twited-grpds}*{Proposition 4.16}] \label{prop:rep red and ess calgs:ess}
    Let \(G\) be an \'etale groupoid with Hausdorff unit space \(X \subseteq G\).
  If $G$ can be covered by countably many open bisections, the representation
        \[
            \Lambda^\ess\coloneqq\bigoplus_{x \in X \setminus D} \lambda_x \colon \fullalg{G} \to  \B\left(\bigoplus_{x \in X \setminus D}\ell^2\left(Gx\right)\right)
        \]
        induces a faithful representation of \(\essalg{G} \hookrightarrow \B(\oplus_{x \in X \setminus D} \ell^2(Gx))\).
\end{proposition}
\begin{proof}
    Since $\Lambda^\ess$ is just a restriction of $\Lambda$, it is clear that $\Lambda^\ess(\fullalg G)$ is a quotient of $\redalg G=\Lambda(\fullalg G)$.
    Moreover, we have a canonical expectation $Q\colon \Lambda^\ess(\fullalg G)\to \ell^\infty(X\backslash D)$, given by
    $$Q(\Lambda^\ess(a))(x)\coloneqq\braket{\delta_x}{\lambda_x(a)\delta_x},\quad x\in X\backslash D.$$
    An argument as in the proof of \cref{rem:j-map-red} shows that $Q$ is faithful. Furthermore, notice that
    $$Q(\Lambda^\ess(a))=j(\Lambda(a))|_{X\backslash D}=(R\circ P)(\Lambda(a)),$$
    where $R\colon \ell^\infty(X)\to \ell^\infty(X\backslash D)$ denotes the restriction map. It follows from \cref{cor:expectation-ess-cb} that $\Lambda^\ess(\fullalg G)\cong \essalg G$, as desired.     
\end{proof}

\begin{corollary} \label{cor:diagram}
Let \(G\) be an \'etale groupoid with Hausdorff unit space \(X \subseteq G\). The identity map \(\id \colon \algalg{G} \to \algalg{G}\) induces quotient maps making
\begin{center}
\begin{tikzcd}[scale=50em]
\fullalg{G} \arrow[two heads]{d}{\Lambda^{\rm max}_{E_\mathcal{L},P}} \arrow[two heads]{rd}{\Lambda^{\rm ess}} \arrow[two heads]{r}{\Lambda_P} & \redalg{G} \arrow[two heads]{d}{\Lambda^{\rm red}_{E_\mathcal{L},P}} \arrow{r}{P} & \borelb\left(X\right) \arrow[two heads]{d}{\pi_\ess} \\
\essmaxalg{G} \arrow[two heads]{r}{\Lambda_{E_\mathcal{L}}} & \essalg{G} \arrow{r}{E_\L} & \borelb\left(X\right)/\meager\left(X\right)
\end{tikzcd}
\end{center}
commute. Moreover, the following assertions hold:
\begin{enumerate}[label=(\roman*)]
    \item \(P\) and \(E_\L\) are faithful.
    \item \(\Lambda_P\) can be taken to be \(\oplus_{x \in X} \lambda_x\), and \(\Lambda^{\rm ess}\) can be taken to be \(\oplus_{x \in X \setminus D} \lambda_x\) whenever \(D\sbe G\) is meager (e.g. $G$ can be covered by countably many open bisections).
\end{enumerate}
\end{corollary}

We end the section with the following, which we hope are illuminating classes of examples.
\begin{example} \label{ex:collapse}
    Let $G$ be an \'{e}tale groupoid with locally compact Hausdorff unit space $X \subseteq G$.
    \begin{enumerate}[label=(\roman*)]
        \item If $G$ is itself Hausdorff, then the diagram in \cref{cor:diagram} collapses ``vertically'', \emph{i.e.}\ \(\Lambda^{\text{max}}_{E_\mathcal{L},P}\) and \(\Lambda^{\text{red}}_{E_\mathcal{L},P}\) are injective, and hence isomorphisms.
        \item If $G$ is amenable then $\fullalg{G} \cong \redalg{G}$ via the left regular representation (even if $G$ is not Hausdorff, see \cref{thm:wk-cont-nuc} below).
    \end{enumerate}
\end{example}

\begin{example} \label{ex:trivial-action}
  Let $\Gamma$ be a discrete non-trivial group, and define the inverse semigroup 
  \( S \coloneqq \Gamma \sqcup \{0\}, \)
  where $0 \in S$ acts as a zero element on $\Gamma$. Let $X$ be a compact Hausdorff space, and fix a non-isolated point $x_0 \in X$. Define the open, dense, and proper subspace
  \( O \coloneqq X \setminus \{x_0\} \subseteq X. \)
  Suppose $S$ acts trivially on $X$, with the domain of $0 \in S$ defined as $O$. Define the groupoid of germs of this action by
  \( G \coloneqq X \rtimes S. \)
  This is a group bundle, with trivial isotropy at each $x \in O$ and isotropy group $x_0 G x_0 \cong \Gamma$. Hence, as a set, we have
  \( G \cong O \sqcup \Gamma. \)
  The topology on $G$ restricts to the original topology on $O$ and is such that any net $(x_i)_i \subseteq O$ converging to $x_0$ in $X$ also converges to every $\gamma \in \Gamma \subseteq G$. The set of \emph{dangerous arrows} in $G$ is given by $D = \Gamma$. This groupoid appears as a fundamental example in \cite{Khoshkam-Skandalis-reg-rep}.
  
  The algebra $\algalg{G}$ consists of functions $a \colon G \to \mathbb{C}$ that are continuous on $O$ and satisfy the condition
  \begin{equation}\label{eq:lim-condition-functions-C_c(G)}
      \lim_{x \to x_0} a(x) = \sum_{\gamma \in \Gamma} a(\gamma), 
  \end{equation} 
  where the sum is finite. Consequently, we obtain a canonical isomorphism of \Star{}algebras:
  \[ \algalg{G} \cong \left\{(\varphi, \psi) \in C(X) \oplus C_c(\Gamma) : \varphi(x_0) = \epsilon(\psi) \right\}, \]
  where $\epsilon \colon \Gamma \to \mathbb{C}$ denotes the trivial representation.
  
  As a result (see \cite{Khoshkam-Skandalis-reg-rep}), we have the following isomorphisms:
  \begin{enumerate}[label=(\alph*)]
    \item $\fullalg{G} \cong \left\{ (\varphi, \psi) \in C(X) \oplus \fullalg{\Gamma} : \varphi(x_0) = \epsilon(\psi) \right\}$;
    \item $\redalg{G} = \fullalg{G}$ if and only if $\Gamma$ is amenable, which is equivalent to $G$ being amenable;
    \item if $\Gamma$ is not amenable, then $\redalg{G} \cong C(X) \oplus \redalg{\Gamma}$;
    \item and in all cases, $\essmaxalg{G} = \essalg{G}=\essalgalg{G}\cong C(X)$.
  \end{enumerate}
  
  In particular, the algebra $\essalg{G}$ does \emph{not} depend on the group $\Gamma$. This behavior arises because $\redalg{G}$ is quotiented by functions supported on meager sets, and the non-trivial part of the action of $S = \Gamma \sqcup \{0\}$ on $X$ occurs only on the meager set $X \setminus O = \{x_0\}$.
\end{example}

\subsection{A class of essential representations} \label{sec:class of ess reps}
In this subsection we describe a class of essential representations of a given \'etale groupoid \(G\) (cf.\ \cref{prop: a class of ess reps}). A particular sub-class, in fact, has already appeared in \cite{bruce-li:2024:alg-actions} (see \cref{sec:bruce-li} for more details).
We start with the following well-known definition.
\begin{definition} \label{def:inv set of units}
    Let \(G\) be an \'etale groupoid, and let \(X \subseteq G\) be its unit space.
    We say that \(\Gamma \subseteq X\) is \emph{invariant} if \(\s(g) \in \Gamma\) whenever \(\rg(g) \in \Gamma\), that is, $\s^{-1}(\Gamma)=\rg^{-1}(\Gamma)$.
\end{definition}

\begin{example} \label{ex:inv sets}
    \(X \subseteq G\) itself is clearly invariant. By \cref{lemma:dangerous arrows}, the set \(D \cap X\) is also invariant (and so is its complement).
\end{example}

For the following proposition (and only then) we will change our convention in \eqref{eq:convention about AB}. Given a unit \(x \in X\) and a bisection \(s \subseteq G\) such that \(x \in s^*s=\s(s)\), we will denote by \(s \gamma\coloneqq\rg(g)\) the \emph{range} of the (unique) arrow \(g \in s\) with \(\s(g) = \gamma\). In the previous naming convention in \eqref{eq:convention about AB} this would have been denoted by \(s \gamma s^*\). Nevertheless, we find this notation quite troublesome in some of the computations below.
\begin{proposition} \label{prop: a class of ess reps}
    Let \(G\) be an \'etale groupoid with locally compact Hausdorff unit space \(X \subseteq G\). Suppose that \(G\) can be covered by countably many open bisections, and let \(\Gamma \subseteq X\) be a given invariant set. Then
    \begin{align*}
        \pi \colon \algalg{G} & \to \B\left(\ell^2\left(\Gamma\right)\right), \\
        \pi\left(a v_s\right)\left(\delta_{\gamma}\right) & \coloneqq
            \left\{\begin{array}{rl}
                 a\left(s \gamma\right) \delta_{s \gamma} & \text{if } \;\gamma \in s^*s, \\
                 0 & \text{otherwise.}
            \end{array} \right.
    \end{align*}
    defines a representation of \(\algalg{G}\).\footnote{\, Here, \(a v_s\) is a continuous function supported on the open bisection \(s \subseteq G\) and vanishing at infinity, see \cref{notation:cont functions as sums}.} Moreover, if \(\Gamma\) contains no dangerous unit, that is, if $\Gamma\sbe X\backslash D$, then \(\pi\) is essential.
\end{proposition}
\begin{proof}
    We first observe that if \(s \subseteq G\) is an (open) bisection, then \(s\gamma \in \Gamma\) for all \(\gamma \in \Gamma \cap s^*s\) (recall the discussion before the statement of the result). Indeed, this just follows since \(\gamma = \s(g)\) for some \(g \in s\), and then \(s\gamma = \rg(g)\). Invariance of \(\Gamma\) then yields that \(s \gamma \in \Gamma\). The map \(\pi\) is clearly linear, and multiplicativity is a routine computation. In fact, if \(a = \sum_{i = 1}^k a_i v_{s_i}\) and \(b = \sum_{j = 1}^\ell b_j v_{t_j}\), then
    \[
        \pi\left(ab\right) \left(\delta_\gamma\right) = \sum_{i, j} \pi\left(a_i \left(s_i^* b_j\right) v_{s_i t_j}\right)\left(\delta_\gamma\right) = \sum_{\substack{i, j \\ \gamma \in \left(s_it_j\right)^*\left(s_it_j\right)}} a_i\left(s_it_j \gamma\right) b_j\left(t_j \gamma\right) \delta_{s_i t_j \gamma}.
    \]
    Likewise,
    \[
        \pi\left(a\right) \pi\left(b\right) \left(\delta_\gamma\right) = \pi\left(a\right) \sum_{\substack{j \\ \gamma \in t_j^*t_j}} b_j\left(t_j \gamma\right) \delta_{t_j \gamma} = \sum_{\substack{i, j \\ \gamma \in t_j^*t_j, t_j \gamma \in s_i^*s_i}} a_i\left(s_it_j \gamma\right) b_j\left(t_j \gamma\right) \delta_{s_i t_j \gamma}.
    \]
    Observing that \(\gamma \in (s_it_j)^*(s_it_j)\) if and only if \(\gamma \in t_j^*t_j\) and \(t_j \gamma \in s_i^*s_i\), it follows that \(\pi(ab) = \pi(a) \pi(b)\), as desired.

    Now assume that $\Gamma\sbe X\backslash D$. Let \(a \in \algalg{G}\) be such that \(\supp(a^*a)\) is meager (in \(G\)). By \cref{cor:meager supp is in d} it follows that \((a^*a)(g) = 0\) whenever \(g \not\in D\).
    In order to prove that \(\pi\) is essential it is enough to show that \(\pi(a) \delta_\gamma = 0\) for all \(\gamma \in \Gamma\sbe X \setminus D\).
    Say that \(a = a_1 v_{s_1} + \dots + a_k v_{s_k}\) (recall \cref{notation:cont functions as sums}), and \(a^*a = \sum_{i, j = 1}^k \overline{a_i} (s_i a_j) v_{s_i^* s_j}\). Hence,
    \begin{align*}
        \left(a^*a\right)\left(\gamma\right) & = \sum_{h \in G \gamma} \overline{a\left(h\right)} a\left(h\gamma\right) = \sum_{h \in G\gamma} \; \sum_{\substack{i, j = 1, \dots, k \\ h \in s_i \cap s_j}} \overline{a_i\left(\rg(h)\right)} a_j\left(\rg(h)\right) \\
        & = \sum_{\substack{i, j = 1, \dots, k \\ \gamma \in \left(s_i \cap s_j\right)^*\left(s_i \cap s_j\right)}} \overline{a_i\left(s_i \gamma\right)} a_j\left(s_i \gamma\right) = \sum_{\substack{ i, j = 1, \dots, k \\ \gamma \in s_i^* s_i \cap s_j^* s_j }} \langle a_i\left(s_i\gamma\right) \delta_{s_i \gamma}, a_j\left(s_j\gamma\right) \delta_{s_j \gamma}\rangle \\
        & = \langle \pi\left(a\right) \delta_\gamma, \pi\left(a\right) \delta_\gamma\rangle = \norm{\pi\left(a\right) \delta_\gamma}^2.
    \end{align*}
    Therefore, if \(\pi(a) \delta_\gamma \neq 0\) then \((a^*a)(\gamma) \neq 0\), which implies that \(\gamma \in \supp(a^*a) \subseteq D\) (by the assumption on \(a\) and \cref{cor:meager supp is in d}). This contradicts the hypothesis on \(\Gamma\), and hence finishes the proof.
\end{proof}

\begin{remark}
    Observe that the assumption that \(\Gamma\) contains no dangerous unit in \cref{prop: a class of ess reps} does \emph{not} imply that \(G\) is Hausdorff.
    For instance, if \(G\) is topologically free then \(\Gamma\) may be chosen to be dense (even in \cref{cor: ess reps are nice} below).
\end{remark}

\begin{remark}
    The \emph{moreover} part of \cref{prop: a class of ess reps} was found independently in \cite{KwasniewskiMeyer-essential-cross-2021}*{Lemma 7.28} and \cite{bkmk-2024-twited-grpds}*{Lemma 5.12}, where this kind of representations are called an \emph{orbit} or \emph{trivial representation}.
\end{remark}

\begin{corollary} \label{cor: ess reps are nice}
    Let \(\Gamma \subseteq X \subseteq G\) and \(\pi \colon \algalg{G} \to \B(\ell^2(\Gamma))\) be as in \cref{prop: a class of ess reps}. Suppose \(\Gamma\) contains no dangerous unit, and let \(\mathfrak{A}_\pi\) be the (concrete) \cstar algebra generated by the image of \(\pi\). Then \(\pi\) induces a quotient map
    \[
        \pi_{\rm max} \colon \essmaxalg{G} \twoheadrightarrow \mathfrak{A}_\pi \subseteq \B\left(\ell^2\left(\Gamma\right)\right).
    \]
    Moreover, if \(\Gamma\) is dense in \(X\) and \(\gamma G \gamma = \{\gamma\}\) for all \(\gamma \in \Gamma\), then the canonical quotient map \(\Lambda_{E_\L} \colon \essmaxalg{G} \to \essalg{G}\) (cf.\ \cref{cor:diagram}) factors through \(\pi_{\rm max}\), \emph{i.e.}\
    \begin{center}
        \begin{tikzcd}[scale=50em]
            \essmaxalg{G} \arrow[two heads]{r}{\Lambda_{E_\mathcal{L}}} \arrow[two heads]{rd}{\pi_{\rm max}} & \essalg{G} \arrow{r}{E_\L} & \borelb\left(X\right)/\meager\left(X\right) \\
            & \mathfrak{A}_\pi \arrow[two heads]{u}{} \arrow{r}{E} & \ell^\infty\left(\Gamma\right) \arrow{u}{Q_{X, \Gamma}},
        \end{tikzcd}
    \end{center}
    where \(E \colon \B(\ell^2(\Gamma)) \to \ell^\infty(\Gamma)\) denotes the canonical faithful conditional expectation and \(Q_{X, \Gamma}\) denotes the map \(p_A \mapsto [\chi_{\overline{A}}]\) for any \(A \subseteq \Gamma \subseteq X\).
\end{corollary}
\begin{proof}
    The fact that there is some map \(\pi_{\rm max}\) follows from \(\pi\) being essential (cf.\ \cref{prop: a class of ess reps}) and the universal property of \(\essmaxalg{G}\) (recall \cref{def:max ess alg}).

    Suppose now that no unit in \(\Gamma\) has isotropy, \emph{i.e.}\ \(\gamma G \gamma = \{\gamma\}\) for all \(\gamma \in \Gamma \subseteq X\). By \cref{prop:rep red and ess calgs:ess}, the map \(\oplus_{x \in X \setminus D} \lambda_x\) is a faithful representation of \(\essalg{G}\). Consider the maps
    \[
        \lambda_\gamma^\mathfrak{A} \colon \pi\left(\algalg{G}\right) \to \B\left(\ell^2\left(G\gamma\right)\right), \;\;
        \lambda_\gamma^\mathfrak{A}\left(\pi\left(a\right)\right) \coloneqq \lambda_\gamma\left(a\right).
    \]
    One can show that \(\lambda_\gamma^\mathfrak{A}\) is, in fact, well defined (and hence a representation).\footnote{\, This is where we use that \(\gamma G \gamma = \{\gamma\}\) for all \(\gamma \in \Gamma\).}
    The following claim is the key part of the proof.
    \begin{claim} \label{claim:ess reps are nice}
        \(\sup_{x \in X \setminus D} \norm{\lambda_x(a)} = \sup_{\gamma \in \Gamma} \norm{\lambda_\gamma(a)}\) for all \(a \in \algalg{G}\).
    \end{claim}
    \begin{proof}
        Since \(\Gamma\) contains no dangerous unit the inequality ``\(\geq\)'' is clear.
        Let \(\kappa \coloneqq \sup_{x \in X \setminus D} \norm{\lambda_x(a)}\) and let \(\varepsilon \in (0, \kappa/2)\) be given.
        Fix some \(x_0 \in X \setminus D\) in such a way that \(\kappa \approx_\varepsilon \norm{\lambda_{x_0}(a)}\). Since \(\Gamma \subseteq X\) is dense, let \((\gamma_n)_n \subseteq \Gamma\) be some net converging to \(x_0\) (and to nothing else, in fact, since \(x_0 \not\in D\)).
        It suffices to prove that \(\norm{\lambda_{\gamma_n}(a)} \to \norm{\lambda_{x_0}(a)}\), but this is clear since \(\Gamma\) contains no dangerous units, and hence any finite behaviour at \(x_0\) can be witnessed at \(\gamma_n\) for all large \(n\).
        In fact, we may suppose, without loss of generality, that \(0 < \norm{\lambda_{x_0}(a)} - \varepsilon \leq \norm{\lambda_{x_0}(a)(v)}\) for some finitely supported unit vector \(v = v_1 \delta_{g_1} + \dots + v_\ell \delta_{g_\ell} \in \ell^2(G x_0)\).
        Since \(x_0 \not\in D\) (and neither is \(\gamma_n\) for any \(n\)) it follows that there are \((h_{n, j})_{n, j = 1, \dots, \ell} \subseteq G\) such that
        \begin{enumerate}[label=(\roman*)]
            \item the source of \(h_{n, j}\) is \(\gamma_n\) for all \(n\) and \(j = 1, \dots, \ell\); and
            \item the net \((h_{n, j})_n\) converges to \(g_{j}\) for all \(j = 1, \dots, \ell\).
        \end{enumerate}
        It then follows that (for all large \(n\)) \(v_n \coloneqq v_1 \delta_{h_{n,1}} + \dots + v_\ell \delta_{h_{n, \ell}}\) is a unit vector in \(\ell^2(G \gamma_n)\) such that \(\norm{\lambda_{\gamma_n}(a)(v_n)} = \norm{\lambda_{x_0}(a)(v)} \approx_\varepsilon \norm{\lambda_{x_0}(a)} \approx_\varepsilon \kappa\).
        As \(\varepsilon\) was arbitrary, the claim follows.
    \end{proof}

    By the discussion above, \cref{prop:rep red and ess calgs:ess,claim:ess reps are nice} it follows that
    \[
        \norm{a}_{\rm ess} = \sup_{x \in X \setminus D} \norm{\lambda_x\left(a\right)} = \sup_{\gamma \in \Gamma} \norm{\lambda_\gamma\left(a\right)} = \sup_{\gamma \in \Gamma} \norm{\lambda_\gamma^\mathfrak{A} \left(\pi\left(a\right)\right)}
    \]
    for all \(a \in \algalg{G}\). Thus, the map \(\Lambda_{E_\L} \colon \essmaxalg{G} \to \essalg{G}\) does factor through \(\pi_{\rm max}\). The fact that the same factorization of maps happens at the level of conditional expectations, that is, that the map \(Q_{X, \Gamma}\) in the statement of the theorem is well defined and satisfies that
    \[
        Q_{X, \Gamma}\left(E\left(\pi\left(a\right)\right)\right) = E_\L\left(\pi_{\rm red}\left(a\right)\right)
    \]
    for all \(a \in \essalgalg{G}\) is routine to check.
    Indeed, it suffices to show the above for functions supported on \emph{one} bisection, \emph{i.e.}\ \(a = av_s\). In such case \(E_\L(av_s) = [a|_{s \cap X}] \in \borelb(X)/\meager(X)\).
    Likewise, \(E(\pi(a v_s)) = [ \gamma \mapsto a(\gamma)] \in \ell^\infty(\Gamma)\) and, thus, \(Q_{X, \Gamma}(E(\pi(a v_s)))\) can be approximated up to \(\varepsilon\) by a finite linear combination of \([\chi_{\overline{A_i}}]\), where \(A_i \subseteq \Gamma\) are some level sets of \(a\).
\end{proof}

%%%%%%%%%%%%%%%%%%%%%%%%%%%%%%%%%%%%%%%%%%%%%%%%%%%%%%%%%%%%%%%%%%%%%%%%%%%%%%%%%
% BOREL ALGEBRAS
%%%%%%%%%%%%%%%%%%%%%%%%%%%%%%%%%%%%%%%%%%%%%%%%%%%%%%%%%%%%%%%%%%%%%%%%%%%%%%%%%
\section{Borel algebras} \label{sec:aux-alg}
This auxiliary section introduces additional (\cstar{})algebras that properly contain \(\algalg{G}\) (or its completions).  
These new algebras become necessary in \cref{thm:wk-cont-nuc,thm:wk-cont-nuc-ess} (see also \cref{def:ame,def:ess-ame}) specifically in the case when \(G\) is not Hausdorff.
The main issue in the non-Hausdorff setting is that \(\algalg{G}\) is \emph{not} closed under pointwise multiplication. This subtlety plays a crucial role in \cref{def:ame,def:ess-ame}; see also \cref{ex:trivial-action-ame}~\cref{ex:trivial-action-ame:weak}.  

Nevertheless, note that every \(a \in \algalg{G}\) is a bounded, compactly supported, and \emph{Borel} measurable function. Moreover, this class of functions is closed under convolution, as convolution preserves Borel measurability. Exploiting this property, we now define several (\cstar{})algebras.  

\begin{definition} \label{def:borel-algebras}
    Let \(G\) be an \'etale groupoid with unit space \(X \subseteq G\).
    \begin{enumerate}[label=(\roman*)]
        \item \(\borelalg{G}\) is the \Star{}algebra (under convolution) of bounded Borel functions \(a \colon G \to \C\) whose support is contained in a compact set of \(G\).\footnote{\, This is \emph{different} from the function having compact support, which is usually understood as the closed support of the function being compact, as compact subsets need not be closed. In particular, \(G\) may not have many closed compact subsets.}
        We view $\borelalg X$ as the (commutative) \Star{}subalgebra of functions in $\borelalg G$ whose support is contained in $X$.
        \item \(\borelessalg{G}\) is the quotient of \(\borelalg{G}\) by the ideal \begin{equation}\label{eq:ideal-J_D^c}
            J_D^c:=\left\{a \in \borelalg{G} \mid \supp(a) \subseteq D\right\}
        \end{equation} of functions supported on \(D\), that is,
            \(
                \borelessalg{G} \coloneqq \borelalg{G}/J_D^c.
            \)
        \item \(\fullborel{G}\) is the universal \cstar{}algebra of \(\borelalg{G}\).
        \item \(\redborel{G}\) is the completion of \(\borelalg{G}\) under the ``reduced norm'', \emph{i.e.}\
            \[
                \norm{a}_{\rm red} \coloneqq \sup_{x \in X} \norm{\lambda_x\left(a\right)}.\footnote{\, Note that one has to extend \(\lambda_x\) to Borel functions, which is done via the same formula.}
            \]
        \item \(\essmaxborel{G}\) is the universal \cstar{}algebra of \(\borelessalg{G}\).
        \item \(\essborel{G}\) is the completion of \(\borelessalg{G}\) under the ``essential norm'', \emph{i.e.}\
            \[
                \norm{a}_{\rm ess} \coloneqq \sup_{x \in X \setminus D} \norm{\lambda_x\left(a\right)}.
            \]
        \item \(\auxalg{G}\) is the subspace of \(\borelalg{G}\) generated by pointwise multiplications of functions in  \(\algalg G\), that is, elements of $\auxalg G$ are sums of products of the form \(a_1 \cdot a_2\cdots a_n\) for \(a_1,\ldots, a_n \in \algalg G\).
    \end{enumerate}
\end{definition}

\begin{remark} \label{rem:borel-algs}
    We briefly observe the following.
    \begin{enumerate}[label=(\roman*)]
        \item Since compact subsets of \( G \) are not necessarily closed, they may fail to be Borel. Consequently, the characteristic functions of compact sets do not always belong to \( \borelalg G \).
        Nevertheless, any compact subset of \(G\) is contained in a \emph{Borel} compact subset (this uses \cref{lemma:compact set is covered by fin many bisections}). Hence, every $a\in \borelalg G$ is supported on a Borel compact subset.
        Indeed, take any compact $K\sbe G$, and cover it by finitely many open bisections $K\sbe s_1\cup\ldots \cup s_n$. Then the closure $K_i$ of $K\cap s_i$ within $s_i$ is compact and Borel, and $K\sbe K_1\cup\ldots \cup K_n$ is compact.
        \item A standard argument (see \cite{Thomsen}) shows that the convolution and involution of functions still make sense on \(\borelalg{G}\) and turn it into a \Star{}algebra. Moreover, it admits a largest \cstar{}norm and, thus, $\fullborel{G}$ exists. Similarly $\essmaxborel G$ also exists, cf.\ \cref{cor:borel-algs}. We also note that the subspace of functions \(a\in \borelalg G\) with \(\supp(a)\sbe D\) is an ideal. This fact follows from the fact that \(D\) is an ideal subgroupoid of \(G\), see \cref{lemma:dangerous arrows}~\cref{lemma:dangerous arrows:subgroupoid}.
        \item Observe that \(a v_s \in \borelalg{G}\) for all open bisection \(s \subseteq G\) and Borel compactly supported \(a \colon ss^* \to \C\). Indeed, if \(B \subseteq \C\) is Borel, then
        \begin{align*}
            \left(av_s\right)^{-1}\left(B\right) & = \left(av_s\right)^{-1}\left(B \setminus \{0\}\right) \sqcup \left(av_s\right)^{-1}\left(0\right) \\
            & = a^{-1}\left(B \setminus \{0\}\right) \cdot s \sqcup G \setminus s \sqcup a^{-1}\left(0\right) \cdot s.
        \end{align*}
        The right hand side is the union of two Borel and a closed set, hence Borel as well. In particular, \(\algalg{G} \subseteq \auxalg{G} \subseteq \borelalg{G}\). Likewise, from \cref{cor:nucleus is in d} and the previous item it follows that \(\essalgalg{G} \subseteq \borelessalg{G}\).
        \item Assuming that $G$ can be covered by countably many open bisections, it follows from \cref{prop:rep red and ess calgs:red,prop:rep red and ess calgs:ess} and the previous item that \(\redalg{G} \subseteq \redborel{G}\) and \(\essalg{G} \subseteq \essborel{G}\) canonically, \emph{i.e.}\ these embeddings are extensions of their algebraic counterparts.
    \end{enumerate}
\end{remark}

\begin{lemma} \label{lemma:structure of mathfrak a}
    The space \(\auxalg{G}\) defined in \cref{def:borel-algebras} is exactly the linear span of monomials of the form \(a_1 \cdots a_n\), where \(a_i\) is a continuous, compactly supported function on some open bisection \(s_i \subseteq G\).
\end{lemma}
\begin{proof}
This follows from the fact that every $a\in \algalg G$ is a sum of functions as in the statement (i.e. compactly supported and continuous on an open bisection).
\end{proof}

\begin{remark}
    It would be interesting to know if/when the monomials \(a_1 \cdots a_n\) appearing in \cref{lemma:structure of mathfrak a} can be taken to be of \emph{uniformly} bounded length.
    This is akin to the phenomenon of \emph{contracting} self-similar groups (and their associated groupoids), as studied, for instance, in \cites{szakacs-2021,szakacs-2023} and references therein.
\end{remark}

\begin{lemma}\label{lem:ess-norm-supp-bisection}
If \( a\in \borelalg G \), then
$$\|a\|_\infty\leq \|a\|_{\rm red} \leq \|a\|_\max,$$
and if $\supp(a)$ is a (Borel but not necessarily open) bisection, then  
\[
\|a\|_\ess =\|a\|_{\infty,\ess} \coloneqq \sup_{g\in G\setminus D} |a(g)|.
\]
\end{lemma}

\begin{proof}
The first estimates follow the standard arguments, as in \cref{lemma:norms red infty max} for elements in $\algalg G$. We only prove the last estimate.
We begin by noting that
\[
\|a\|^2_\ess = \|a^*a\|_\ess = \sup_{x\in X\setminus D} \|\lambda_x(a^*a)\|.
\]
Since \( s \coloneqq \supp(a)\) is a bisection, we have  
\[
\supp(a^*a) \subseteq s^*s = \s(s) \subseteq X,
\]
which implies that \( \lambda_x(a^*a) \) acts as the (pointwise) multiplication operator by the function \( (a^*a)\circ\rg \). Therefore $\|\lambda_x(a^*a)\|=\sup_{g\in Gx}|(a^*a)(\rg(g))|$.
Since  
\((a^*a)(\rg(g)) = |a(g)|^2\) if  $\rg(g) \in \s(s)$, and is zero otherwise,
and because \( x\in X\setminus D \) if and only if \( \s^{-1}(x) \subseteq G\setminus D \), we obtain  
\[
\|\lambda_x(a^*a)\| = \sup_{g\in G\setminus D} |a(g)|^2.
\]
Taking square roots on both sides completes the proof.
\end{proof}

Observe that $\borelalg G$ is also a (commutative) \Star{}algebra under the pointwise operations of multiplication and conjugation. Indeed, notice that, by definition
\[\borelalg G=\bigcup_{K\sbe G} B_K(G),\footnote{\, By \(B_K(G)\) we mean the functions in \(\borelalg{G}\) whose support is contained in \(K\).}\]
where the union runs over all compact subsets of $G$ (we could also restrict the union to compact Borel subsets).
Each $B_K(G)$ is then a \cstar{}subalgebra of the (commutative) \cstar{}algebra $\borelb(G)$ of all bounded Borel functions $G\to \C$.\footnote{\, Here we start considering \(B_K(G)\) as a commutative \cstar{}algebra under pointwise multiplication. This is due to technical reasons, and will be important in \cref{lemma:taking square roots}.}  

The following works as a replacement for (in general non-existing) partitions of unit in non-Hausdorff groupoids.
\begin{lemma}\label{lem:partition-into-borel-bisections}
   Given $K\sbe G$ compact with $K\sbe U \coloneqq s_1\cup \ldots \cup s_l$, for open bisections $s_1, \dots, s_l \subseteq G$, there is a partition of $U=t_1\sqcup \ldots \sqcup t_l$ into Borel bisections $t_1,\ldots,t_l\sbe G$. Moreover, we can choose $U\sbe L$ for some compact Borel subset $L\sbe G$, so that the characteristic functions $\chi_{t_i}\in \borelalg G$.
   Thus, every $a\in B_K(G)$ can be decomposed as 
   \begin{equation}\label{eq:decomposition-in-Bc(G)}
   a=a_1+\ldots +a_l,
   \end{equation}
   where $a_i \coloneqq \chi_{t_i}\cdot a\in B_K(G)$ have pairwise disjoint Borel supports.
\end{lemma}
\begin{proof}
    Cover $K$ by finitely many open bisections $s_1,\ldots, s_l$, for some natural number $l$ (that only depends on $K$, cf.\ \cref{lemma:compact set is covered by fin many bisections}).
    It follows that the closure $K_i \coloneqq \overline{s_i\cap K}^{s_i}$ as a subset of $s_i$ is compact and (closed in $s_i$, as the latter is Hausdorff), hence, Borel.
    Decompose the bisections $s_i$ into the Borel sets
    \begin{itemize}
        \item $t_1 \coloneqq s_1 \cap K_1$,
        \item $t_2 \coloneqq (s_2\backslash t_1) \cap K_2$, ... and
        \item $t_l \coloneqq (s_l\backslash (t_1\cup\ldots\cup t_{l-1})) \cap K_l$.
    \end{itemize}
    This gives a cover of $K$ by disjoint Borel bisections $t_1,\ldots,t_l$, with $t_i\sbe s_i\sbe K_1 \cup \dots \cup K_l$ for all $i$. Then $\chi_1+\ldots+\chi_l=\chi_K$ and therefore the decomposition~\eqref{eq:decomposition-in-Bc(G)} follows for all $a\in \borelalg G$ with $\supp(a)\sbe K$.
\end{proof}

The following result demonstrates a strong connection between \(\algalg{G}\) and \(\borelalg{G}\) through their representations. Specifically, we show that every representation of \(\algalg{G}\) admits a canonical extension to \(\borelalg{G}\) that preserves commutants.
\begin{lemma}\label{lem:extension-rep-to-borel}
    Let \(G\) be covered by countably many open bisections. Any representation \(\pi\colon \algalg G\to \B(\Hilb)\) extends to a representation \(\pi^{\rm B} \colon \borelalg{G} \to \B(\Hilb)\) with the same commutants in $\B(\Hilb)$, that is, $\pi(\algalg G)'=\pi^{\rm B}(\borelalg G)'$.
\end{lemma}
\begin{proof}
    Given \(\xi,\eta\in \Hilb\), consider the coefficient functional \(\pi_{\xi,\eta}\colon \algalg G\to \C\) given by $a\mapsto \pi_{\xi,\eta}(a)\coloneqq \braket{\xi}{\pi(a)\eta}$. Fix an open bisection $s\sbe G$. Restricting $\pi_{\xi,\eta}$ to functions in $\contc(s)$, we get a linear functional $\pi_{\xi,\eta}^s\colon \contc(s)\to \C$. This is bounded for the supremum norm, because $\pi$ is contractive, that is, we have
    $$\|\pi_{\xi,\eta}^s(a)\|\leq \norm{a}_\infty \cdot \norm{\xi} \norm{\eta} \quad \text{ for all } \; a\in \contc(s).$$
    Hence $\pi_{\xi,\eta}^s$ extends to a bounded linear functional on $\contz(s)$, and therefore there exists a complex bounded Borel measure $\mu_{\xi,\eta}^s$ on $G$ (with support contained in $s$) that represents $\pi_{\xi,\eta}^s$, in the sense that
    $$\pi_{\xi,\eta}^s(a)=\int a(g)d\mu_{\xi,\eta}^s(g)\quad \text{ for all } \; a\in \contc(s).$$
    An argument as in the proof of \cite{BussExel:Fell.Bundle.and.Twisted.Groupoids}*{Theorem~2.13} shows that the family of measures $(\mu_{\xi,\eta}^s)_{s \in S}$ is compatible, in the sense that $\mu_{\xi,\eta}^s(B)=\mu_{\xi,\eta}^t(B)$ whenever $B$ is a Borel subset of $s\cap t$ for $s,t$ two open bisections. Since $G$ is covered by countably many open bisections, a standard argument (as in \cite{BussExel:Fell.Bundle.and.Twisted.Groupoids}*{Theorem~2.13}) shows that there exists a complex bounded Borel measure $\mu_{\xi,\eta}$ on $G$ extending all $\mu_{\xi,\eta}^s$. We can then define for $a\in \borelalg{G}$, $\pi^{\rm B}(a)\in \B(\Hilb)$ to be the unique operator satisfying
    \begin{equation}\label{eq:def-tilde-pi}
        \braket{\xi}{\pi^{\rm B}(a)\eta}=\int a(g)d\mu_{\xi,\eta}\quad \text{ for all } \; \xi,\eta\in \Hilb.
    \end{equation}
    Let us now show that $\pi^{\rm B}\colon \borelalg{G} \to \B(\Hilb)$ is a representation. First we show that $\pi^{\rm B}(a)^*=\pi^{\rm B}(a^*)$ for all $a\in \borelalg{G}$. Given $\xi,\eta\in \Hilb$, we have
    $$\braket{\xi}{\pi^{\rm B}(a)^*\eta}=\braket{\pi^{\rm B}(a)\xi}{\eta}=\overline{\braket{\eta}{\pi^{\rm B}(a)\xi}}=\overline{\int a(g)d\mu_{\eta,\xi}(g)},$$
    and
    $$\braket{\xi}{\pi^{\rm B}(a^*)\eta}=\int a^*(g)d\mu_{\xi,\eta}(g)=\int \overline{a(g^{-1})}d\mu_{\xi,\eta}(g)=\overline{\int a(g^{-1})d\overline{\mu}_{\xi,\eta}(g)}.$$
    We know that $\pi(a)^*=\pi(a^*)$ for all $a\in \algalg G$. By the above computation, this means that the measure $\mu_{\eta,\xi}$ coincides with $\overline{\mu}_{\xi,\eta}\circ\mathrm{inv}$, the pullback of $\overline{\mu}_{\xi,\eta}$ along the inversion map $\mathrm{inv}\colon G\to G$. But then the same computation above implies that $\braket{\xi}{\pi^{\rm B}(a)^*\eta}=\braket{\xi}{\pi^{\rm B}(a^*)\eta}$. Since $\xi,\eta\in \Hilb$ were arbitrary, we get $\pi^{\rm B}(a)^*=\pi^{\rm B}(a^*)$, as desired. The proof that $\pi^{\rm B}$ is multiplicative, that is, $\pi^{\rm B}(ab) = \pi^{\rm B}(a)\pi^{\rm B}(b)$ is similar. Fixing $\xi,\eta\in\Hilb$, we consider the equation
    \begin{equation}\label{eq:tilde-pi-mult}
        \braket{\xi}{\pi^{\rm B}(ab)\eta}=\braket{\xi}{\pi^{\rm B}(a)\pi^{\rm B}(b)\eta}.
    \end{equation}
    We claim that if this equation holds for a fixed $a\in \borelalg{G}$ and for all $b\in \algalg{G}$, then it also holds for the same $a\in \borelalg{G}$ and for all $b\in \borelalg{G}$. Indeed, as every $a\in \borelalg{G}$ is a finite sum of functions supported on open bisections, we may already assume that $a$ is supported on some open bisection $s\sbe G$. Since we already know that $\braket{\xi}{\pi^{\rm B}(a)\pi^{\rm B}(b)\eta}=\braket{\pi^{\rm B}(a^*)\xi}{\pi^{\rm B}(b)\eta}$, \eqref{eq:tilde-pi-mult} is equivalent to 
    $$\int (ab)(g)d\mu_{\xi,\eta}(g)=\int b(g)d\mu_{\pi^{\rm B}(a^*)\xi,\eta}(g).$$
    Now notice that the convolution product $a*b$ is given by $(a*b)(g)=a(h)b(h^{-1}g)$ if there is (a necessarily unique) $h\in s$ with $\rg(h)=\rg(g)$, and it is zero otherwise. We may write this as $(a*b)(g)=a_s(g)\alpha_s(b)(g)$, where $a_s(g)\coloneqq a(\rg(s^*\rg(g)s))$ and $\alpha_s(b)(g) \coloneqq b(h^{-1}g)$ if there is $h\in s$ with $\rg(h)=\rg(g)$, and zero otherwise. \cref{eq:tilde-pi-mult} for $b\in \algalg{G}$ is then equivalent to the equality of the measures
    \[
        a_s(g)d\theta_s(\mu_{\xi,\eta})(g) = d\mu_{\pi^{\rm B}(a^*)\xi,\eta}(g),
    \]
    where $\theta_s(\mu)(B) \coloneqq \mu(s\cdot B)$. But this, in turn, means that~\eqref{eq:tilde-pi-mult} also holds for all $b\in \borelalg{G}$. Therefore, if~\eqref{eq:tilde-pi-mult} holds for a fixed $a\in \borelalg{G}$ and all $b\in \algalg{G}$, then it also holds for the same $a$ and all $b\in \borelalg{G}$. By symmetry (or by replacing $a$ and $b$ by $a^*$ and $b^*$ and using that $\pi^{\rm B}$ is already involutive), if~\eqref{eq:tilde-pi-mult} holds for a fixed $b\in \algalg{G}$ and all $a\in \algalg{G}$, then it also holds for the same $b$ and all $a\in \borelalg{G}$. However, we know that~\eqref{eq:tilde-pi-mult} holds for all $a,b\in \algalg{G}$. It follows from the claim that it holds for all $a,b\in \borelalg{G}$, as desired.

    To prove the statement about the commutants, take \(T\in \pi(\algalg G)'\), that is, \(\braket{\xi}{T\pi(a)\eta} = \braket{\xi}{\pi(a)T\eta}\) for all \(\xi,\eta\in \Hilb\) and all \(a\in \algalg G\).
    This means that the measures \(\mu_{T^*\xi,\eta}\) and \(\mu_{\xi,T\eta}\) coincide, and therefore also \(T\in \pi^B(\borelalg G)'\).
\end{proof}

\begin{remark}\label{rem:non-essential-Borel-rep}
  Even if $\pi$ is an essential representation of $\algalg G$, the extended representation $\pi^{\rm B}\colon \borelalg G\to \B(\Hilb)$ is not necessarily \emph{`Borel essential'}, in the sense that $\pi^{\rm B}$ need not vanish on the ideal $J_D^c \subseteq \borelalg G$ of functions supported on $D$, as defined in \eqref{eq:ideal-J_D^c}. Notably, it is possible for $\redalg G$ to coincide with $\essalg G$ even when $G$ is non-Hausdorff; see \cite{CEPSS-2019}. In such cases, the singular ideal is zero, so that every representation of $\algalg G$ is essential and therefore $\essmaxalg G=\fullalg G$. Let $x \in X$ be a dangerous unit, and consider the regular representation $\lambda_x\colon \algalg G\to \B(\Hilb)$. The extended representation $\lambda_x^B$ coincides with the canonical left regular representation $\borelalg G\to \B(\Hilb)$, as used in \cref{def:borel-algebras}. This representation does not vanish on the ideal $J_D^c$; for instance, the characteristic function $\chi_x \in \borelalg G$ of the singleton $\{x\}$ belongs to $J_D^c$, and yet $\lambda_x(\chi_x) \neq 0$.  

  We thank Xin Li for bringing this example to our attention.
\end{remark}

We end the section with two immediate corollaries. The first of these should be compared with \cref{cor:diagram}.
\begin{corollary} \label{cor:borel-algs}
    Let \(G\) be an \'etale groupoid that is covered by countably many open bisections. The identity map \(\text{id} \colon \algalg{G} \to \algalg{G}\) induces maps making the following diagram commute:
    \begin{center}
        \begin{tikzcd}[scale=50em,row sep=large, column sep=large]
            \algalg{G} \arrow[hook]{r}{} \arrow[two heads]{d}{} & \borelalg{G} \arrow[two heads]{d}{} \arrow[hook]{r}{{\rm dense}} & \fullborel{G} \arrow[two heads]{d}{} \arrow[two heads]{r} & \redborel{G} \arrow[two heads]{d}{} \\
            \essalgalg{G} \arrow[hook]{r}{} & \borelessalg{G} \arrow[hook]{r}{{\rm dense}} & \essmaxborel{G} \arrow[two heads]{r}{} & \essborel{G}.
        \end{tikzcd}
    \end{center}
\end{corollary}
\begin{proof}
    The inclusions \(\algalg{G} \subseteq \borelalg{G}\) and \(\essalgalg{G} \subseteq \borelessalg{G}\) were already discussed in \cref{rem:borel-algs}. We turn our attention to \(\borelalg{G} \subseteq \redborel{G}\) and \(\borelessalg{G} \subseteq \essborel{G}\). Take some non-zero \(a \in \borelalg{G}\).
    It suffices to show that \(\oplus_{x \in X} \lambda_x(a) \neq 0\).
    Since \(a \neq 0\) there is some \(g \in G\) with \(a(g) \neq 0\). Putting \(x \coloneqq g^{-1}g\), one has that
    \[
        \langle \lambda_x(a) \delta_x, \delta_g \rangle = a(g) \neq 0,
    \]
    and thus \(\lambda_x(a) \neq 0\), as desired. It follows that \(\borelalg{G}\) does embed densely into \(\redborel{G}\) and, since the latter is quotiented onto by \(\fullborel{G}\), it also follows that \(\borelalg{G}\) embeds densely into \(\fullborel{G}\). 
    
    The analogous claims for the essential counterparts is more subtle. Note 
    \[
        \left\{a \in \borelalg G \mid \supp(a) \subseteq D\right\} \subseteq \ell^\infty\left(D\right),
    \]    
    (as vector spaces). Again, this follows from \(D\) being an ideal subgroupoid, cf.\ \cref{lemma:dangerous arrows}. In particular, it follows that there is a (purely algebraic) faithful \(j\)-map
    \[
        j^{\rm B}_\ess \colon \borelessalg G \to \ell^\infty\left(G \setminus D\right).
    \]
    If \(a \in \borelessalg{G}\) is non-zero then there is some \(g \in G\setminus D\) such that \(j^{\rm B}_\ess(a)(g) \neq 0\). Again, putting \(x \coloneqq g^{-1}g\), one has that
    \(\langle \lambda_x(a) \delta_x, \delta_g \rangle = j^{\rm B}_\ess(a)(g) \neq 0\), proving that \(\oplus_{x \in X \setminus D} \lambda_x\) is a faithful representation of \(\borelessalg{G}\).

    Lastly, the quotient maps between the \cstar{}algebras exist by construction, since the norms align.
\end{proof}

Before \cref{cor:inclusions c into b} we just recall the following definition, see \cite{Pisier}.

\begin{definition}
    An inclusion of \cstar{}algebras \(\iota \colon A \into B\) is \emph{max-injective} if it induces an embedding under the maximal tensor product. More precisely, for every other \cstar{}algebra $C$, the canonical homomorphism
    \[
        \iota\otimes_{\max}\id_C\colon A\otimes_{\max}C\to B\otimes_{\max}C
    \]
    is injective.
\end{definition}

\begin{corollary} \label{cor:inclusions c into b}
    Let \(G\) be an \'etale groupoid that is covered by countably many open bisections. The canonical inclusion \(\algalg G \subseteq \borelalg G\) extends to
    \begin{enumerate}[label=(\roman*)]
        \item \label{cor:inclusions c into b:1} an embedding \(\fullalg G \subseteq \fullborel G\); 
        \item \label{cor:inclusions c into b:2} a \Star{}homomorphism \(\psi \colon \essmaxalg{G} \to \essmaxborel G\);
        \item \label{cor:inclusions c into b:3} an embedding \(\redalg G \subseteq \redborel G\); and
        \item \label{cor:inclusions c into b:4} an embedding \(\essalg{G} \subseteq \essborel G\).
    \end{enumerate}
    Moreover, the embedding in \cref{cor:inclusions c into b:1} is max-injective.
\end{corollary}
\begin{proof}
    The embeddings \cref{cor:inclusions c into b:3,cor:inclusions c into b:4} were discussed in \cref{rem:borel-algs}. For \cref{cor:inclusions c into b:1,cor:inclusions c into b:2}, observe that the canonical embeddings \(\algalg G \subseteq \borelalg{G} \subseteq \fullborel{G}\) and \(\essalgalg G \subseteq \borelessalg{G} \subseteq \essmaxborel{G}\) naturally extend to \Star{}homomorphisms \(\varphi \colon \fullalg{G} \to \fullborel{G}\) and \(\psi \colon \essmaxalg{G} \to \essmaxborel{G}\) due to the universal properties of the domains. It suffices to show that \(\varphi\) is isometric.

    Take some representation \(\pi \colon \fullalg G \to \B(\Hilb)\). By \cref{lem:extension-rep-to-borel}, the restriction of \(\pi\) to \(\algalg{G}\) extends to \(\borelalg{G}\), and thus to a map \(\pi^{\rm B} \colon \fullborel{G} \to \B(\Hilb)\) by the universal property of the latter. Thus,
    \[
        \norm{\pi(a)} = \norm{\pi^{\rm B}(a)} = \norm{\pi^{\rm B}(\varphi(a))}
    \]
    for all \(a \in \algalg{G} \subseteq \fullalg{G}\). Since \(\pi\) was arbitrary, this shows that \(\norm{\varphi(a)} = \norm{a}\) for all \(a \in \algalg{G}\) and, in particular, \(\varphi\) is isometric.

    Finally, the max-injectivity of the embedding in \cref{cor:inclusions c into b:1} follows directly from \cite{Brown-Ozawa:Approximations}*{Proposition 3.6.6}, which gives a characterization of max-injective embeddings $\iota\colon A\into B$ in terms of the existence of c.c.p.\ maps $\varphi\colon B\to \pi(A)''$ satisfying $\varphi(a)=\pi(a)$ for all $a\in A$ whenever $\pi\colon A\to \B(\Hilb)$ is a representation. Indeed, for the embedding in \cref{cor:inclusions c into b:1} we can even find a \Star{}homomorphism $\varphi\colon B\to \pi(A)''$ (namely, $\varphi=\pi^{\mathrm B}$) satisfying the above condition by \cref{lem:extension-rep-to-borel}. 
\end{proof}

\begin{remark} \label{remark:problem extending to essential}
    The argument used in the proof of \cref{cor:inclusions c into b}~\cref{cor:inclusions c into b:1} does \emph{not} extend to \cref{cor:inclusions c into b:2} due to the obstruction discussed in \cref{rem:non-essential-Borel-rep}. Specifically, it is unclear whether the canonical map \(\psi \colon \essmaxalg{G} \to \essmaxborel{G}\) is max-injective. This difficulty arises because \cref{lem:extension-rep-to-borel} does not apply to essential representations, as already noted in \cref{rem:non-essential-Borel-rep}.

    Nevertheless, it remains possible that \(\psi\) is max-injective. In fact, we are not aware of any examples where \(\psi\) fails to be injective. For instance, if \(\essmaxalg{G} = \essalg{G}\) is nuclear, then \(\psi\) is max-injective: since \(\essalg{G}\) embeds into \(\essborel{G}\), and every embedding of a nuclear \(\mathrm{C}^*\)-algebra is necessarily max-injective.
\end{remark}

\begin{notation}
Given $K\sbe G$ compact, we shall write $B_K(G,A)$ for the Banach space completion of $B_K(G)\odot A$ viewed as a subspace of the Banach space $B_b(G,A)$ with respect to the supremum norm.
Viewing $B_K(G)$ as a commutative \cstar{}algebra with pointwise multiplication, $B_K(G,A)$ is, as Banach space, just the \cstar{}tensor product $B_K(G)\otimes A$.
\end{notation}

\begin{lemma}\label{lem:embedding-max-tensor}
Let \(A\) be an arbitrary \cstar{}algebra. Given a compact subset $K\sbe G$, the inclusion map $B_K(G)\odot A\into \redborel G\otimes_{\max} A$ extends to a continuous linear embedding 
$$\iota_K\colon B_K(G,A)\into \redborel G\otimes_{\max}A$$
that remains injective after composition with the canonical quotient homomorphism $Q\colon \redborel G\otimes_{\max}A \onto \redborel G\otimes_{\min}A$. More precisely, there is a constant $C_K$ satisfying
$$\|x\|_\infty \leq \|Q(\iota_K(x))\|_{\redborel G\otimes_{\min}A}\leq \|\iota_K(x)\|_{\redborel G\otimes_{\max}A}\leq C_K\cdot\|x\|_{\infty}$$
for all $x \in B_K(G,A)$. Likewise, $\iota_K$ induces an embedding 
$$\iota_K^\ess\colon B_K^\ess(G,A)\into \essborel G\otimes_{\max}A$$
satisfying the inequalities
$$\|x\|_{\infty,\ess} \leq \|Q(\iota_K^\ess(x))\|_{\essborel G\otimes_{\min}A}\leq \|\iota_K^\ess(x)\|_{\essborel G\otimes_{\max}A}\leq C_K\cdot\|x\|_{\infty,\ess},$$
where $Q^\ess\colon \essborel G\otimes_{\max}A \onto \essborel G\otimes_{\min}A$ denotes the quotient map, and
\[\|x\|_{\infty,\ess} \coloneqq \sup_{g\in G\backslash D}\|x(g)\|.\]
\end{lemma}
\begin{proof}
    Let $b=\sum_{i=1}^nf_i\otimes a_i\in B_K(G)\odot A$ with $f_i\in B_K(G)$ and $a_i\in A$, $i=1,\ldots, n$. Choose a cover $t_1,\ldots,t_l$ of $K$ as in \cref{lem:partition-into-borel-bisections}.    
    Defining $f_{i,j}:=f_i\cdot \chi_{t_j}$, we can decompose $b$ as
    $$b=\sum_{j=1}^lb_j=\sum_{j=1}^l\left(\sum_{i=1}^n f_{i,j}\otimes a_i\right),$$
    where $b_j\coloneqq\sum_{i=1}^n f_{i,j}\otimes a_i$ and each $f_{i,j}$ is supported on the bisection $t_j$ for all $i=1,\ldots, n$. It follows that $\supp(f_{i,j}^**f_{k,j})\sbe X$. Since $\borelb(X)$ is a (commutative and hence) nuclear \cstar{}algebra, there exists a unique \cstar{}norm on $\borelb(X)\odot A$, so that, for each $j=1,\ldots, l$, we get
    \begin{align*}
    \left\|\sum_{i=1}^n f_{i,j}\otimes a_i\right\|^2=\left\|\sum_{i,k=1}^nf_{i,j}^**f_{k,j}\otimes a_i^*a_k\right\|
    =\sup_{x\in X}\left\|\sum_{i,k=1}^n\sum_{g\in G_x}\overline{f_{i,j}(g)}f_{k,j}(g)a_i^*a_k\right\|\\
    =\sup_{g\in G }\left\|\sum_{i,k=1}^n\overline{f_{i,j}(g)}f_{k,j}(g)a_i^*a_k\right\|=\left(\sup_{g\in G }\left\|\sum_{i=1}^nf_{i,j}(g)a_i\right\|\right)^2.
    \end{align*}
    Therefore, defining $C_K \coloneqq l=l(K)$, we get
    $$\|\iota_K(b)\|\leq \|\iota_K(b_1)\|+\ldots + \|\iota_K(b_l)\|=\|b_1\|_\infty+\ldots + \|b_l\|_\infty\leq C_K\cdot \|b\|_\infty.$$
    Finally, the inequality $\|b\|_\infty \leq \|Q(\iota_K(b))\|$ is obtained by representing $\redborel G\otimes_{\min}A$ on $\ell^2(G)\otimes \Hilb$ via $\lambda\otimes\pi$, where $\pi\colon A\to \B(\Hilb)$ is a faithful representation. Indeed, for $b\in B_K(X)\odot A$ as above, the norm $\|Q(\iota_K(b))\|$ is then given by $\sup_{x\in X}\|(\lambda_x\otimes\pi)(b)\|$, and from \cref{rem:j-map-red} we have
    $$\braket{\delta_g\otimes\eta}{(\lambda_{\s(g)}\otimes\pi)(b)(\delta_{\s(g)}\otimes\eta)}=\braket{\xi}{\pi(b(g))\eta}$$
    for all $g\in G$, $\xi,\eta\in \Hilb$. This yields the desired norm inequality.

    The statements involving the essential Borel algebra are proved in a similar way, using the definition of the reduced essential norm and the equality $\|a\|_\ess=\|a\|_{\infty,\ess}$ for $a\in \borelalg G$ supported on a bisection, see \cref{lem:ess-norm-supp-bisection}. 
\end{proof}

The above result, in particular, allows us to view $B_K(G,A)$ (resp.\ $B_K^\ess(G,A)$) as a closed subspace of both tensor products $\redborel G\otimes_{\max}A$ and $\redborel G\otimes_{\min}A$ (resp.\ $\essborel G\otimes_{\max}A$ and $\essborel G\otimes_{\min}A$).

\begin{corollary}\label{cor:norms-equivalent-B_K(G)}
All norms $\|\cdot\|_\max$, $\|\cdot\|_{\rm red}$ and $\|\cdot\|_\infty$ are equivalent on $B_K(G)$, in particular this is a closed subspace of both $\fullborel{G}$ and $\redborel G$.

Similarly, all norms $\|\cdot\|_{\max,\ess}$, $\|\cdot\|_\ess$ and $\|\cdot\|_{\infty,\ess}$ are equivalent on $B_K^\ess(G)$ and this is a closed subspace of both $\essmaxborel{G}$ and $\essborel G$.
\end{corollary}

We end the section observing the following relationship between the kernels of the canonical quotient maps between reduced and essential \cstar{}algebras and their Borel counterparts.
\begin{proposition} \label{thm:quotient maps}
    Let \(G\) be an \'etale groupoid that is covered by countably many open bisections. Consider the assertions:
    \begin{enumerate}[label=(\roman*)]
        \item \label{thm:quotient maps:borel} The canonical map \(\redborel{G} \twoheadrightarrow \essborel{G}\) is injective.
        \item \label{thm:quotient maps:red} The canonical map \(\redalg{G} \twoheadrightarrow \essalg{G}\) is injective.
        \item \label{thm:quotient maps:zero} \(\NN_{E_\L} \cap \algalg G = 0\).
    \end{enumerate}
    Then \cref{thm:quotient maps:borel} \(\Rightarrow\) \cref{thm:quotient maps:red} \(\Rightarrow\) \cref{thm:quotient maps:zero}.
\end{proposition}
\begin{proof}
    \cref{thm:quotient maps:borel} clearly implies \cref{thm:quotient maps:red}, since \(\redalg G \subseteq \redborel G\) canonically (and similarly for the essential algebras, see \cref{cor:borel-algs}).
    The fact that \cref{thm:quotient maps:red} implies \cref{thm:quotient maps:zero} is also clear, as \(\NN_{E_\L} \cap \algalg{G} \subseteq \NN_{E_\L}\), which is the kernel of the left regular representation \(\redalg{G} \twoheadrightarrow \essalg{G}\).
\end{proof}

%%%%%%%%%%%%%%%%%%%%%%%%%%%%%%%%%%%%%%%%%%%%%%%%%%%%%%%%%%%%%%%%%%%%%%%
% ESSENTIAL AMENABILITY AND NUCLEARITY OF GROUPOID ALGEBRAS
%%%%%%%%%%%%%%%%%%%%%%%%%%%%%%%%%%%%%%%%%%%%%%%%%%%%%%%%%%%%%%%%%%%%%%%
\section{Essential amenability and nuclearity of groupoid algebras} \label{sec:ess-ap}
We now wish to study when the \emph{essential left regular representation}
\[
    \Lambda_{E_\L} \colon \essmaxalg{G} \twoheadrightarrow \essalg{G} \;\; \text{(cf.\ \cref{cor:diagram})}
\]
is an isomorphism. When \(G\) is Hausdorff, the above property is usually called the \emph{weak containment property}. In the non-Hausdorff scenario (both ``reduced'' and ``essential'' cases) we define the following (recall \cref{cor:diagram}).
\begin{definition} \label{def:weak containment}
    An \'etale groupoid \(G\) has the:
    \begin{enumerate}[label=(\roman*)]
        \item \emph{weak containment property} if \(\fullalg{G} \cong \redalg{G}\) via the map \(\Lambda_P\); and
        \item \emph{essential containment property} if \(\essmaxalg{G} \cong \essalg{G}\) via the map \(\Lambda_{E_\L}\).
    \end{enumerate}
    The map \(\Lambda_P\) (resp.\ \(\Lambda_{E_\L}\)) is called the \emph{left regular representation} (resp.\ \emph{essential left regular representation}).
\end{definition}

Observe that the weak (resp.\ essential) containment property can be reformulated as follows: every representation (resp.\ essential representation) of $\fullalg G$ factors through $\redalg G$ (resp.\ $\essalg G$). Generally, this property is difficult to verify. Rather than examining the (essential) weak containment property independently, we will show that the (essential) amenability of $G$ implies it -- see \cref{sec:ess-amenability} below. The converse in the non-essential case (\emph{i.e.}\ for \(\redalg{G}\)), however, remains less understood: it does not hold in general and only applies under certain exactness assumptions (see \cites{Willett:Non-amenable, Kranz:weak-containment}).

\subsection{(Essential) amenability of groupoids}\label{sec:ess-amenability}
In the non-essential case (\emph{e.g.}, when \(G\) is Hausdorff), a key technical tool for proving that a groupoid \(G\) satisfies the weak containment property is its amenability. We now introduce this notion in the more general, non-Hausdorff setting. 

\begin{definition} \label{def:ame}
    Let \(G\) be an \'etale groupoid with Hausdorff unit space \(X \subseteq G\).
    We say that \(G\) is \emph{Borel amenable} if there is a net \((\xi_i)_i \sbe \borelalg{G}\) of non-negative functions such that
    \begin{enumerate}[label=(\roman*)]
        \item \label{def:ame:norm} \(\sup_{x \in X} \sum_{g \in Gx} \xi_i\left(g\right)^2 \leq 1\) for all \(i\); and
        \item \label{def:ame:inva} for every compact subset \(K \subseteq G\) we have
            \[
                \sup_{g \in K} \abs{\left(\xi_i^* * \xi_i\right)\left(g\right) - 1} \to 0.
            \]
    \end{enumerate}
    We say that \(G\) is \emph{(topologically) amenable} if we can find \((\xi_i)_i \subseteq \auxalg{G}\) (cf.\ \cref{def:borel-algebras}) satisfying the above conditions \cref{def:ame:norm} and \cref{def:ame:inva}.
\end{definition}

\begin{remark}\label{rem:mistakes}
   We are grateful to Xin Li for pointing out an error in our original definition of amenability, where we used the space $\algalg G$ instead of the larger space $\auxalg G$. This is closely related the fact that \(\algalg G\) is \emph{not} closed under pointwise multiplication. Consequently, the definition above might \emph{not} align with the standard notion of (topological) amenability when \(G\) is not Hausdorff, as appears in \cite{renault-2013}.

       To address this, we use \(\auxalg{G}\) instead of the smaller \(\algalg{G}\). The necessity of this distinction is best illustrated in \cref{ex:trivial-action-ame}~\cref{ex:trivial-action-ame:weak} and \cref{prop:schur multiplier}. Notably, in the latter case, the images of the pointwise multiplication maps lie outside \(\algalg{G}\), reinforcing the need for the use of the larger space $\auxalg G$. 

    We use this remark to point out a similar mistake appearing in~\cite{renault-2013}.
    More precisely, \cite{renault-2013}*{Proposition 2.8} is wrong as stated: the groupoid from \cref{ex:trivial-action-ame}  admits a topological approximate invariant density in the sense of \cite{renault-2013}*{Proposition 2.7}, but it has no topological approximate invariant mean in the sense of \cite{renault-2013}*{Proposition 2.6}.
The issue lies in a continuity condition appearing in \cite{renault-2013}*{Proposition 2.6 (i)}, which is too strong to require in general for non-Hausdorff groupoids. It is not clear to us whether the main result of \cite{renault-2013} is affected by these issues. 
\end{remark}

It is also natural to expect that an \emph{amenability-like} condition would also play a role in determining whether a groupoid has the essential containment property. The following definition should therefore be compared with \cref{def:ame}.

\begin{definition} \label{def:ess-ame}
Let \(G\) be an \'etale groupoid with Hausdorff unit space \(X \subseteq G\) and dangerous arrows $D$. 
Suppose \(G\) can be covered by countably many open bisections.
We say that \(G\) is \emph{Borel essentially amenable} if there exists a net \((\xi_i)_i \subseteq \borelalg{G}\) of non-negative functions such that
\begin{enumerate}[label=(\roman*)]
    \item \label{def:ess-ame:norm} \(\sup_{x \in X \setminus D} \sum_{g \in Gx} \xi_i\left(g\right)^2 \leq 1\) for all \(i\); and
    \item \label{def:ess-ame:inva} for every compact subset \(K \subseteq G\),
        \[
            \sup_{g \in K \setminus D} \abs{\left(\xi_i^* * \xi_i\right)\left(g\right) - 1} \to 0.
        \]
\end{enumerate}
Likewise, \(G\) is \emph{essentially amenable} if we can find a net \((\xi_i)_i \subseteq \auxalg{G}\) (cf.\ \cref{def:borel-algebras}) satisfying the above conditions.
\end{definition}

\begin{remark} \label{rem:ess-ame}
    The following are worthwhile remarks about \cref{def:ame,def:ess-ame}.
    \begin{enumerate}[label=(\roman*)]
        \item \label{rem:ess-ame:norm} Item \cref{def:ess-ame:norm} may be dropped by an appropriate renormalization in both \cref{def:ame,def:ess-ame} (see, \emph{e.g.}\ \cite{Brown-Ozawa:Approximations}*{Definition 5.6.13}).
        \item \label{rem:ess-ame:complicated space} In the case when \(G\) is Hausdorff, \cref{def:ess-ame} reduces to \cref{def:ame} because in this case \(D\) is empty. In this case \cref{def:ess-ame}~\cref{def:ess-ame:inva} may be restated as some uniform convergence over compact sets.
        \item In the non-Hausdorff case, however, one should not talk about compact sets of \(K \setminus D\) as this might be a badly behaved topological space. For instance, it may fail to be locally compact (cf.\ \cref{rem:d not loc comp}).
        \item We believe that \cref{def:ess-ame} should, in fact, \emph{not} be considered the ‘correct’ definition of essential amenability if \( G \) cannot be covered by countably many open bisections, as in this case \( D \subseteq G \) may not be meager. In such cases, \cref{def:ess-ame} should be modified to accommodate any choice of meager set. However, we do \emph{not} pursue this modification here for two reasons. First, the added technical complexity does not justify the limited applicability to niche cases (see \cref{rem:covered by cnt bisec}). Second, in typical settings where \( G \) \emph{is} covered by countably many open bisections, both definitions should coincide.
        \item \label{rem:ess-ame:compact in bisection} The compact set \(K \subseteq G\) in both \cref{def:ame}~\cref{def:ame:inva} and \cref{def:ess-ame}~\cref{def:ess-ame:inva} may be assumed to be a compact subset of an open bisection (in particular it may be assumed to be Hausdorff). Indeed, any compact set $K$ can be covered by finitely many open bisections $s_1,\ldots, s_n$ (cf.\ \cref{lemma:compact set is covered by fin many bisections}), from which it follows that $K\sbe K_1\cup\ldots\cup K_n$ with $K_i\sbe s_i$ defined as the closure of $s_i\cap K$ within $s_i$. Each $K_i$ is compact and contained in the open bisection $s_i$. Thus, the supremums in \cref{def:ame}~\cref{def:ame:inva} and \cref{def:ess-ame}~\cref{def:ess-ame:inva} can be restated as maximums over (uniformly) finitely many supremums with compact sets contained in some open bisections.
        \item \cref{thm:wk-cont-nuc} proves that Borel amenability and (topological) amenability are, in fact, equivalent, and similarly for the essential case. Thus, for the rest of the section we will just work with amenability/essential amenability.
    \end{enumerate}
\end{remark}

We point out the following basic properties.
\begin{proposition} \label{prop:ess-ame-basics}
Let $G$ be an \'etale groupoid with Hausdorff unit space \(X \subseteq G\).
\begin{enumerate}[label=(\roman*)]
  \item \label{prop:ess-ame-basics:1} If \(G\) is amenable then it is also essentially amenable.
  \item \label{prop:ess-ame-basics:2} If \(G\) is amenable then \(xGx\) is amenable for every \(x \in X\).
  \item \label{prop:ess-ame-basics:3} If \(G\) is essentially amenable then \(xGx\) is amenable for every \(x \in X \setminus D\).
\end{enumerate}
\end{proposition}
\begin{proof}
    Assertion \cref{prop:ess-ame-basics:1} is clear. Indeed, \cref{def:ess-ame} is on the nose a weakening of \cref{def:ame}.
    Likewise, assertion \cref{prop:ess-ame-basics:2} is well known. Moreover, as, in fact, the proof of \cref{prop:ess-ame-basics:2,prop:ess-ame-basics:3} are very similar, we will be content with proving \cref{prop:ess-ame-basics:3}.

    Let \(x \in X \setminus D\) be a non-dangerous unit, and let \((\xi_i)_{i}\) be as in \cref{def:ess-ame}. In particular, observe that it follows from \cref{def:ess-ame}~\cref{def:ess-ame:inva} (and \cref{lemma:dangerous arrows}) that for every finite \(K \subseteq xG\)
    \[
        \sup_{g \in K} \abs{\left(\xi_i^* * \xi_i\right) \left(g\right) - 1} = \sup_{g \in K} \abs{\sum_{h \in G \rg(g)} \overline{\xi_i\left(h\right)} \xi_i\left(gh\right) - 1} \to 0 \;\; \text{(in } \, i \text{).}
    \]
    Given \(g, h \in Gx\) we write \(g \sim_\mathcal{R} h\) whenever \(\rg(g) = \rg(h)\).
    Clearly \(\sim_{\mathcal{R}}\) is an equivalence relation, so let \(\{g_y\}_{y \in Y} \subseteq Gx\) be a set of representatives of \(Gx/\sim_\mathcal{R}\) (note that \(Y\) is, simply, the orbit of \(x\) in the unit space \(X \subseteq G\)).
    Then the maps
    \[
        m_i \colon xGx \to \CCC, \;\; m_i\left(h\right) \coloneqq \sum_{y \in Y} \xi_i\left(g_yh\right).
    \]
    witness the amenability of \(xGx\), since the invariance conditions \cref{def:ess-ame}~\cref{def:ess-ame:inva} for \(\xi_i\) and \(m_i\) are precisely the same.
\end{proof}

The following remark basically states that essential amenability is a \emph{badly} behaved notion when it comes to quotients, and yet it is well behaved when it comes to ideals.
This is not surprising at all, since it is well known that essential crossed products are \emph{not} functorial (see \cite{KwasniewskiMeyer-essential-cross-2021}*{Remark 4.8}). Indeed, the problem derives from the obvious fact that closed sets may be meager, and thus any behaviour that ignores meager sets (such as essential amenability in the sense of \cref{def:ess-ame}) will ignore those meager closed sets as well.
On the contrary, open sets are never meager (unless they are empty) and thus behaviour in \(G\) \emph{is} inherited within any (non-empty) open set.
\begin{remark}
    \begin{enumerate}[label=(\roman*)]
        \item Essential amenability does \emph{not} pass to quotients or isotropy groups $xGx$ of dangerous points $x\in D$, in general. Indeed, let \(G \coloneqq X \rtimes S\) be as in \cref{ex:trivial-action} with \(\Gamma\) non-amenable. Then \(G\) is always essentially amenable (as noted in \cref{ex:trivial-action-ame} below), but \(x_0Gx_0 \cong \Gamma\) is \emph{not} amenable.
        \item On the other hand, whenever \(U \subseteq X\) is some \(G\)-invariant open subset and \(G\) is essentially amenable, then so is \(G_U \coloneqq UGU\). Indeed, let \((\xi_i)_{i}\) be as in \cref{def:ess-ame}. We may then consider \(\zeta_i \coloneqq \xi_i|_{G_U}\). It is routine to show that \((\zeta_i)_{i}\) witness the essential amenability of \(G_U\).
    \end{enumerate}
\end{remark}

Heuristically, \cref{def:ess-ame} says that \(G\) is amenable outside of the dangerous arrows.\footnote{\, This does \emph{not} mean saying that the restriction of \(G\) to \(X \setminus D\) is amenable. Indeed, \(X \setminus D\) may be a badly behaved topological space, cf.\ \cref{rem:ess-ame} \cref{rem:ess-ame:complicated space}.} This behaviour is best exemplified by the following.
\begin{example} \label{ex:trivial-action-ame}
    Let \(G \coloneqq X \rtimes S\) for $S\coloneqq\Gamma\cup \{0\}$ acting on $X$ as in \cref{ex:trivial-action}. We already notice in this example that $\algalg G$ consists of functions $G\to \C$ that are continuous on $O=X\backslash\{x_0\}$ and satisfying the limit condition~\eqref{eq:lim-condition-functions-C_c(G)}. On the other hand, the auxiliary space $\auxalg G$ is, indeed, larger and consists of all bounded functions $G\to \C$ that are continuous on $O$ and attain only finitely many nonzero values on $\Gamma$ (so that they are compactly supported). Indeed, multiplying (pointwise) two continuous functions supported on basic bisections of the form $O_\gamma:=O\sqcup \{\gamma\}\cong X$ for $\gamma\in \Gamma$, we get in $\auxalg G$ all continuous functions on $O$ that are zero on $\Gamma$. Taking linear combinations of those functions with continuous functions on $O_\gamma$, we get all compactly supported functions $G\to \C$ that are continuous on $O$. Regarding (essential) amenability of $G$, we note the following:   
    \begin{enumerate}[label=(\roman*)]
        \item \label{ex:trivial-action-ame:weak} \(G\) is amenable if and only if \(\Gamma\) is amenable (since \(O \subseteq X\) is a \emph{proper} open set).
        Indeed, if $\Gamma$ is amenable and \((F_n)_n\) denotes a F\o lner sequence for \(\Gamma\), then the functions
        \[
            \xi_n \colon G \to \C, \; g \mapsto \left\{\begin{array}{rl} 1/\abs{F_n}^{1/2} & \text{if }  g \in F_n, \\
            0 & \text{if } g \in \Gamma\backslash F_n, \\
            1 & \text{otherwise,} \end{array} \right. 
        \]
        can be seen to witness the amenability of \(G\). Notice that $\xi_n\sbe \auxalg G$, and also \(\xi_n^2 \in \algalg{G}\) (but $\xi_n\notin \algalg G$ in general), as
        \[
            \xi_n^2 = \sum_{\gamma \in F_n} 1/\abs{F_n} \cdot v_\gamma,
        \]
        which is contained in \(\algalg{G}\) by definition of the latter. It is not clear to us, in general, whether we can always find functions in $\auxalg G$ with squares in $\algalg G$ witnessing the amenability of an (amenable) groupoid $G$.
        But it definitely is not possible to get functions $\xi_n \in \algalg G$ witnessing the amenability of \(G\) for the groupoid above, so that in \cref{def:ame} we cannot replace $\auxalg G$ by $\algalg G$, in general. Indeed, take any net $(\xi_n)_n\sbe \algalg G$ satisfying the conditions of \cref{def:ame}. Define $\varphi_n:=\xi_n^**\xi_n\in \algalg G$. We have
        $$\varphi_n(x)=\sum_{g\in Gx}\xi_n(g)^2\quad \text{ for all } x\in X,$$
        and $x\mapsto \varphi_n(x)$ is a continuous function on $X$ (this is the continuity of the counting measures on $Gx$, which forms a Haar system for an \'etale groupoid).
        By \cref{def:ame}~\cref{def:ame:norm}, $\varphi_n(x)\to 1$ uniformly on $X$. On the other hand, 
        $$\varphi_n(x_0)=\sum_{\gamma\in \Gamma}\xi_n(g)^2$$
        and each $\xi_n(g)\to 1$. This is a contradiction unless $\Gamma$ is the trivial group.
        \item \label{ex:trivial-action-ame:ess} \(G\) is always essentially amenable (since \(O \subseteq X\) is dense), regardless of \(\Gamma\). Indeed, we may simply take \(\xi_i \coloneqq 1_X \in \algalg{G}\).
        Since \(D \cap X = X \setminus O\), it follows that \cref{def:ess-ame}~\cref{def:ess-ame:inva}, for instance, reduces to
        
        \[
            \sup_{g \in K \setminus D} \abs{\left(\xi_i^* * \xi_i\right)\left(g\right) - 1} \leq \sup_{g \in G O} \abs{1_X\left(g\right) - 1} = \sup_{x \in O = XO} \abs{1_X\left(x\right) - 1} = 0.
        \]
        Similar computations show that \cref{def:ess-ame}~\cref{def:ess-ame:norm} is also met.
    \end{enumerate}
\end{example}

We now point out the following example, see \cite{exel-starling:2017:amen-actions} for the motivation behind it.
\begin{example}
    Essential amenability does \emph{not} pass to preimages of \emph{\(d-\)bijective} groupoid homomorphisms \(\varphi \colon G \to H\) (cf.\ \cite{exel-starling:2017:amen-actions}*{Proposition 2.4}), as meager sets can have non-meager preimages. For instance, let \(\Gamma\) be some non-amenable group, and let \(S \coloneqq \Gamma \sqcup \{e_n\}_{n \in \N}\), where \(\{e_n\}_{n \in \N}\) are idempotents such that
    \begin{enumerate}[label=(\roman*)]
        \item \(e_n < e_{n+1} < e\) for all \(n \in \N\) (here \(e \in \Gamma\) is the identity element); and
        \item \(e_n \gamma = \gamma e_n = e_n\) for all \(\gamma \in \Gamma\) and \(n \in \N\).
    \end{enumerate}
    Let \(\alpha \colon S \curvearrowright [0,2]\) be the trivial action, where the domain of \(e_n\) is \([0,1)\) for all \(n \in \N\). Likewise, let \(\beta \colon S \curvearrowright [0,1]\) be also trivial, where the domain of \(e_n\) is \([0,1)\) as well.
        Then the map
        \[
            \varphi \colon \left[0,2\right] \rtimes_\alpha S \to \left[0,1\right] \rtimes_\beta S, \;\; \varphi\left(\left[s, x\right]\right) \coloneqq
                \left\{
                    \begin{array}{rl}
                        \left[s, x\right] & \text{if } x \in [0, 1], \\
                        \left[s, 1\right] & \text{otherwise,}
                    \end{array}
                \right. 
        \]
        is \(d\)-bijective in the sense of \cite{exel-starling:2017:amen-actions}*{Definition 2.2}, \([0,1] \rtimes_\beta S\) is essentially amenable (cf.\ \cref{ex:trivial-action-ame}), whereas \([0,2] \rtimes_\alpha S\) is not.

        As a consequence, contrary to the amenable case, cf.\ \cite{exel-starling:2017:amen-actions}*{Theorem 3.3}, it is \emph{not} true that Paterson's universal groupoid of an inverse semigroup \(S\) is essentially amenable if and only if \(X \rtimes S\) is essentially amenable for all actions \(S \curvearrowright X\), where \(X\) is some locally compact Hausdorff space. 
\end{example}

\begin{lemma} \label{prop:ame reduces hausdorff case}
    If \(G\) is a Hausdoff \'etale groupoid, then \(\auxalg{G} = \algalg{G}\).
    In particular, if \(G\) is Hausdorff then \cref{def:ame,def:ess-ame} reduce to the usual definition of amenability for Hausdorff groupoids.
\end{lemma}

%%%%%%%%%%%%%%%%%%%%%%%%%%%%%%%%%%%%%%%%%%%%%%%%%%%%%%%%%%%%%%%%%%%%%%%
% ESSENTIAL WEAK CONTAINMENT PROPERTY
%%%%%%%%%%%%%%%%%%%%%%%%%%%%%%%%%%%%%%%%%%%%%%%%%%%%%%%%%%%%%%%%%%%%%%%
\subsection{Essential weak containment property and nuclearity} \label{sec:wk-cont}
We are now ready to state the main theorems of the paper (cf.\ \cref{thm-intro}).
The following is the ``reduced'' part of \cref{thm-intro}, whereas \cref{thm:wk-cont-nuc-ess} is the ``essential'' one.
We divide them for the convenience of the reader. \cref{thm:wk-cont-nuc} particularly should be compared with \cite{brix-gonzalez-hume-li-2025}*{Corollary 6.10}.
\begin{theorem} \label{thm:wk-cont-nuc}
    Let \(G\) be an \'etale groupoid with locally compact Hausdorff unit space that can be covered by countably many open bisection.
    Consider the assertions.
    \begin{enumerate}[label=(\roman*)]
        \item \label{thm:wk-cont-nuc:n} \(\redalg{G}\) is nuclear.
        \item \label{thm:wk-cont-nuc:a} \(G\) is amenable.
        \item \label{thm:wk-cont-nuc:bn} \(\redborel G\) is nuclear.
        \item \label{thm:wk-cont-nuc:ba} \(G\) is Borel amenable.
        \item \label{thm:wk-cont-nuc:bwc} \(\Lambda_P^{\rm B} \colon \fullborel G \twoheadrightarrow \redborel G\) is injective.
        \item \label{thm:wk-cont-nuc:wc} \(\Lambda_P \colon \fullalg G \twoheadrightarrow \redalg G\) is injective.
    \end{enumerate}
    Then \cref{thm:wk-cont-nuc:n} \(\Leftrightarrow\) \cref{thm:wk-cont-nuc:a} \(\Leftrightarrow\) \cref{thm:wk-cont-nuc:bn} \(\Leftrightarrow\) \cref{thm:wk-cont-nuc:ba} \(\Rightarrow\) \cref{thm:wk-cont-nuc:bwc} \(\Rightarrow\) \cref{thm:wk-cont-nuc:wc}.
\end{theorem}

\begin{theorem} \label{thm:wk-cont-nuc-ess}
    Let \(G\) be an \'etale groupoid with locally compact Hausdorff unit space that can be covered by countably many open bisections.
    Consider the assertions.
    \begin{enumerate}[label=(\roman*)]
        \item \label{thm:wk-cont-nuc-ess:n} \(\essalg{G}\) is nuclear.
        \item \label{thm:wk-cont-nuc-ess:a} \(G\) is essentially amenable.
        \item \label{thm:wk-cont-nuc-ess:bn} \(\essborel G\) is nuclear.
        \item \label{thm:wk-cont-nuc-ess:ba} \(G\) is Borel essentially amenable.
        \item \label{thm:wk-cont-nuc-ess:bwc} \(\Lambda_{E_\L}^{\rm B} \colon \essmaxborel G \twoheadrightarrow \essborel G\) is injective.
        \item \label{thm:wk-cont-nuc-ess:wc} \(\ker(\essmaxalg G \twoheadrightarrow \essalg G) = \ker(\essmaxalg G \to \essmaxborel G)\).
    \end{enumerate}
    Then \cref{thm:wk-cont-nuc-ess:n} \(\Rightarrow\) \cref{thm:wk-cont-nuc-ess:a} \(\Rightarrow\) \cref{thm:wk-cont-nuc-ess:bn} \(\Leftrightarrow\) \cref{thm:wk-cont-nuc-ess:ba} \(\Rightarrow\) \cref{thm:wk-cont-nuc-ess:bwc} \(\Rightarrow\) \cref{thm:wk-cont-nuc-ess:wc}. Moreover, if the homomorphism
    $$\psi\colon \essmaxalg G \to \essmaxborel G$$
    is max-injective, then all conditions \cref{thm:wk-cont-nuc-ess:n}-\cref{thm:wk-cont-nuc-ess:ba} are equivalent.
\end{theorem}

\begin{remark}
    It remains unclear whether $\psi$ is always max-injective and whether essential amenability alone (see \cref{def:ess-ame}) is sufficient to guarantee the essential containment property \(\essmaxalg{G} = \essalg{G}\) and/or the nuclearity of one of these \cstar{}algebras. The underlying difficulty is similar to that discussed in \cref{remark:problem extending to essential}: there may exist essential representations \(\pi \colon \algalg{G} \to \B(H)\) that do not extend to representations of \(\borelessalg{G}\). As a consequence, \(\essmaxalg{G}\) may fail to embed into \(\essmaxborel{G}\).  
    In fact, even under the assumption of amenability it is not clear to us whether \(\essmaxalg{G} = \essalg{G}\) in general. Nonetheless, we can prove that it holds under the following additional condition:
    \begin{equation} \label{eq:extra condition-ess}
        \overline{\NN_{E_\L}^c}^{\|\cdot\|_\red} = \left\{ x \in \redalg{G} \mid E_\L(x^*x) = 0 \right\} = \NN_{E_\L}.
    \end{equation}
    Indeed, if \(G\) is amenable, then \(\|\cdot\|_\max = \|\cdot\|_\red\) on \(\algalg G\) by \cref{thm:wk-cont-nuc}. It follows that the ideal \(\NN_{E_\L}^\max = \overline{\NN_{E_\L}^c}^{\|\cdot\|_\max}\), as defined in \cref{prop:description-ess-max-algebra-quotient}, coincides with \(\overline{\NN_{E_\L}^c}^{\|\cdot\|_\red} = \NN_{E_\L}\), by \eqref{eq:extra condition-ess}. This implies that the quotient map
    \[
    \essmaxalg G = \fullalg G / \NN_{E_\L}^\max \longrightarrow \redalg G / \NN_{E_\L} = \essalg G
    \]
    is an isomorphism.
    It remains an open question whether there exist groupoids that do \emph{not} satisfy condition \eqref{eq:extra condition-ess} (see \cite{brix-gonzalez-hume-li-2025}*{Questions 4.11}, for instance).
\end{remark}

Generally speaking, the proofs of \cref{thm:wk-cont-nuc,thm:wk-cont-nuc-ess} follow the same strategy, but the latter is technically more involved.
In particular, we will sometimes be sketchy with the proofs of the former theorem, but explicit with the proofs of the latter one.
Moreover, the proofs are rather lengthy and technical, and the ``forward'' implications\footnote{\, Namely the fact that Borel amenability (resp.\ Borel essential amenability) implies \(\redalg{G}\) (resp.\ \(\essalg{G}\)) is nuclear.} hinge on the use of specific \emph{Herz-Schur multipliers}, constructed in the following proposition (recall the notation from \cref{sec:aux-alg}).

\begin{proposition} \label{prop:schur multiplier}
    Let \(\xi \in \borelalg{G}\), and \(\varphi \coloneqq \xi^* * \xi \in \borelalg G\) with $\|\varphi\|_\infty\leq 1$. Then the pointwise multiplication map
    \[
        m_\varphi \colon \borelalg{G} \to \borelalg{G}, \;\; m_\varphi\left(a\right)\left(g\right) \coloneqq \varphi\left(g\right) \cdot a \left(g\right)
    \]
    induces c.c.p. maps
    \begin{enumerate}[label=(\roman*)]
        \item \label{prop:schur multiplier:max-p} \(m_{\varphi, \rm max}^{P} \colon \fullborel{G} \to \fullborel{G}\);
        \item \label{prop:schur multiplier:red-p} \(m_{\varphi, \rm red}^{P} \colon \redborel{G} \to \redborel{G}\);
        \item \label{prop:schur multiplier:max-e} \(m_{\varphi, \rm max}^{E_\L} \colon \essmaxborel{G} \to \essmaxborel{G}\); and
        \item \label{prop:schur multiplier:red-e} \(m_{\varphi, \rm red}^{E_\L} \colon \essborel{G} \to \essborel{G}\).
    \end{enumerate}
\end{proposition}
\begin{proof}
    We start by noting that \(m_\varphi\) does define a map from \(\borelalg{G}\) to \(\borelalg{G}\).
    Indeed, this is clear as $\borelalg G$ is closed under pointwise multiplication and $\varphi\in \borelalg G$.
    
    The ``reduced'' cases \cref{prop:schur multiplier:red-p,prop:schur multiplier:red-e} are proved both in a similar way, the only difference being in which set one takes the supremum (\(X\) for \cref{prop:schur multiplier:red-p} and \(X \setminus D\) for \cref{prop:schur multiplier:red-e}, cf.\ \cref{prop:rep red and ess calgs:ess}). In fact, given some fixed \(x \in X\), the map
    \begin{align*}
        m_\varphi^x \colon \B\left(\ell^2\left(Gx\right)\right) & \to \B\left(\ell^2\left(Gx\right)\right), \;\; m_\varphi^x\left(\left[t_{g, h}\right]_{g, h \in Gx}\right) & \coloneqq \left[\varphi\left(gh^{-1}\right) \cdot t_{g, h}\right]_{g, h \in Gx}
    \end{align*}
    is contractive and completely positive by \cite{Brown-Ozawa:Approximations}*{Theorem D.3}. Moreover, the following holds.
    \begin{claim} \label{claim:schur multiplier behaves with lambdax}
        \(\lambda_x(m_\varphi(a)) = m_\varphi^x (\lambda_x(a))\) for every \(a \in \borelalg{G}\) and \(x \in X\).
    \end{claim}
    \begin{proof}
        Let \(h \in Gx\). Observe that
        \begin{align*}
            \lambda_x\left(m_\varphi\left(a\right)\right) \delta_h & = \sum_{g \in G \rg(h)} \left(m_\varphi \left(a\right)\right)\left(g\right) \delta_{gh} = \sum_{g \in G \rg(h)} \left(\xi^* * \xi\right) \left(g\right) a\left(g\right) \delta_{gh}.
        \end{align*}
        Likewise,
        \begin{align*}
            m_\varphi^x\left(\lambda_x\left(a\right)\right) \delta_h & = \sum_{g_1 \in Gx} \varphi\left(g_1 h^{-1}\right) \langle \lambda_x\left(a\right) \delta_h, \delta_{g_1} \rangle \delta_{g_1} \\
            & = \sum_{g_1 \in Gx} \left(\xi^* * \xi\right)\left(g_1 h^{-1}\right) a\left(g_1h^{-1}\right) \delta_{g_1}.
        \end{align*}
        By making a change of variables, where \(g_1h^{-1} = g\), the claim follows.
    \end{proof}

    \begin{claim} \label{claim:multiplier descends}
        \(m_\varphi\) descends to a map \(m_{\varphi}^{E_\L} \colon \borelessalg{G} \to \borelessalg G\).
    \end{claim}
    \begin{proof}
        Simply observe that \(\supp(m_\varphi(a)) \subseteq \supp(a)\). Thus, should the latter be contained in \(D\) then so is the former.
     \end{proof}

    Now statements \cref{prop:schur multiplier:red-p,prop:schur multiplier:red-e} follow. Indeed, by construction
    \(\oplus_{x \in X} \lambda_x\) is a faithful representation of \(\redborel{G}\), proving assertion \cref{prop:schur multiplier:red-p} in the statement when combined with \cref{claim:schur multiplier behaves with lambdax}. Likewise, assertion \cref{prop:schur multiplier:red-e} follows since \(\oplus_{x \in X \setminus D} \lambda_x\) descends to a faithful representation of \(\borelessalg{G}\).

    \

    We now prove assertion \cref{prop:schur multiplier:max-e}, as \cref{prop:schur multiplier:max-p} is an easier version of this. We follow a similar argument as the one in \cite{Brown-Ozawa:Approximations}*{Proposition 5.1.16} but adapted to the essential non-Hausdorff setting.
    More precisely, observe that \cref{claim:multiplier descends} yields that the diagram
    \begin{center}
        \begin{tikzcd}[scale=50em,row sep=large, column sep=large]
            \Mat_n\left(\borelalg{G}\right) \arrow{r}{\id_n \otimes m_\varphi} \arrow[two heads]{d}{\id_n \otimes \Lambda^{\rm ess}_{\rm B}} & \Mat_n\left(\borelalg{G}\right) \arrow{d}{\id_n \otimes \Lambda^{\rm ess}_{\rm B}} \\
            \Mat_n\left(\borelessalg{G}\right) \arrow{r}{\id_n \otimes m_\varphi^{E_\L}} & \Mat_n \left(\borelessalg{G}\right)
        \end{tikzcd}
    \end{center}
    commutes for all \(n \in \N\), where \(\Lambda^{\rm ess}_{\rm B} \colon \borelalg{G} \twoheadrightarrow \borelessalg{G}\) is the canonical quotient map.\footnote{Recall that \(\Lambda_{\rm B}^\ess\) is the quotient map killing \(I_D\), cf.\ \cref{def:borel-algebras}, which might be strictly smaller than all Borel functions whose supports are contained in \(D\).}

    The goal now is to prove that this map \(m_\varphi\) is \emph{completely algebraically positive}, in the sense that its amplification to matrices \(\Mat_n(\borelalg G)\to \Mat_n(\borelalg{G})\) preserves \emph{algebraically positive elements} for all \(n \in \N\). Then it will also follow that $m_\varphi^{E_\L}$ is completely algebraically positive.
    Here, an \emph{algebraically positive element} is a matrix that is a finite sum of matrices of the form \((a_i^*a_j)_{i,j = 1}^n \in \Mat_n(A)\), where \((a_i)_{i = 1}^n \in A^n\) and $A$ is some \Star{}algebra.
    
    \begin{claim} \label{claim:multiplier comp pos}
        Both \(m_\varphi \colon \borelalg{G} \to \borelalg G\) and \(m_\varphi^{E_\L} \colon \borelessalg{G} \to \borelessalg G\) are completely algebraically positive.
    \end{claim}
    \begin{proof}
        We will first prove the claim for
        \[
            m_\varphi \colon \borelalg{G} \to \borelalg{G},
        \]
        and then apply \cref{claim:multiplier descends} in order to prove the analogous statement for \(m_\varphi^{E_\L}\).

        Since \(\xi \in \borelalg{G}\) there is some decomposition \(\xi = \rho_1 v_{s_1} + \dots + \rho_p v_{s_p}\), where \(s_1, \dots, s_p\) are open bisections in \(G\) and \(\rho_k\) is a Borel compactly supported function \(\rho_k \colon s_k s_k^* \to \C\). 
        In particular, let \(C_k \coloneqq \overline{\supp}(\rho_k)\), which is a compact set in \(s_k s_k^*\).
        We now iteratively construct projections \(p_1, \dots, p_q\) as follows.
        \begin{itemize}
            \item Let \(p_1 \colon s_1 \to \C\) be the characteristic function of \(C_1 s_1\), which is a compact subset of \(s_1\). In particular, \(p_1\) is Borel and compactly supported (as a function from \(s_1\)), so \(p_1 \in \borelalg{G}\).
            \item Given \(p_1, \dots, p_k\), let \(p_{k+1} \colon s_{k+1} \to \C\) be the characteristic function of \(C_{k+1} s_{k+1} \setminus (C_1 s_1 \cup \dots \cup C_k s_k)\). Observe that \(p_{k+1} \in \borelalg{G}\).
        \end{itemize}
        Consider the decomposition
        \[
            \xi\left(g\right) = \sum_{k = 1}^q \xi_k\left(\s(g)\right),
        \]
        where
        \[
            \xi_k\left(\s(g)\right) \coloneqq \xi\left(g\right) \cdot p_k\left(g\right) \text{ for all } k = 1, \dots, q.
        \]
        Notice that \(\xi_k\) is a well-defined element of $\borelalg G$ since \(p_k\) is supported only on the Borel bisection \(s_k\). 
        Furthermore,
        \[
            \varphi\left(g\right) = \left(\xi^* * \xi\right)\left(g\right) = \sum_{h \in G \rg(g)} \overline{\xi\left(h\right)} \xi\left(hg\right) = \sum_{k, l = 1}^q \overline{\xi_{k}\left(\rg(g)\right)} \xi_{l}\left(\s(g)\right)
        \]
        for all \(g \in G\).
        It then follows that for all \((a_i)_{i = 1}^n \in (\algalg{G})^n\) and \(g \in G\),
        \begin{align}
            m_{\varphi}\left(a^*_i * a_j\right)\left(g\right) & = \varphi\left(g\right) \cdot \left(a^*_i * a_j\right)\left(g\right) = \left(\xi^* * \xi\right)\left(g\right) \sum_{h \in G\rg(g)} \overline{a_i\left(h\right)} a_j\left(hg\right) \nonumber \\
            & = \sum_{k, l = 1}^q \sum_{h \in G\rg(g)} \overline{a_i\left(h\right) \xi_k\left(\rg(g)\right)} a_j\left(hg\right) \xi_{l}\left(\s(g)\right) \nonumber \\
            & = \sum_{k, l = 1}^q \sum_{h \in G\rg(g)} \overline{a_{i,k}\left(h\right)} a_{j,l}\left(hg\right) \nonumber \\
            & = \left(\left(\sum_{k = 1}^q a_{i,k}^* \right) * \left(\sum_{k = 1}^q a_{j,k}\right)\right)\left(g\right), \label{claim:multiplier comp pos:equation}
        \end{align}
        where \(a_{i,k}(g) \coloneqq a_i(g) \cdot \xi_k(\s(g))\), which is a Borel bounded function compactly supported on the bisection \(s_k \subseteq G\). In addition, note that for all \(i, j = 1, \dots, n\),
        \begin{align*}
            \left(\left(\id_n \otimes m_\varphi\right)\left(\left(a_k^* * a_\ell\right)_{k, \ell = 1}^n\right)\right)_{i, j} = m_{\varphi}\left(a^*_i * a_j\right).
        \end{align*}
        In particular, it follows that \(\id_n \otimes m_\varphi\) maps the algebraically positive element \((a^*_i*a_j)_{i, j = 1}^n\) to the algebraically positive element in \(\Mat_n(\borelalg G)\) whose \((i,j)\)-entry is
        \[
            \left(\sum_{k = 1}^q a_{i,k}^* \right) * \left(\sum_{k = 1}^q a_{j,k}\right) \in \borelalg G,
        \]
        as desired.
        By the considerations before the statement of the claim, we see that the map \(m_\varphi\) induces maps making
        \begin{center}
            \begin{tikzcd}[scale=50em,row sep=large, column sep=large]
                \Mat_n\left(\borelalg{G}\right) \arrow{r}{\id_n \otimes m_\varphi} \arrow[two heads]{d}{\id_n \otimes \Lambda^{\rm ess}_{\rm B}} & \Mat_n\left(\borelalg{G}\right) \arrow[two heads]{d}{\id_n \otimes \Lambda^{\rm ess}_{\rm B}} \\
                \Mat_n\left(\borelessalg{G}\right) \arrow{r}{\id_n \otimes m_\varphi^{E_\L}} & \Mat_n\left(\borelessalg{G}\right)
            \end{tikzcd}
        \end{center}
        a commutative diagram. In particular, for all \((a_i)_{i = 1}^n \in (\borelalg{G})^n\),
        \[
            \left(\id_n \otimes m_\varphi^{E_\L}\right)\left(e_{i,j} \otimes \Lambda^{\rm ess}\left(a_i^* * a_j\right)\right) = e_{i,j} \otimes \Lambda_{\rm B}^{\rm ess}\left(m_{\varphi}\left(a_i^**a_j\right)\right),
        \]
        where \(e_{i,j} \in \Mat_n\) denotes the canonical matrix unit. Thus, given some algebraically positive \((a_i^* * a_j)_{i, j = 1}^n \in \Mat_n(\borelessalg G)\), we may lift it to some algebraically positive \(b \in \Mat_n(\borelalg G)\). In such case, by the argument above the element \((\id_n \otimes m_\varphi)(b)\) coincides with some algebraically positive element.
        Thus, taking the quotient map \(\id_n \otimes \Lambda^{\rm ess}\), it follows that
        \[
            \left(\id_n \otimes m_\varphi^{E_\L}\right)\left(\left(a_i^* * a_j\right)_{i, j = 1}^n\right) = \id_n \otimes \Lambda^{\rm ess} \left(\id_n \otimes m_\varphi\left(b\right)\right)
        \]
        is algebraically positive, as desired.
    \end{proof}

    Now the proof of Stinespring's Dilation Theorem yields the necessary boundedness of \(m_\varphi^{E_\L}\). Indeed, let us fix a faithful representation \(\pi\colon \essmaxborel G\into \B(\Hilb)\) on a Hilbert space \(\Hilb\), and consider the map 
    \(\pi\circ m_\varphi^{E_\L}\colon \borelessalg G\to \B(\Hilb)\). We then define a sesquilinear form on $\borelessalg G\odot \Hilb$ by
    \[
        \braket{\sum_ia_i\otimes\xi_i}{\sum_j b_j\otimes\eta_j} \coloneqq \sum_{i,j}\braket{\xi_i}{\pi\circ m_\varphi^{E_\L}\left(a_i^**b_j\right)\eta_j}.
    \]
    The complete algebraic positivity of \(m_\varphi^{E_\L}\) (cf.\ \cref{claim:multiplier comp pos}) implies that \(\braket{\cdot}{\cdot}\) is a well-defined semi-inner product. We sketch this (rather standard) computation:
    \begin{align*}
        \braket{\sum_ia_i\otimes\xi_i}{\sum_j a_j\otimes\xi_j} & = \sum_{i,j}\braket{\xi_i}{\pi\circ m_\varphi^{E_\L}\left(a_i^**a_j\right)\xi_j} \\
        & \stackrel{\eqref{claim:multiplier comp pos:equation}}{=} \sum_{i,j} \braket{\xi_i}{\pi\left(\left(\sum_{k = 1}^q a_{i,k}^* \right) * \left(\sum_{k = 1}^q a_{j,k}\right)\right) \xi_j}\\
        &\,\,\,\,=\braket{\sum_{i,k}\pi\left(a_{i,k}\right)\xi_i}{\sum_{j,k} \pi\left(a_{j,k}\right) \xi_j}\geq 0.
    \end{align*}
    Taking the Hausdorff completion of \(\borelessalg{G} \odot \Hilb\) with respect to this semi-inner product, we get a Hilbert space \(\hat\Hilb\).
    Now, given \(a\in \borelessalg G\), we define a linear operator \(\rho(a)\) on \(\borelessalg G\odot \Hilb\) by
    \[
        \rho\left(a\right)\Big(\sum_j b_j\otimes\xi_j\Big) \coloneqq \sum_j ab_j\otimes \xi_j.
    \]
    In order to see that \(\rho(a)\) is well defined and extends to a bounded operator on \(\hat\Hilb\), we use the estimate
    \[
        \left(b_i^*a^*ab_j\right)_{i, j = 1}^n \leq \norm{a}^2 \left(b_i^*b_j\right)_{i, j = 1}^n,
    \]
    that holds in \(\Mat_n(\essmaxborel G)\) (as the latter is a \cstar algebra). In order to ensure that the above estimate is preserved by applying \(m_\varphi^{E_\L}\), we need to use the complete algebraic positivity of the latter, and for this we need to know that the element
    \begin{equation}\label{eq:alg-positive-element}
        \norm{a}^2\left(b_i^*b_j\right)_{i, j = 1}^n - \left(b_i^*a^*ab_j\right)_{i, j = 1}^n
    \end{equation}
    is algebraically positive for all \(a\in \borelessalg G\) and \((b_i)_{i = 1}^n \in (\borelessalg G)^n\).
    This is easy to see if \(a \in \borelalg X\) and \(\norm{a} = \norm{a}_\infty = 1\),
    because in that case we have
    \[
        \left(b_i^*b_j - b_ia^*ab_j\right)_{i, j = 1}^n = \left(c_i^*c_j\right)_{i, j = 1}^n,
    \]
    where \(c_i(g) \coloneqq (1-|a(\rg(g))|^2)^{1/2}b_i(g)\).
    Now, if \(a\in\borelessalg G\) is the image of some function \(\tilde{a} \in \borelalg G\) supported on a single bisection, then \(a^*a\in \borelessalg X \subseteq \borelessalg{G}\).
    It then follows from the previous argument (applying it to \(a^*a/\norm{a}^2\)) that \eqref{eq:alg-positive-element} is algebraically positive.
    Therefore \(\rho(a)\) extends to a bounded operator on \(\hat\Hilb\) (with norm \(\norm{\rho(a)} \leq \norm{a}\)) whenever \(a\) is the image of some \(\tilde{a} \in \borelalg G\) supported on a bisection.
    But since every element of \(\borelessalg G\) is a finite sum of such functions, it follows that \(\rho(a)\) extends to a bounded operator for every \(a\in \borelessalg G\), as desired. Straightforward computations show that $\rho$ is a representation.
    
    Finally, take \(e\in \borelalg X\) with \(0\leq e\leq 1\) and \(e(\s(g))=e(\rg(g))=1\) for all \(g \in \supp(\varphi)\). For instance, we can take $e$ to be the characteristic function of the compact subset $\s(K)\cup \rg(K)\sbe X$, where $K\sbe G$ is some compact subset containing $\supp(\varphi)$. Then
    \begin{equation} \label{eq:diego needs this equation}
        m_\varphi\left(e^*\tilde{a}e\right) = m_\varphi\left(\tilde{a}\right)
    \end{equation}
    for all \(\tilde{a} \in \borelalg G\) (note that here \(\tilde{a} \in \borelalg{G}\), as opposed to \(\borelessalg G\)).
    In particular, this does imply that the analogue for \(m_\varphi^{E_\L}\) holds, that is,
    \[
        m_\varphi^{E_{\L}}\left(e^*ae\right) = m_\varphi^{E_\L}\left(a\right)
    \]
    for all \(a \in \borelessalg{G}\). Thus
    \[
        m_\varphi^{E_\L}\left(a\right) = V^*\rho\left(a\right) V,
    \]
    where \(V\colon \Hilb\to \hat\Hilb\), \(V(\xi) \coloneqq e\otimes\xi\).
    This is a bounded operation with $\|V\|\leq\|e\|\leq 1$. Since \(\rho\) is a representation, it follows that \(m_\varphi^{E_\L}\) extends to a contractive completely positive linear map on \(\essmaxborel{G}\), as desired. 

    \

    Lastly, in order to finish the proof of \cref{prop:schur multiplier} we need to show \cref{prop:schur multiplier:max-p}, but this is an easier version of \cref{prop:schur multiplier:max-e}, so we only sketch the proof.
    \cref{claim:multiplier comp pos} shows that the map \(m_\varphi\) is also completely algebraically positive.
    Following Stinespring's Dilation procedure for \(\borelalg{G} \odot \Hilb\), as opposed to \(\borelessalg{G} \odot \Hilb\), yields a Hilbert space \(\tilde \Hilb\), which we may now use to represent \(\borelalg{G}\) by left multiplication, the same element \(e\) appearing in \eqref{eq:diego needs this equation} then yields that \(m_\varphi\) is bounded (as a map \(m_\varphi \colon \borelalg{G} \to \fullborel{G}\)). This implies that it extends to \(\fullborel{G}\), as claimed.
\end{proof}

\begin{remark} \label{rem:schur multipliers}
The following observations are noteworthy:
\begin{enumerate}[label=(\roman*)]
    \item \label{rem:schur multipliers:hausdorff} If \(G\) is Hausdorff, then \(m_{\varphi, \mathrm{red}}^P = m_{\varphi, \mathrm{red}}^{E_\mathcal{L}}\) and \(m_{\varphi, \mathrm{max}}^P = m_{\varphi, \mathrm{max}}^{E_\mathcal{L}}\).
    \item Even when \(a, \varphi \in \algalg G\), it is \emph{not} true that \(m_\varphi(a) \in \algalg G\) (see \cref{ex:trivial-action-ame}).
\end{enumerate}
\end{remark}

The following few results are the last pieces of the puzzle needed before the proof of \cref{thm:wk-cont-nuc,thm:wk-cont-nuc-ess}. The first three are obvious in the Hausdorff case, since every function \(a \in \algalg{G}\) is automatically continuous and compactly supported.
In the general setting these results deserve a computation, since, in fact, the following is why we need to consider the larger algebra $\auxalg{G}$ and the nets \((\xi_i)_i\) witnessing (essential) amenability in \cref{def:ame,def:ess-ame} to be in $\auxalg{G}$.

\begin{lemma} \label{lemma:taking square roots}
    Let \(\zeta \colon G \to \C\) be pointwise positive and such that \(\zeta^2 \in \auxalg{G}\). Then, for all \(\mu > 0\), there is some \(\zeta_\mu \in \auxalg{G}\) such that \(\norm{\zeta_\mu - \zeta}_{\infty} \leq \mu\).
\end{lemma}
\begin{proof}
    Observe that \(\zeta\) is supported on some compact set \(K_0 \subseteq G\). Since \(\auxalg{G}\) is a commutative algebra when equipped with the pointwise multiplication and
    \[
        \auxalg{G} = \bigcup_{K \subseteq G} \mathfrak{A}_K(G),\footnote{\, Here the notation \(\mathfrak{A}_K(G)\) means the elements in \(\auxalg{G}\) whose support is contained in the compact set \(K\).}
    \]
    it follows that the completion of \(\mathfrak{A}_{K_0}(G)\) in the supremum norm \(\norm{\cdot}_{\infty}\) is a commutative \cstar{}subalgebra of \(B_{K_0}(G)\). In particular, the claim follows simply by recalling that we may take square roots in this completion, and then approximating by elements \(\zeta_\mu \in \auxalg G\) where \(\mu > 0\).
\end{proof}

\begin{lemma} \label{lemma:taking norm is continuous}
    Let \(\Hilb\) be a finite dimensional Hilbert space; \(\eta \in \algalg{G, \Hilb} \coloneqq \algalg{G} \otimes \Hilb\) and \(\mu > 0\) be given. Let \(\zeta(g) \coloneqq \norm{\eta(g)}_{\Hilb}\). Then there is some \(\zeta_\mu \in \auxalg{G}\) such that \(\norm{\zeta-\zeta_\mu}_{\infty} \leq \mu\).
\end{lemma}
\begin{proof}
    By assumption, and following \cref{notation:cont functions as sums}, \(\eta = a_1 v_{s_1} + \dots + a_k v_{s_k}\), where \(s_1, \dots, s_k \subseteq G\) are open bisections and \(a_i \in \contc(s_is_i^*, \Hilb)\). Fix some orthonormal basis \((e_n)_{n = 1}^m \subseteq \Hilb\). We have that
    \begin{align*}
        \norm{\eta\left(g\right)}^2 & = \langle \sum_{g \in s_i} a_i\left(\rg(g)\right), \sum_{g \in s_j} a_j\left(\rg(g)\right)\rangle_{\Hilb} = \sum_{g \in s_i \cap s_j} \langle a_i\left(\rg(g)\right), a_j\left(\rg(g)\right)\rangle_{\Hilb} \\
        & = \sum_{n = 1}^m \sum_{g \in s_i \cap s_j} \overline{a_{i,n}(\rg(g))} \cdot a_{j,n}(\rg(g))
    \end{align*}
    Observe that the function \(s_i \cap s_j \ni g \mapsto \overline{a_{i,n}(\rg(g))} \cdot a_{j,n}(\rg(g))\) is pointwise product of continuous functions in \(\algalg G\), and thus is contained in \(\auxalg{G}\).
    The claim then follows from \cref{lemma:taking square roots}.
\end{proof}

\begin{lemma} \label{lemma:taking norm is borel}
    Let \(\Hilb\) be a finite dimensional Hilbert space; and \(\eta \in \borelalg{G, \Hilb} \coloneqq \algalg{G} \otimes \Hilb\). Let \(\zeta(g) \coloneqq \norm{\eta(g)}_{\Hilb}\). Then \(\zeta \in \borelalg{G}\).
\end{lemma}

\begin{proof}
    The proof simply follows by observing that the composition of Borel functions is Borel, and that taking the norm is a Borel operation.
\end{proof}

\begin{remark} \label{rem:evaluate-eta-at-non-dang-units}
    In the same vein as in \cref{lemma:taking norm is continuous}, we may work with elements \(\eta \in \essalgalg{G, \Hilb} \coloneqq \essalgalg{G} \otimes \Hilb\). Indeed, in such case we let
    \[
        \langle h_1 \otimes a_1, h_2 \otimes a_2\rangle \coloneqq \langle h_1, h_2\rangle_\Hilb \, a_1^* a_2 \in \essalgalg{G}
    \]
    for all \(h_1, h_2 \in \Hilb\) and \(a_1, a_2 \in \essalgalg{G}\). In particular, by \cref{cor:j-ess} we may view elements of $\essalg G$ as functions on $G\backslash D$ and therefore evaluate the above expression at non-dangerous units \(x \in X \setminus D\), \emph{i.e.}\
    \begin{equation}
        \langle h \otimes a, h \otimes a\rangle \left(x\right) \coloneqq \langle h, h\rangle_\Hilb \, \sum_{g \in Gx} \overline{a\left(g\right)} a\left(g\right) \label{eq:evaluate-eta-at-non-dang-units}
    \end{equation}
    for all \(a \in \essalgalg{G}\).
    Likewise, the above evaluation extends to a well-defined expression \(\langle \eta, \eta \rangle(x)\) for all \(\eta \in \Hilb \otimes \essalgalg{G}\) and non-dangerous \(x \in X \setminus D\). 
    In the non-essential case, as expected, the expression \eqref{eq:evaluate-eta-at-non-dang-units} makes sense for all units \(x \in X\) as, in that case, we do not quotient by anything.
\end{remark}

\begin{lemma} \label{lemma:image of schur multipliers}
    Let \(\varphi = \xi^* * \xi\) be as in \cref{prop:schur multiplier}. 
    If $\supp(\varphi)\sbe K$ for some compact subset $K\sbe G$, then
    the images of \(m_{\varphi, \rm max}^P\) and \(m_{\varphi, \rm red}^P\) are contained in the closed subspace $B_K(G)\sbe B_c(G)$. 
    Similarly, the images of \(m_{\varphi, \rm max}^{E_\L}\) and \(m_{\varphi, \rm red}^{E_\L}\) are contained in $B_K^\ess(G)\sbe \borelessalg G$.
\end{lemma}
\begin{proof}
This follows directly from the fact that $B_K(G)$ (resp. $B_K^\ess(G)$) is a closed subspace of both $\fullalg G$ and $\redalg G$ (resp. $\essmaxalg{G}$ and $\essalg G$), see \cref{cor:norms-equivalent-B_K(G)}.
\end{proof}

The following is surely well-known to experts.
\begin{lemma} \label{lemma:max-inj-preserves-nuc}
    Let \(B \subseteq A\) be a max-injective inclusion of \cstar{}algebras. If \(A\) is nuclear, then so is \(B\).
\end{lemma}
\begin{proof}
    Let \(C\) be a \cstar{}algebra. Observe that max-injectivity implies that the map \(B \otimes_\max C \to A \otimes_\max C\) is injective. Since the diagram
    \begin{center}
        \begin{tikzcd}[scale=50em]
            B \otimes_\max C \arrow[hook]{r}{} \arrow[two heads]{d}{} & A \otimes_\max C \arrow[equal]{d}{} \\
            B \otimes_{\rm min} C \arrow[hook]{r}{} & A \otimes_{\rm min} C
        \end{tikzcd}
    \end{center}
    commutes, the claim follows.
\end{proof}

We are ready to prove \cref{thm:wk-cont-nuc,thm:wk-cont-nuc-ess}.
\begin{proof}[Proofs of \cref{thm:wk-cont-nuc,thm:wk-cont-nuc-ess}]
    As mentioned, we will be explicit with the proof of \cref{thm:wk-cont-nuc-ess} but both proofs are (both formally and heuristically) very similar. Some of the arguments for some of the implications are different, and we will address these differences should they occur. We refer the reader to \cite{Brown-Ozawa:Approximations}*{Theorem 5.6.18} for a proof of the (non-Borel part of) \cref{thm:wk-cont-nuc} in the Hausdorff setting.

    \

    \cref{thm:wk-cont-nuc-ess:a} \(\Rightarrow\) \cref{thm:wk-cont-nuc-ess:ba} is clear.

    \cref{thm:wk-cont-nuc-ess:ba} \(\Rightarrow\) \cref{thm:wk-cont-nuc-ess:bwc}. Let \((\xi_i)_{i} \subseteq \borelalg{G}\) witness the Borel (resp.\ essential) amenability of \(G\) (cf.\ \cref{def:ess-ame}), and consider the Herz-Schur multipliers
    \begin{enumerate}[label=(\roman*)]
        \item \(m_{\varphi_i, \rm max}^{E_\L} \colon \essmaxborel{G} \to \essmaxborel{G}\);
        \item \(m_{\varphi_i, \rm red}^{E_\L} \colon \essborel{G} \to \essborel{G}\);
        \item \(m_{\varphi_i, \rm max}^{P} \colon \fullborel{G} \to \fullborel{G}\); and
        \item \(m_{\varphi_i, \rm red}^{P} \colon \redborel{G} \to \redborel{G}\),
    \end{enumerate}
    where \(\varphi_i \coloneqq \xi_i^* * \xi_i\), given in \cref{prop:schur multiplier}.
    Suppose \(a \in \borelalg G\) is a function supported on a (not necessarily open) bisection \(s \subseteq G\). Then \(m_{\varphi_i}(a)=\varphi_i\cdot a\) is also a function supported on (the same) \(s\).
    Therefore, by \cref{lem:ess-norm-supp-bisection},
    \begin{align}
        \norm{m_{\varphi_i}(a) - a}_\ess =
      \sup_{g \in s \setminus D} \abs{\varphi_i(g)a(g)- a(g)} \leq \norm{a}_\infty \cdot \sup_{g \in s \setminus D} \abs{\varphi_i\left(g\right) - 1} \to 0, \label{eq:thm:wk-cont-nuc: tends to 0}
    \end{align}    
    and, by definition of the reduced norm,
    \begin{align}
        \norm{m_{\varphi_i}(a) - a}_\red =
      \sup_{g \in s} \abs{\varphi_i(g)a(g)- a(g)} \leq \norm{a}_\infty \cdot \sup_{g \in s} \abs{\varphi_i\left(g\right) - 1} \to 0, \label{eq:thm:wk-cont-nuc: tends to 0-red}
    \end{align}
    Since every element in \(\borelalg{G}\) (resp.\ \(\borelessalg{G}\)) is a finite sum of functions (cf.\ \cref{lem:partition-into-borel-bisections}) supported on bisections, the same estimate \eqref{eq:thm:wk-cont-nuc: tends to 0} (resp.\ \eqref{eq:thm:wk-cont-nuc: tends to 0-red}) holds for all elements in \(\borelalg{G}\) (resp.\ \(\borelessalg{G}\)). Therefore, by an approximation argument, the same holds for all \(a \in \fullborel{G}\) (resp.\ \(\essmaxborel{G}\)).
    Furthermore, \cref{lemma:image of schur multipliers} implies that the image of the multiplier maps are contained in \(\borelalg G \subseteq \fullborel{G}\) (resp.\ \(\borelessalg{G} \subseteq \essmaxborel{G}\)).
    Thus
    \[
        \Lambda^{\rm B}_{E_\L}\left(m_{\varphi_i, \rm max}^{E_\L}\left(a\right)\right) = m_{\varphi_i, \rm red}^{E_\L}\left(\Lambda^{\rm B}_{E_\L}\left(a\right)\right) = 0
    \]
    for all \(a \in \ker(\Lambda^{\rm B}_{E_\L})\) (and similarly for the non-essential counterparts). As \(\Lambda^{\rm B}_{E_\L}\) is injective on $\borelessalg G$, it follows that \(m_{\varphi_i, \rm max}^{E_\L}(a) = 0\). However, \(m_{\varphi_i, \rm max}^{E_\L}(a) \to a\) by the discussion following \eqref{eq:thm:wk-cont-nuc: tends to 0}. Hence \(a = 0\), \emph{i.e.}\ the essential left regular representation
    \[
        \Lambda^{\rm B}_{E_\L} \colon \essmaxborel{G} \to \essborel{G}
    \]
    is injective. The argument for \cref{thm:wk-cont-nuc}: \cref{thm:wk-cont-nuc:ba} \(\Rightarrow\) \cref{thm:wk-cont-nuc:bwc} is analogous.

    \

    \cref{thm:wk-cont-nuc-ess:ba} \(\Rightarrow\) \cref{thm:wk-cont-nuc-ess:bn}.
    We keep the notation from the previous implication, including \(\varphi_i \coloneqq \xi_i^* * \xi_i\) and the multiplier maps \(m_i \coloneqq m_{\varphi_i, {\rm max}}^{E_\L} = m_{\varphi_i, {\rm red}}^{E_\L}\).\footnote{\, Since we only explicitly prove \cref{thm:wk-cont-nuc-ess} we will not be needing the maps \(m_{\varphi_i, {\rm max}}^{P} = m_{\varphi_i, {\rm red}}^{P}\), as these are only needed in the proof of \cref{thm:wk-cont-nuc}.}
    Let \(A\) be any given \cstar{}algebra, and let
    \(Q \colon \essborel{G} \otimes_{\rm max} A \twoheadrightarrow \essborel{G} \otimes_{\rm min} A\) be the canonical surjection. We shall prove that $Q$ is is injective. 

    \begin{claim}
        \((m_i \otimes \id_A)(\essborel{G} \otimes_{\rm max} A) \subseteq \iota_{K_i}(B_{K_i}^{\ess}(G, A)) \subseteq \essborel{G} \otimes_{\max} A\), where \(K_i \subseteq G\) is any compact subset containing the support of \(\varphi_i\).
    \end{claim}
    \begin{proof}
        The fact that such \(K_i\) exists follows from \(\varphi_i \in \borelessalg{G}\) and \cref{lem:partition-into-borel-bisections}.
        The image of $\essborel{G}\odot A$ under \(m_i \otimes \id_A\) is contained in $B_{K_i}^{\ess}(G) \odot A$ by \cref{lemma:image of schur multipliers}, and this is contained in the closed subspace $\iota_{K_i}(B_{K_i}^{\ess}(G, A))\sbe \essborel{G} \otimes_{\max} A$ by \cref{lem:embedding-max-tensor}. 
    \end{proof}

    Observe that the diagram
    \begin{center}
        \begin{tikzcd}[scale=50em,row sep=large, column sep=large]
            \essborel{G} \otimes_{\rm max} A \arrow{r}{m_i \otimes_{\max} \id} \arrow[two heads]{d}{Q}  & \essborel{G} \otimes_{\rm max} A \arrow[two heads]{d}{Q} \\
            \essborel{G} \otimes_{\rm min} A \arrow{r}{m_i \otimes_{\min} \id} &   \essborel{G} \otimes_{\rm min} A
        \end{tikzcd}
    \end{center}
    commutes.
    To prove that \(Q\) is injective, take some \(x \in \ker(Q)\). Then
    \[
        Q\left(\left(m_i \otimes_{\max} \id\right)\left(x\right)\right) = \left(m_i\otimes_{\min} \id\right)\left(Q\left(x\right)\right) = 0.
    \]
    Since \((m_i \otimes \id_A) (x) \in \borelessalg{G} \odot A\); \((m_i \otimes \id_A) (x) \to x\) by the computations in the previous part of the proof; and \(Q\) is injective on \(\borelessalg{G} \odot A\), it follows that \((m_i \otimes \id_A) (x) = 0\) for all \(i\). In particular \(x = \lim_i m_i \otimes \id_A (x) = 0\), as desired.
    
    \

    \cref{thm:wk-cont-nuc:bwc} \(\Rightarrow\) \cref{thm:wk-cont-nuc:wc} (only for \cref{thm:wk-cont-nuc}).
    Follows from \cref{cor:borel-algs,cor:inclusions c into b}.

    \cref{thm:wk-cont-nuc-ess:bwc} \(\Rightarrow\) \cref{thm:wk-cont-nuc-ess:wc} (only for \cref{thm:wk-cont-nuc-ess}).
    Let \(\psi \colon \essmaxalg{G} \to \essmaxborel{G}\) be as in \cref{cor:inclusions c into b}~\cref{cor:inclusions c into b:2}. Item \cref{thm:wk-cont-nuc-ess:wc} follows from the fact that we have the following commuting diagram:
    \begin{center}
        \begin{tikzcd}[scale=50em]
            \essmaxalg{G} \arrow{r}{\psi} \arrow[two heads]{d}{\Lambda_{E_\L}} & \essmaxborel{G} \arrow[equal]{d}{} \\
            \essalg{G} \arrow[hook]{r}{} & \essborel{G}
        \end{tikzcd}
    \end{center}

    \

    \cref{thm:wk-cont-nuc-ess:n} \(\Rightarrow\) \cref{thm:wk-cont-nuc-ess:a}. By nuclearity of \(\essalg{G}\) there are c.c.p. maps
    \[
        \varphi_i \colon \essalg{G} \to \Mat_{k\left(i\right)} \; \text{ and } \; \psi_i \colon \Mat_{k\left(i\right)} \to \essalg{G}
    \]
    such that \(\psi_i(\varphi_i(a)) \to a\) (in norm) for all \(a \in \essalg{G}\).
    Let \(b_i \in \Mat_{k(i)}(\essalg{G})\) be the square root of the positive matrix \((\psi_i(e_{p,q}^{(i)}))_{p, q = 1}^{k(i)}\in \Mat_{k(i)}(\essalg G)\), see \cite{Brown-Ozawa:Approximations}*{Proposition~1.5.12}. Here 
    \(e_{p,q}^{(i)}\) denotes the unit matrix in \(\Mat_{k(i)}(\C)\) whose \((p,q)\)-entry is \(1\) and every other entry is \(0\).
    Consider the element
    \[
        \eta_i \coloneqq \sum_{p, q = 1}^{k(i)} \delta_{p}^{(i)} \otimes \delta_{q}^{(i)} \otimes b_{i,p,q} \in \ell^2_{k(i)} \otimes \ell^2_{k(i)} \otimes \essalg{G}
    \]
    for every \(i\), where \((\delta_{p}^{(i)})_{p = 1}^{k(i)}\) is the canonical orthonormal basis of the Euclidean Hilbert space \(\ell^2_{k(i)}=\C^{k(i)}\) and \(b_{i,p,q}\) is the \((p,q)\)-entry of \(b_i\). Let \(\ell^2_{k(i)} \otimes \ell^2_{k(i)} \otimes \essalg{G}\) denote the standard (right) Hilbert \(\essalg{G}\)-module (recall \cref{lemma:taking norm is continuous,rem:evaluate-eta-at-non-dang-units}), whose \cstar{}algebra of adjointable operators we identify with $\Mat_{k(i)}(\C)\otimes\Mat_{k(i)}(\C)\otimes\M(\essalg G)$.\footnote{\, By \(\M(A)\) we mean the multiplier algebra of a given \cstar algebra \(A\). In particular, \(\M(\essalg G) = \essalg G\) whenever \(X\) is compact.} 
    Denoting by \(1_{k(i)}\) the unit of \(\Mat_{k(i)}(\C)\) and \(1_{\rm ess}\) the unit of \(\M(\essalg{G})\), we have
    \begin{align*}
        \langle \eta_i, (e_{p,q}^{(i)} \otimes 1_{k(i)} \otimes 1_{\rm ess}) \eta_i\rangle = \sum_{s = 1}^{k(i)} b_{i,p,s}^* b_{i,q,s} = \psi_i(e_{p,q}^{(i)})
    \end{align*}
    for all \(p, q = 1, \dots, k(i)\). As the matrix units linearly generate all other matrices, the above also holds for all matrices \(a \in \Mat_{k(i)}\), \emph{i.e.}\
    \begin{equation} \label{eq:nuc-eta}
        \psi_i\left(a\right) = \langle \eta_i, \left(a \otimes 1_{k\left(i\right)} \otimes 1_{\rm ess}\right)\eta_i\rangle.
    \end{equation}
    In particular, since $\psi_i$ is contractive, notice that
      \begin{equation} \label{eq:norm-less-1}
      \langle \eta_i, \eta_i\rangle=\psi_i(1_{k(i)})\leq 1,
      \end{equation}
    so that $\|\eta_i\|\leq 1$ for all $i$. Fix \(\varepsilon > 0\). As \(\essalgalg G\) is dense in \(\essalg{G}\), we can find \(\eta_{i,\varepsilon} \in \ell^2_{k(i)} \otimes \ell^2_{k(i)} \otimes \essalgalg{G}\) with $\|\eta_i-\eta_{i,\varepsilon}\|< \varepsilon$ and $\|\eta_{i,\varepsilon}\|\leq 1$. 
    Since \(\algalg{G}\) quotients onto \(\essalgalg{G}\), there exists some \(\tilde{\eta}_{i,\varepsilon} \in \ell^2_{k(i)} \otimes \ell^2_{k(i)} \otimes \algalg{G}\) such that \(\Id\otimes\Id\otimes \Lambda_{E_\L}(\tilde{\eta}_{i,\varepsilon}) = \eta_{i,\varepsilon}\).
    Moreover, we may view \(\tilde{\eta}_{i,\varepsilon}\) as an element of \(\mathfrak{C}_c(G, \ell^2_{k(i)} \otimes \ell^2_{k(i)})\).
    Consider then:
    \begin{equation} \label{eq:nuc-def-of-xi}
        \xi_{i,\varepsilon}\left(g\right) \coloneqq \norm{\tilde{\eta}_{i,\varepsilon}\left(g\right)}_{\ell^2_{k(i)} \otimes \ell^2_{k(i)}},
    \end{equation}
    and let \((\xi_{i,\varepsilon,\mu})_\mu \subseteq \auxalg{G}\) be some net converging to \(\xi_{i,\varepsilon}\), that is,
    \begin{equation}
        \norm{\xi_{i,\varepsilon} - \xi_{i,\varepsilon,\mu}}_\infty \to 0 \; \text{ when } \mu \to 0.
    \end{equation}
    (Observe that such nets exist by \cref{lemma:taking norm is continuous}.)

    We claim that the net \((\xi_{i,\varepsilon,\mu})_{i,\varepsilon,\mu}\) witnesses the essential amenability of \(G\).
    First, observe that by taking smaller \(\varepsilon\) and \(\mu\) if necessary, we may substitute \(\xi_{i,\varepsilon,\mu}\) with \(\xi_{i,\varepsilon}\).
    Therefore, it suffices to check that conditions \cref{def:ess-ame:norm,def:ess-ame:inva} in \cref{def:ess-ame} are met.
    For this, write $\tilde\eta_{i,\varepsilon}$ as \(\tilde{\eta}_{i,\varepsilon} = \sum_{k,j} \delta_k^{(i)} \otimes \delta_j^{(i)} \otimes \zeta_{k,j}^{(i,\varepsilon)}\). Then
    \begin{align}
        \sum_{g \in Gx} \xi_{i,\varepsilon} \left(g\right)^2 & = \sum_{g \in Gx} \norm{\tilde{\eta}_{i,\varepsilon}\left(g\right)}^2 = \sum_{g \in Gx} \sum_{k,j} \abs{\zeta_{k,j}^{(i, \varepsilon)}\left(g\right)}^2 = \sum_{k,j} \sum_{g \in Gx} \abs{\zeta_{k,j}^{(i, \varepsilon)}\left(g\right)}^2 \nonumber \\
        & = \sum_{k,j} \left(\zeta_{k,j}^{(i, \varepsilon)*} * \zeta_{k,j}^{(i, \varepsilon)}\right) \left(x\right) = \langle \tilde{\eta}_{i, \varepsilon}, \tilde{\eta}_{i, \varepsilon} \rangle \left(x\right) \label{eq:nuc-smaller-than-1}
    \end{align}
    for all \(x \in X\). Furthermore, whenever \(x\) is not dangerous, \emph{i.e.}\ \(x \not\in D\), it follows that 
    \[
        \langle \tilde{\eta}_{i, \varepsilon}, \tilde{\eta}_{i, \varepsilon} \rangle \left(x\right) = \langle \eta_{i, \varepsilon}, \eta_{i, \varepsilon} \rangle \left(x\right) \leq \|\eta_{i,\varepsilon}\|^2\leq 1,
    \]
    where we used notation of \cref{rem:evaluate-eta-at-non-dang-units}.
    Thus
    \begin{equation}\label{eq:bound-1}
        \sup_{x \in X \setminus D} \sum_{g \in Gx} \overline{\xi_{i,\varepsilon}\left(g\right)} \xi_{i,\varepsilon}\left(g\right) =\sup_{x \in X \setminus D} \sum_{g \in Gx} \xi_{i,\varepsilon}\left(g\right)^2   \leq 1.
    \end{equation}
    This yields \cref{def:ess-ame}~\cref{def:ess-ame:norm}.    
    In order to prove \cref{def:ess-ame}~\cref{def:ess-ame:inva} it suffices to prove it for all compact sets \(K\) contained in some (henceforth fixed) open bisection \(s \subseteq G\) (cf.\ \cref{rem:ess-ame}~\cref{rem:ess-ame:compact in bisection}). We fix a pointwise positive function \(a v_s \in \algalg{G}\) (recall \cref{notation:cont functions as sums}) with $0\leq a\leq 1$ and such that $a v_s$ is equal to \(1\) on \(K\) and is compactly supported and continuous on \(s\) (such a function exists since \(G\) is \'etale, and thus \(s\) is homeomorphic to some open subset of a Hausdorff space).
    Fix any given \(g \in K \setminus D \subseteq s \setminus D\), and let \(x \coloneqq \s(g)\).
    By \cref{lemma:dangerous arrows} it follows that \(x \in X \setminus D\).
    We then have \(1 = (av_s)(g) = a(\rg(g))\) and, thus,
    \begin{equation} \label{eq:nuc:some ev is 1}
        1 = \overline{a\left(\rg(g)\right)} a\left(\rg(g)\right) = \sum_{h \in Gx} \overline{\left(av_s\right)\left(h\right)} \left(av_s\right)\left(h\right) = \left(\left(av_s\right)^* * av_s\right)\left(x\right),
    \end{equation}
    where the second equality uses that \(s\) is a bisection containing \(g\) (and thus \(s\) does \emph{not} contain any other element of \(Gx\), as \(x = \s(g)\)).
    Recall that we may see \(\algalg{G} \subseteq \fullalg{G}\) and, thus, \(\Lambda^{\rm ess}(av_s)\) is naturally an element in \(\essalg{G}\) (cf.\ \cref{cor:diagram}).
    For the sake of readability, we let
    \[
        c \coloneqq \Lambda^{\rm ess}\left(av_s\right) \in \essalg{G} \;\; \text{ and } \;\; d_i \coloneqq \varphi_i\left(c\right) \in \Mat_{k\left(i\right)}(\C).
    \]
    Thus \(d_i\) is a contraction, as \(\norm{av_s} \leq 1\) and \(\psi_i\) is contractive. Moreover,
    \begin{align*}
        \norm{\left(\varphi_i\left(c\right) \otimes 1_{k\left(i\right)} \otimes 1_{\rm ess}\right) \eta_{i} - \eta_{i} c}^2 & = \langle \left(d_i \otimes 1_{k\left(i\right)} \otimes 1_{\rm ess}\right) \eta_i, \left(d_i \otimes 1_{k\left(i\right)} \otimes 1_{\rm ess}\right) \eta_i\rangle \\
        & \quad\quad\quad + \langle \eta_i c, \eta_i c\rangle - \langle \left(d_i \otimes 1_{k\left(i\right)} \otimes 1_{\rm ess}\right) \eta_i, \eta_i c\rangle \\
        & \quad\quad\quad\quad - \langle \eta_i c, \left(d_i \otimes 1_{k\left(i\right)} \otimes 1_{\rm ess}\right) \eta_i\rangle \\
        & = \psi_i\left(d_i^*d_i\right) + c^* \psi_i\left(1_{k\left(i\right)}\right) c - \psi_i\left(d_i\right)^* c - c^* \psi_i\left(d_i\right) \\
        & \leq \left(\psi_i \circ \varphi_i\right)\left(c^*c\right) + c^*c - \left(\psi_i \circ \varphi_i\right)\left(c\right)^* c \\
        & \quad\quad\quad\quad - c^* \left(\psi_i \circ \varphi_i\right)\left(c\right) \to 0 \;\;\; \text{(in } i \text{).}
    \end{align*}
    Note that the second equality uses \eqref{eq:nuc-eta}, whereas the inequality uses the fact that both \(\varphi_i\) and \(\psi_i\) are completely positive maps.
    In particular, the above computations imply that
    \begin{equation*}
        \norm{\left(\varphi_i\left(c\right) \otimes 1_{k\left(i\right)} \otimes 1_{\rm ess}\right) \eta_{i, \varepsilon} - \eta_{i,\varepsilon} c} \to 0
    \end{equation*}
    when \(i\) grows and \(\varepsilon\) tends to \(0\) (henceforth we shall only say that \((i,\varepsilon)\) \emph{grows} to mean this behavior).
    Moreover    
    \begin{equation}
        \langle \left(\varphi_i\left(c\right) \otimes 1_{k\left(i\right)} \otimes 1_{\rm ess}\right) \eta_{i,\varepsilon}, \eta_{i,\varepsilon} c \rangle \left(x\right) \to \left(\left(av_s\right)^* * av_s\right)\left(x\right) = 1 \label{eq:converges to 1}
    \end{equation}
    when \((i,\varepsilon)\) grows (recall \eqref{eq:nuc:some ev is 1}).
    Using \cref{cor:j-ess} we may now observe that
    \begin{equation} \label{eq:tilde eta is eta out of d}
        \tilde{\eta}_{i, \varepsilon}\left(y\right) = \eta_{i,\varepsilon}\left(y\right)
    \end{equation}
    for all \(y \in X \setminus D\), \emph{i.e.}\ \(\eta_{i,\varepsilon}\) and \(\tilde{\eta}_{i, \varepsilon}\) coincide outside of the dangerous units (recall that \(x \not\in D\)).
    Furthermore, for all \(h \in G \rg(g)\),
    \begin{equation} \label{eq:eta gh is etac h}
         \tilde{\eta}_{i,\varepsilon}\left(gh\right) = \left(\tilde{\eta}_{i,\varepsilon} \, av_s\right)\left(h\right).
    \end{equation}
    After all the above computations, we may now give the estimate: 
    \begin{align*}
        \left(\xi_{i,\varepsilon}^* * \xi_{i,\varepsilon}\right)\left(g\right) & = \sum_{h \in G \rg(g)} \overline{\xi_{i,\varepsilon}\left(h\right)} \cdot \xi_{i, \varepsilon}\left(gh\right) \stackrel{\eqref{eq:nuc-def-of-xi}}{=} \sum_{h \in G \rg(g)} \norm{\tilde{\eta}_{i,\varepsilon}\left(h\right)} \cdot \norm{\tilde{\eta}_{i,\varepsilon}\left(gh\right)} \allowdisplaybreaks \\
        & \stackrel{\eqref{eq:eta gh is etac h}}{=} \sum_{h \in G \rg(g)} \norm{\tilde{\eta}_{i,\varepsilon}\left(h\right)} \cdot \norm{\left(\tilde{\eta}_{i,\varepsilon} \, av_s\right)\left(h\right)} \allowdisplaybreaks \\
        & \geq \abs{\sum_{h \in G\rg(g)} \langle \left(\varphi_i\left(c\right) \otimes 1_{k\left(i\right)} \otimes 1_{\rm ess}\right) \tilde{\eta}_{i,\varepsilon}\left(h\right), \left(\tilde{\eta}_{i,\varepsilon} \, a v_s\right) \left(h\right) \rangle} \allowdisplaybreaks \\
        & \stackrel{\eqref{eq:evaluate-eta-at-non-dang-units}}{=} \abs{\langle \left(\varphi_i\left(c\right) \otimes 1_{k\left(i\right)} \otimes 1_{\rm ess}\right) \tilde{\eta}_{i,\varepsilon}, \tilde{\eta}_{i,\varepsilon} c\rangle \left(x\right)} \\
        & \stackrel{\eqref{eq:tilde eta is eta out of d}}{=} \abs{\langle \left(\varphi_i\left(c\right) \otimes 1_{k\left(i\right)} \otimes 1_{\rm ess}\right) \eta_{i,\varepsilon}, \eta_{i,\varepsilon} c\rangle \left(x\right)} \stackrel{\eqref{eq:converges to 1}}{\to} 1.
    \end{align*}
    Since the choice of \(g \in K \setminus D\) was arbitrary, \emph{i.e.}\ the function \(a v_s\) depends only on \(K\) and not on \(g\), the above convergence is uniform for $g$ in $K\setminus D$. Moreover, by~\eqref{eq:bound-1} and the Cauchy-Schwartz inequality we also have $\left(\xi_{i,\varepsilon}^* * \xi_{i,\varepsilon}\right)\left(g\right) \leq 1$ for all $g\in K\setminus D$. All this shows that $\left(\xi_{i,\varepsilon}^* * \xi_{i,\varepsilon}\right)\left(g\right)\to 1$ uniformly for $g\in K\setminus D$, yielding \cref{def:ess-ame}~\cref{def:ess-ame:inva} and finishing the proof.

    \

    \cref{thm:wk-cont-nuc-ess:bn} \(\Rightarrow\) \cref{thm:wk-cont-nuc-ess:ba}. The proof of this implication is very similar to that of \cref{thm:wk-cont-nuc-ess:n} \(\Rightarrow\) \cref{thm:wk-cont-nuc-ess:a}, so we will only sketch it. By nuclearity of \(\essborel{G}\) there are c.c.p. maps
    \[
        \varphi_i \colon \essborel{G} \to \Mat_{k\left(i\right)} \; \text{ and } \; \psi_i \colon \Mat_{k\left(i\right)} \to \essborel{G}
    \]
    approximating the identity map in the point-norm topology. Following the same strategy, one arrives at \[\xi_{i,\varepsilon}(g) \coloneqq \norm{\tilde{\eta}_{i,\varepsilon}(g)}_{\ell^2_{k(i)} \otimes \ell^2_{k(i)}}\] as in \eqref{eq:nuc-def-of-xi}.
    Nevertheless, note that in this setting \(\tilde{\eta} \in \borelalg{G}\).
    Applying \cref{lemma:taking norm is borel} instead of \cref{lemma:taking norm is continuous} one deduces that \(\xi_{i,\varepsilon} \in \borelalg{G}\) on the nose (no need to consider nets on \(\mu > 0\)). The rest of the proof now goes through, since it was only about computing certain values of certain functions, and thus apply just as well to this case.

    \

    \cref{thm:wk-cont-nuc:ba} \(\Rightarrow\) \cref{thm:wk-cont-nuc:n} (only for \cref{thm:wk-cont-nuc}). We know that \cref{thm:wk-cont-nuc:ba} \(\Rightarrow\) \cref{thm:wk-cont-nuc:bn,thm:wk-cont-nuc:bwc,thm:wk-cont-nuc:wc}. In particular,
    \begin{itemize}
        \item \(\fullborel G \cong \redborel G\) is nuclear.
        \item \(\fullalg G \cong \redalg G\).
    \end{itemize}
    By the `moreover' statement in \cref{cor:inclusions c into b},
    \[
        \redalg G \cong \fullalg G \subseteq \fullborel G \cong \redborel G
    \]
    is a max-injective inclusion of \cstar{}algebras. It then follows from \cref{lemma:max-inj-preserves-nuc} that \(\redalg G\) is nuclear.
    
    \cref{thm:wk-cont-nuc-ess:ba} \(\Rightarrow\) \cref{thm:wk-cont-nuc-ess:n} (only for \cref{thm:wk-cont-nuc-ess} under the assumption that $\psi$ is max-injective). We know that \cref{thm:wk-cont-nuc:ba} \(\Rightarrow\) \cref{thm:wk-cont-nuc:bn,thm:wk-cont-nuc:wc}. In particular, \(\essmaxborel G \cong \essborel G\) is nuclear. If \(\psi \colon \essmaxalg{G} \to \essmaxborel{G}\) is max-injective, then $\essmaxalg G$, and hence also $\essalg G$, is nuclear by \cref{lemma:max-inj-preserves-nuc}.
\end{proof}

\begin{remark}
    \cref{ex:trivial-action} (see also \cref{ex:trivial-action-ame}) shows that \(G\) may be essentially amenable but not amenable. Moreover, in these examples $\essmaxalg{G} \cong \essalg{G}$ is nuclear (commutative). Thus the diagram in \cref{cor:diagram} collapses as follows:
    \begin{center}
    \begin{tikzcd}[scale=50em]
        \fullalg{G} \arrow[two heads]{dr}{\Lambda^{\text{max}}_{E_\mathcal{L},P}} \arrow[two heads]{d}{\Lambda_P} &  \\
        \redalg{G} \arrow[two heads]{r}{\Lambda^{\text{red}}_{E_\mathcal{L},P}} & \essmaxalg{G} \cong \essalg{G}
    \end{tikzcd}
    \end{center}
\end{remark}

\subsection{Application to classical crossed products}
In this section, we derive some direct consequences of our main results for crossed products by ordinary group actions on spaces. If \(\Gamma\) is a discrete group acting on a locally compact Hausdorff space \(X\), its transformation groupoid \(G \coloneqq X \rtimes \Gamma\) is a locally compact Hausdorff groupoid. In this case, \(\borelalg G\) (respectively, \(\algalg G\)) denotes the space of bounded Borel (respectively, continuous) compactly supported functions \(X \times \Gamma \to \mathbb{C}\). The associated \cstar{}algebras are given by  
\[
\fullborel G = B_0(X) \rtimes_{\max} \Gamma \quad \text{and} \quad \fullalg G = \contz(X) \rtimes_{\max} \Gamma,
\]  
where these crossed products are taken with respect to the canonical \(\Gamma\)-actions on \(B_0(X)\) and \(\contz(X)\). Here, \(B_0(X)\) denotes the (commutative) \cstar{}algebra of bounded Borel functions vanishing at infinity, while \(\contz(X) \subseteq B_0(X)\) is the usual \cstar{}algebra of continuous functions vanishing at infinity. Similarly, we have  
\[
\redborel G = B_0(X) \rtimes_{\red} \Gamma \quad \text{and} \quad \redalg{G} = \contz(X) \rtimes_{\red} \Gamma.
\]  
Since \(G\) is Hausdorff, there is no distinction between essential and non-essential \cstar{}algebras. As an immediate consequence of our main result (cf.\ \cref{thm:wk-cont-nuc}), we obtain \cref{thm-intro-borel} from the introduction. In particular, it follows that \(\contz(X) \rtimes_{\red} \Gamma\) is nuclear if and only if \(B_0(X) \rtimes_{\red} \Gamma\) is nuclear.  

The algebra \(B_0(X) \rtimes_{\red} \Gamma\) is typically large and often non-separable, even if \(G=X\times \Gamma\) is second countable. Structurally, it resembles a von Neumann algebra or an AW\(^*\)-algebra, though it generally belongs to neither of these classes. For instance, it may fail to be unital if \(X\) is not compact. The following example illustrates the subtleties involved, and it indicates that we are at the ``borderline'' where our techniques apply and/or the results are even possible.

\begin{example}\label{exa:crossed-product-l-infty}
Assume \(X\) is compact and \(\Gamma\) acts freely on \(X\).  
Then, as a $\Gamma$-set, we can identify \(X\) with \(X/\Gamma \times \Gamma\) by choosing a section for the orbit map \(X \to X/\Gamma\). However, this section is typically neither continuous nor Borel measurable.  

In particular, if \(\Gamma\) is exact, then its action on \(\ell^\infty(X)\) is amenable, which implies that the maximal and reduced crossed products coincide:
\[
\ell^\infty(X) \rtimes_{\max} \Gamma = \ell^\infty(X) \rtimes_{\red} \Gamma.
\]
Moreover, this \cstar{}algebra is nuclear. However, the original action of \(\Gamma\) on \(X\) can be any free action, which need not be amenable. Consequently, the algebras  
\[
B_b(X) \rtimes_{\max} \Gamma, \quad B_b(X) \rtimes_{\red} \Gamma, \quad \cont(X) \rtimes_{\max} \Gamma, \quad \cont(X) \rtimes_{\red} \Gamma
\]
may all be distinct and non-nuclear.  

We would like to thank Rufus Willett for bringing this example to our attention.    
\end{example}

%%%%%%%%%%%%%%%%%%%%%%%%%%%%%%%%%%%%%%%%%%%%%%%%%%%%%%%%%%%%%%%%%%%%%%%
% THE BOREL SINGULAR IDEAL AND EQUALITY OF REDUCED AND ESSENTIAL ALGEBRAS
%%%%%%%%%%%%%%%%%%%%%%%%%%%%%%%%%%%%%%%%%%%%%%%%%%%%%%%%%%%%%%%%%%%%%%%

\section{The Borel singular ideal and equality of reduced and essential algebras} \label{sec:simplicity}
In this section we give sufficient conditions for the canonical quotient map \(\redalg{G} \twoheadrightarrow \essalg{G}\) to be an isomorphism, cf.\ \cref{prop:alg-sing-ideal-closure}.
For this, we use all the machinery from \cref{sec:ess-reps-algs,sec:aux-alg,sec:ess-ap} and, in particular, we are especially interested in when the closure of (an ideal containing) \(\NN_{E_\L} \cap \algalg{G}\) (see \cref{cor:nucleus is in d}) is (or, rather, contains) the very ideal \(J_{\text{sing}}\) appearing in \cref{the:KM-ess-J-sing}.

Before we state the next result, we need to introduce some further notation.   We define $$\auxalgc{G}:=\bigcup_{K\sbe G}\overline{\auxalgK{G}},$$
with the union ranging over all compact subsets $K\sbe G$, and $\overline{\auxalgK{G}}$ means the closure of $\auxalgK G$ with respect to the supremum norm. This is then a closed subspace of $B_K(G)$, which is, in turn, a closed subspace of $\redborel{G}$ by \cref{cor:norms-equivalent-B_K(G)}. In particular, we may view $\auxalgc{G}\sbe \redborel G$.

\begin{proposition} \label{prop:alg-sing-ideal-closure}
    If \(G\) is an amenable \'etale groupoid, then
    \[
        J_{\rm sing} \stackrel{\rm (def)}{=} \ker\left(\redalg G \twoheadrightarrow \essalg G\right) \subseteq \overline{ \left\{ a \in \auxalgc{G} \mid \supp(a) \text{ is meager} \right\} }^{\norm{\cdot}_{\rm red}}.\footnote{\, Recall that \(\norm{a}_{\rm red} \coloneqq \sup_{x \in X} \norm{\lambda_x(a)}\) for \(a \in \borelalg G\).}
    \]
\end{proposition}
\begin{proof}
    Let \((\xi_i)_i \subseteq \auxalg{G}\) witness the amenability of \(G\). Just as in the proof of \cref{thm:wk-cont-nuc}, let \( \varphi_i \coloneqq \xi_i^* * \xi_i \in \auxalg{G}\) and \(m_i \coloneqq m_{\varphi_i, \rm red}^P \colon \redalg G \to \redborel{G}\) be the Herz-Schur multipliers of \cref{prop:schur multiplier}. Furthermore, let \(K_i \subseteq G\) be a compact subset containing the support of \(\varphi_i\).
    Take some \(a \in J_{\rm sing}\). First notice that $m_i(a)\in \auxalgc{G}$ because this is true for $a\in \algalg G$ and $\auxalgK G$ is a closed subspace of $\redalg G$.
    By computations similar to those in \cref{thm:wk-cont-nuc}, we have
    \begin{itemize}
        \item \(m_i(a) \to a\) in norm when \(i\) grows; and
        \item \(\supp(m_i(a)) \subseteq \supp(a)\) is meager.
    \end{itemize}
    This is exactly saying that \(a \in \overline{ \left\{ a \in \auxalgc{G} \mid \supp(a) \text{ is meager} \right\} }\), as desired.
\end{proof}

The main interest of \cref{prop:alg-sing-ideal-closure} is that one can test the simplicity of \(\redalg{G}\) by testing the simplicity not of \(\algalg{G}\), but of \(\auxalg{G}\) and/or \(\redborel G\).

\begin{theorem}[cf.\ \cref{thm-intro-red-eq-ess}] \label{cor:simple-red-eq-ess}
    Let \(G\) be an amenable \'etale groupoid. Suppose that
    \[
        \left\{a \in \auxalgc G \mid \supp(a) \text{ is meager}\right\} = 0.
    \]
    Then \(\fullalg{G} \cong \essalg{G}\) via the canonical quotient map. In particular, \(\fullalg{G}\) is simple if and only if \(G\) is minimal and topologically free.
\end{theorem}
\begin{proof}
    The fact that \(\redalg G \twoheadrightarrow \essalg G\) is injective is an immediate consequence of \cref{prop:alg-sing-ideal-closure}.
    Likewise, if \(G\) is minimal and topologically free, then \(\fullalg{G} \cong \redalg{G}\) is simple by \cite{KwasniewskiMeyer-essential-cross-2021}*{Theorem 7.29}.
\end{proof}

%%%%%%%%%%%%%%%%%%%%%%%%%%%%%%%%%%%%%%%%%%%%%%%%%%%%%%%%%%%%%%%%%%%%%%%
% BRUCE-LI SEMIGROUP ALGEBRAS AS QUOTIENTS OF THE MAXIMAL ESSENTIAL C*-ALGEBRA AND THEIR NUCLEARITY
%%%%%%%%%%%%%%%%%%%%%%%%%%%%%%%%%%%%%%%%%%%%%%%%%%%%%%%%%%%%%%%%%%%%%%%

\section{Bruce-Li semigroup algebras as quotients of the maximal essential C*-algebra and their nuclearity} \label{sec:bruce-li}

In this concluding (and rather concise) section of the paper, we apply the machinery developed thus far to prove that the \cstar algebras \(\mathfrak{A}_\sigma\), studied in \cite{bruce-li:2024:alg-actions}, often arise as \emph{essential exotic} \cstar algebras. Specifically, we show that \(\mathfrak{A}_\sigma\) lies between \(\essmaxalg{G_\sigma}\) and \(\essalg{G_\sigma}\) (cf.\ \cref{cor:diagram-bruce-li}).
It is worth noting that \cite{bruce-li:2024:alg-actions} establishes only that the \cstar algebras \(\mathfrak{A}_\sigma\) appear as exotic \cstar algebras in the broader sense, meaning that they lie between \(\fullalg{G}\) and \(\essalg{G}\), while also analyzing various properties of these algebras. 
We express our gratitude to Chris Bruce for several insightful discussions regarding \cite{bruce-li:2024:alg-actions}.

We begin by briefly recalling the notation and terminology from \cite{bruce-li:2024:alg-actions}, while directing the reader to that work for a more detailed exposition.
An \emph{algebraic action} (cf.\ \cite{bruce-li:2024:alg-actions}*{Definition 2.1}) is a faithful action \(\sigma \colon S \curvearrowright \Gamma\), where \(S\) is a countable (and necessarily left-cancellative) monoid, \(\Gamma\) is a discrete countable group, and each map \(\sigma_s \colon \Gamma \to \Gamma\) is an injective group homomorphism for all \(s \in S\).
From such an action, \cite{bruce-li:2024:alg-actions}*{Definition 3.29} constructs an \'etale, locally compact groupoid \(G_\sigma\) with a compact and totally disconnected unit space \(\partial \hat{\mathcal{E}} \subseteq G_\sigma\). Here, \(\mathcal{E}\) denotes the idempotent semilattice of an inverse semigroup \(I_\sigma\) (cf.\ \cite{bruce-li:2024:alg-actions}*{Subsection 3.4}); \(\hat{\mathcal{E}}\) is the space of characters of \(\mathcal{E}\); and \(\partial \hat{\mathcal{E}}\) is the (closed and invariant) subset of \(\hat{\mathcal{E}}\) consisting of \emph{tight} filters (cf.\ \cite{bruce-li:2024:alg-actions}*{Definition 3.20}).
The inverse semigroup \(I_\sigma\), generated by \(\Gamma\) and \(S\), consists of the partial bijections of \(\Gamma\). It acts on \(\hat{\mathcal{E}}\), leaving \(\partial \hat{\mathcal{E}}\) invariant. The groupoid \(G_\sigma\) is then defined as the groupoid of germs of this action, \(I_\sigma \curvearrowright \partial \hat{\mathcal{E}}\).\footnote{\, By \emph{groupoid of germs}, we refer to the construction found in, e.g., \cites{BussMartinez:Approx-Prop,Exel:Inverse_combinatorial}, among others. Note, however, that there exists an alternative construction of the groupoid of germs that eliminates additional arrows.}
The \cstar algebra \(\mathfrak{A}_\sigma\), primarily studied in \cite{bruce-li:2024:alg-actions}*{Definition 3.1}, is a subalgebra of \(\B(\ell^2(\Gamma))\).

Observe that there is a map \(\Gamma \ni \gamma \mapsto \chi_\gamma \in \partial \hat{\mathcal{E}}\) with dense image (cf.\ \cite{bruce-li:2024:alg-actions}*{Definition 3.26 and Remark 3.27}).
In fact, the space \(\partial \hat{\mathcal{E}}\) can, thus, be seen as a completion of a quotient of \(\Gamma\).
Recall (cf.\ \cite{bruce-li:2024:alg-actions}*{Definition 4.11}) that the action \(\sigma \colon S \curvearrowright \Gamma\) is \emph{exact} if the trivial subgroup is the biggest subgroup of \(\Gamma\) that is invariant under \(\sigma_s\) for all \(s \in S\). In particular, it follows that if \(\Gamma\) is abelian then \(\sigma\) is exact if and only if \(G_\sigma\) is topologically free (cf.\ \cite{bruce-li:2024:alg-actions}*{Theorem 4.14}).
In fact, exactness of \(\sigma\) already implies topological freeness of \(G_\sigma\) (cf.\ \cite{bruce-li:2024:alg-actions}*{Corollary 4.18}).
For our purposes, the interest behind exactness of \(\sigma\) stems from \cite{bruce-li:2024:alg-actions}*{Lemma 5.1}, which in particular implies that the canonical map \(\Gamma \ni \gamma \mapsto \chi_\gamma \in \partial \hat{\mathcal{E}}\) is injective. In particular, this allows us to make use of \cref{prop: a class of ess reps}.

\begin{proposition} \label{bruce-li-algs-key}
    Let \(\sigma \colon S \curvearrowright \Gamma\) be an exact algebraic action (cf.\ \cite{bruce-li:2024:alg-actions}*{Definition 2.1}). Consider the representation
    \begin{align*}
        \pi_\sigma \colon \algalg{G_\sigma} & \to \B\left(\ell^2\left(\Gamma\right)\right), \\
        \pi_\sigma\left(f v_s\right)\left(\delta_\gamma\right) & \coloneqq
            \left\{ \begin{array}{rl}
                f\left(s\gamma\right) \delta_{s\gamma} & \text{if } \; \chi_\gamma \in s^*s, \\
                0 & \text{otherwise,}
            \end{array} \right.
    \end{align*}
    where \(s \subseteq G_\sigma\) is an open bisection and \(fv_s \colon G_\sigma \to \CCC\) is a function compactly supported on \(s\).
    The following assertions hold.
    \begin{enumerate}[label=(\roman*)]
        \item \label{bruce-li-algs-key:1} The image of \(\pi_\sigma\) is dense in \(\mathfrak{A}_\sigma\).
        \item \label{bruce-li-algs-key:2} \(\pi_\sigma\) is an essential representation of \(G_\sigma\).
    \end{enumerate}
\end{proposition}
\begin{proof}
    \cref{bruce-li-algs-key:1} is, simply, a reformulation of the fact that (in the notation of \cite{bruce-li:2024:alg-actions}*{Definition 3.1}) the \cstar algebra \(\mathfrak{A}_\sigma\) is generated as the closed span of the image of
    \[
        \Lambda_\sigma \colon I_\sigma \to \B\left(\ell^2\left(\Gamma\right)\right), \;\; \Lambda_\sigma\left(s\right) \delta_\gamma \coloneqq \left\{ \begin{array}{rl}
                \delta_{s\gamma} & \text{if } \; \chi_\gamma \in s^*s, \\
                0 & \text{otherwise,}
            \end{array} \right.
    \]
    where \(I_\sigma\) is the inverse semigroup of partial bijections of \(\Gamma\) generated by \(\{\gamma\}_{\gamma \in \Gamma}\) and \(\{s\}_{s \in S}\).
    It is clear, then, that the image of \(\pi_\sigma\) in the statement contains the image of \(\Lambda_\sigma\) above, so it suffices to prove that \(\pi_\sigma(fv_s)\), where \(fv_s\) is a continuous function compactly supported on a bisection \(s \subseteq G\), lies in the closed span of the image of \(\Lambda_\sigma\).
    Let \(\varepsilon > 0\) be given. Since the unit space \(\partial \hat{\mathcal{E}} \subseteq G_\sigma\) is totally disconnected, it follows that we may approximate \(f\) (which we see as a function supported on \(ss^* \subseteq \partial \hat{\mathcal{E}}\)) by a sum \(a_1 1_{A_1} + \dots + a_\ell 1_{A_\ell}\), where \((a_i)_{i = 1, \dots, \ell} \subseteq \CCC\) and \((A_i)_{i = 1, \dots, \ell}\) is a collection of compact clopen subsets of \(ss^* \subseteq \partial \hat{\mathcal{E}}\).
    That is, we have
    \[
        \norm{f - a_1 1_{A_1} - \dots - a_\ell 1_{A_\ell}} \leq \varepsilon.
    \]
    It follows that
    \begin{align*}
        \norm{\pi_\sigma\left(f v_s\right) - \Lambda_{\sigma} \left(a_1 1_{A_1 s} + \dots + a_\ell 1_{A_\ell s}\right)} & = \norm{\pi_\sigma\left(f v_s - a_1 1_{A_1 s} - \dots - a_\ell 1_{A_\ell s}\right)} \\
        & \leq \norm{f - a_1 1_{A_1} - \dots - a_\ell 1_{A_\ell}} \leq \varepsilon.
    \end{align*}
    Now, item \cref{bruce-li-algs-key:2} just follows from \cref{prop: a class of ess reps}, along with the fact that \(\Gamma\) does embed (as a set) into \(\partial \hat{\mathcal{E}}\) (cf.\ \cite{bruce-li:2024:alg-actions}*{Lemma 5.1}) and the fact that no unit \(\chi_\gamma \in \Gamma \subseteq \partial \hat{\mathcal{E}}\) is dangerous (cf.\ \cite{bruce-li:2024:alg-actions}*{Proposition 5.4}).
\end{proof}

An immediate consequence of \cref{cor: ess reps are nice,bruce-li-algs-key} is the following enrichment of the diagram in \cref{cor:diagram} for this particular class of \'etale groupoids.
(For the following last two results, note that we assume throughout \cref{sec:bruce-li} that both \(S\) and \(\Gamma\) are countable (unlike what is done in \cite{bruce-li:2024:alg-actions}). This, in particular, implies that \(G_\sigma\) is covered by countably many open bisections.)
\begin{corollary} \label{cor:diagram-bruce-li}
    Let \(\sigma \colon S \curvearrowright \Gamma\) be an exact algebraic action (cf.\ \cite{bruce-li:2024:alg-actions}) and let \(\mathfrak{A}_\sigma\) and \(G_\sigma\) be as above.
    Suppose that \(\chi_e G_\sigma \chi_e = \{\chi_e\}\). Then the identity map \(\id \colon \algalg{G_\sigma} \to \algalg{G_\sigma}\) induces maps making  
    \begin{center}
        \begin{tikzcd}[scale=50em]
            \fullalg{G_\sigma} \arrow[two heads]{d}{\Lambda^{\text{max}}_{E_\mathcal{L},P}} \arrow[two heads]{r}{\Lambda_P} & \redalg{G_\sigma} \arrow[two heads]{d}{\Lambda^{\text{red}}_{E_\mathcal{L},P}} \arrow{r}{P} & \borelb(\partial \hat{\mathcal{E}}) \arrow[two heads]{d}{\pi_\ess} \\
            \essmaxalg{G_\sigma} \arrow[two heads]{dr}{\pi_\sigma} \arrow[two heads]{r}{\Lambda_{E_\mathcal{L}}} & \essalg{G_\sigma} \arrow{r}{E_\L} & \borelb(\partial \hat{\mathcal{E}})/\meager(\partial \hat{\mathcal{E}}) \\
            & \mathfrak{A}_\sigma \arrow[two heads]{u}{} \arrow{r}{E} & \ell^\infty\left(\Gamma\right) \arrow{u}{Q_{\partial \hat{\mathcal{E}}, \Gamma}}
        \end{tikzcd}
    \end{center}
    commute, where
    \begin{enumerate}[label=(\roman*)]
        \item \(E \colon \B(\ell^2(\Gamma)) \to \ell^\infty(\Gamma)\) is the usual conditional expectation; and
        \item \(Q_{\partial \hat{\mathcal{E}}, \Gamma}\) is the map sending \(p_A \in \ell^\infty(\Gamma)\) to \([\chi_{\overline{A}}]\) (cf.\ \cref{cor: ess reps are nice}).
    \end{enumerate}
    Moreover, \(P, E_\L\) and \(E\) are all faithful maps.
\end{corollary}

We record the following last result, which relates the (essential) amenability of the groupoid \(G_\sigma\) to the nuclearity of the \cstar algebra \(\mathfrak{A}_\sigma\) (again, see \cite{bruce-li:2024:alg-actions}*{Definitions 3.1 and 3.29} for details).
\begin{corollary} \label{cor:bruce-li-actions}
    Let \(\sigma \colon S \curvearrowright \Gamma\) be an exact algebraic action (cf.\ \cite{bruce-li:2024:alg-actions}) and let \(\mathfrak{A}_\sigma\) and \(G_\sigma\) be as above. Suppose that \(\chi_e G_\sigma \chi_e = \{\chi_e\}\), and consider the assertions.
    \begin{enumerate}[label=(\roman*)]
        \item \label{cor:bruce-li-actions:a nuclear} \(\mathfrak{A}_\sigma\) is nuclear.
        \item \label{cor:bruce-li-actions:ess nuclear} \(\essalg{G_\sigma}\) is nuclear.
        \item \label{cor:bruce-li-actions:ess ame} \(G_\sigma\) is essentially amenable.
    \end{enumerate}
    Then \cref{cor:bruce-li-actions:a nuclear} \(\Rightarrow\) \cref{cor:bruce-li-actions:ess nuclear} \(\Rightarrow\) \cref{cor:bruce-li-actions:ess ame}.
\end{corollary}
    
%%%%%%%%%%%%%%%%%%%%%%%%%%%%%%%%%%%%%%%%%%%%%%%%%%%%%%%%%%%%%%%%%%%%%%%
% BIBLIOGRAPHY
%%%%%%%%%%%%%%%%%%%%%%%%%%%%%%%%%%%%%%%%%%%%%%%%%%%%%%%%%%%%%%%%%%%%%%%

\bibliographystyle{amsalpha}
\bibliography{ess-ame.bib}

% \bib, bibdiv, biblist are defined by the amsrefs package.
\begin{bibdiv}
\begin{biblist}

\bib{bkmk-2024-twited-grpds}{article}{
      author={Bardadyn, Krzysztof},
      author={Kwaśniewski, Bartosz},
      author={McKee, Andrew},
       title={{Banach algebras associated to twisted \'etale groupoids:
  simplicity and pure infiniteness}},
        date={2024},
        note={\arxiv {2406.05717}},
}

\bib{brix-gonzalez-hume-li-2025}{article}{
      author={Brix, Keving~Aguyar},
      author={Gonzales, Julian},
      author={Hume, Jeremy~B.},
      author={Li, Xin},
       title={{On Hausdorff covers for non-Hausdorff groupoids}},
        date={2025},
        note={\arxiv {2503.23203}},
}

\bib{brown_c*-algebras_2008}{book}{
      author={Brown, Nathanial~P.},
      author={Ozawa, Narutaka},
       title={C*-algebras and finite-dimensional approximations},
    language={en},
      series={Graduate studies in mathematics},
   publisher={American Mathematical Society},
     address={Providence, R.I},
        date={2008},
      number={v. 88},
        ISBN={978-0-8218-4381-9},
        note={OCLC: ocn180190949},
}

\bib{Brown-Ozawa:Approximations}{book}{
      author={Brown, Nathanial~P.},
      author={Ozawa, Narutaka},
       title={{$C^*$\nobreakdash -algebras and finite-dimensional
  approximations}},
      series={Graduate Studies in Mathematics},
   publisher={Amer. Math. Soc.},
        date={2008},
      volume={88},
        ISBN={978-0-8218-4381-9},
      review={\MR {2391387}},
}

\bib{bruce-li:2024:alg-actions}{article}{
      author={Bruce, Chris},
      author={Li, Xin},
       title={{Algebraic actions I. C*-algebras and groupoids}},
        date={2024},
     journal={Journal of Functional Analysis},
      volume={286},
      number={4},
       pages={57 pp},
}

\bib{BussExel:Fell.Bundle.and.Twisted.Groupoids}{article}{
      author={Buss, Alcides},
      author={Exel, Ruy},
       title={Fell bundles over inverse semigroups and twisted \'etale
  groupoids},
        date={2012},
        ISSN={0379-4024},
     journal={J. Operator Theory},
      volume={67},
      number={1},
       pages={153\ndash 205},
      review={\MR{2881538}},
}

\bib{Buss-Exel-Meyer:Reduced}{article}{
      author={Buss, Alcides},
      author={Exel, Ruy},
      author={Meyer, Ralf},
       title={{Reduced $C^*$\nobreakdash-algebras of Fell bundles over inverse
  semigroups}},
        date={2017},
        ISSN={0021-2172},
     journal={Israel J. Math.},
      volume={220},
      number={1},
       pages={225\ndash 274},
      review={\MR {3666825}},
}

\bib{BussMartinez:Approx-Prop}{article}{
      author={Buss, Alcides},
      author={Mart\'{i}nez, Diego},
       title={Approximation properties of {Fell} bundles over inverse
  semigroups and non-{Hausdorff} groupoids},
        date={2023},
     journal={Adv. Math.},
      volume={431},
       pages={pp.~54},
}

\bib{chung-martinez-szakacs-2022}{article}{
      author={Chyuan~Chung, Yeong},
      author={Mart\'{i}nez, Diego},
      author={Szak\'{a}cs, N\'{o}ra},
       title={Quasi-countable inverse semigroups as metric spaces, and the
  uniform roe algebras of locally finite inverse semigroups},
        date={2022},
}

\bib{CEPSS-2019}{article}{
      author={Clark, Lisa~Orloff},
      author={Exel, Ruy},
      author={Pardo, Enrique},
      author={Sims, Aidan},
      author={Starling, Charles},
       title={Simplicity of algebras associated to non-hausdorff groupoids},
        date={2019},
        ISSN={0002-9947},
     journal={Transactions of the American Mathematical Society},
      volume={372},
      number={5},
       pages={3669\ndash 3712},
}

\bib{Elliott2015OnTC}{article}{
      author={Elliott, George~A.},
      author={Niu, Zhuang},
       title={{On the classification of simple amenable C*-algebras with finite
  decomposition rank}},
        date={2015},
     journal={arXiv: Operator Algebras},
         url={https://api.semanticscholar.org/CorpusID:73596649},
}

\bib{Exel:Inverse_combinatorial}{article}{
      author={Exel, Ruy},
       title={{Inverse semigroups and combinatorial
  $C^*$\nobreakdash-algebras}},
        date={2008},
        ISSN={1678-7544},
     journal={Bull. Braz. Math. Soc. (N.S.)},
      volume={39},
      number={2},
       pages={191\ndash 313},
      review={\MR{2419901}},
}

\bib{ExelPitts}{book}{
      author={Exel, Ruy},
      author={Pitts, David~R.},
       title={Characterizing groupoid {$\rm C^*$}-algebras of non-{H}ausdorff
  \'{e}tale groupoids},
      series={Lecture Notes in Mathematics},
   publisher={Springer, Cham},
        date={[2022] \copyright 2022},
      volume={2306},
        ISBN={978-3-031-05512-6; 978-3-031-05513-3},
         url={https://doi.org/10.1007/978-3-031-05513-3},
      review={\MR{4510931}},
}

\bib{exel-starling:2017:amen-actions}{article}{
      author={Exel, Ruy},
      author={Starling, Charles},
       title={Amenable actions of inverse semigroups},
        date={2017},
        ISSN={0143-3857,1469-4417},
     journal={Ergodic Theory Dynam. Systems},
      volume={37},
      number={2},
       pages={481\ndash 489},
         url={https://doi.org/10.1017/etds.2015.60},
      review={\MR{3614034}},
}

\bib{Khoshkam-Skandalis-reg-rep}{article}{
      author={Khoshkam, Mahmood},
      author={Skandalis, Georges},
       title={Regular representation of groupoid {$C^*$}-algebras and
  applications to inverse semigroups},
        date={2002},
        ISSN={0075-4102,1435-5345},
     journal={J. Reine Angew. Math.},
      volume={546},
       pages={47\ndash 72},
      review={\MR{1900993}},
}

\bib{Kranz:weak-containment}{article}{
      author={Kranz, Julian},
       title={The weak containment problem for \'etale groupoids which are
  strongly amenable at infinity},
        date={2023},
        ISSN={0379-4024,1841-7744},
     journal={J. Operator Theory},
      volume={89},
      number={2},
       pages={349\ndash 360},
      review={\MR{4591645}},
}

\bib{kumjian-diagonals-1986}{article}{
      author={Kumjian, Alexander},
       title={On {$C^\ast$}-diagonals},
        date={1986},
        ISSN={0008-414X,1496-4279},
     journal={Canad. J. Math.},
      volume={38},
      number={4},
       pages={969\ndash 1008},
         url={https://doi.org/10.4153/CJM-1986-048-0},
      review={\MR{854149}},
}

\bib{KwasniewskiMeyer-essential-cross-2021}{article}{
      author={Kwa\'sniewski, Bartosz~Kosma},
      author={Meyer, Ralf},
       title={Essential crossed products for inverse semigroup actions:
  simplicity and pure infiniteness},
        date={2021},
        ISSN={1431-0635,1431-0643},
     journal={Doc. Math.},
      volume={26},
       pages={271\ndash 335},
}

\bib{li-classifiable-cartans-2020}{article}{
      author={Li, Xin},
       title={{Every classifiable simple C*-algebras has a Cartan subalgebra}},
        date={2020},
     journal={Inventiones Mathematicae},
      volume={219},
       pages={653\ndash 699},
}

\bib{neshveyevetal2023}{article}{
      author={Neshveyev, Sergey},
      author={Schwartz, Gaute},
       title={{Non-Hausdorff \'etale groupoids and C*-algebras of left
  cancellative monoids}},
        date={2023},
     journal={M\"unster J. Math.},
      volume={16},
      number={1},
       pages={147\ndash 175},
}

\bib{Pisier}{book}{
      author={Pisier, Gilles},
       title={Tensor products of {$C^*$}-algebras and operator spaces---the
  {C}onnes-{K}irchberg problem},
      series={London Mathematical Society Student Texts},
   publisher={Cambridge University Press, Cambridge},
        date={2020},
      volume={96},
        ISBN={978-1-108-74911-4; 978-1-108-47901-1},
         url={https://doi.org/10.1017/9781108782081},
      review={\MR{4283471}},
}

\bib{Renault80}{book}{
      author={Renault, Jean},
       title={A groupoid approach to {$C\sp{\ast} $}-algebras},
      series={Lecture Notes in Mathematics},
   publisher={Springer, Berlin},
        date={1980},
      volume={793},
        ISBN={3-540-09977-8},
      review={\MR{584266}},
}

\bib{Renault2008CartanSI}{article}{
      author={Renault, Jean},
       title={Cartan subalgebras in {$C^*$}-algebras},
        date={2008},
        ISSN={0791-5578},
     journal={Irish Math. Soc. Bull.},
      number={61},
       pages={29\ndash 63},
         url={https://api.semanticscholar.org/CorpusID:15246817},
      review={\MR {2460017}},
}

\bib{renault-2013}{article}{
      author={Renault, Jean},
       title={Topological amenability is a borel property},
        date={2015},
     journal={Math. Scand.},
      volume={117},
       pages={5\ndash 30},
}

\bib{SimsNotes2020}{book}{
      author={Sims, Aidan},
      editor={Perera, Francesc},
       title={Hausdorff \'etale groupoids and their {C*-algebras}},
      series={Advanced Courses in Mathematics},
   publisher={Birkh\"auser/Springer},
organization={CRM Barcelona},
        date={2020},
}

\bib{szakacs-2021}{article}{
      author={Steinberg, Benjamin},
      author={Szak\'acs, N\'ora},
       title={Simplicity of inverse semigroup and \'etale groupoid algebras},
        date={2021},
     journal={Advances in Mathematics},
      volume={386},
}

\bib{szakacs-2023}{article}{
      author={Steinberg, Benjamin},
      author={Szak\'acs, N\'ora},
       title={{On the simplicity of Nekrashevych algebras of contracting
  self-similar groups}},
        date={2023},
     journal={Mathematische Annalen},
      volume={386},
       pages={1391\ndash 1428},
}

\bib{Thomsen}{article}{
      author={Thomsen, Klaus},
       title={Semi-\'etale groupoids and applications},
        date={2010},
        ISSN={0373-0956,1777-5310},
     journal={Ann. Inst. Fourier (Grenoble)},
      volume={60},
      number={3},
       pages={759\ndash 800},
         url={https://doi.org/10.5802/aif.2539},
      review={\MR{2680816}},
}

\bib{tikuisis-white-winter-2017}{article}{
      author={Tikuisis, Aaron},
      author={White, Stuart},
      author={Winter, Wilhelm},
       title={Quasidiagonality of nuclear {C}*-algebras},
        date={2017},
        ISSN={0003486X},
     journal={Annals of Mathematics},
      volume={185},
      number={1},
       pages={229\ndash 284},
         url={http://www.jstor.org/stable/24906439},
}

\bib{white_cartan_2018}{article}{
      author={White, Stuart},
      author={Willett, Rufus},
       title={Cartan subalgebras in uniform {Roe} algebras},
        date={2020},
     journal={Groups Geom. Dyn.},
      volume={14},
      number={3},
       pages={949\ndash 989},
}

\bib{Willett:Non-amenable}{article}{
      author={Willett, Rufus},
       title={A non-amenable groupoid whose maximal and reduced
  ${C}^*$\nobreakdash -algebras are the same},
        date={2015},
        ISSN={1867-5778},
     journal={M\"unster J. Math.},
      volume={8},
      number={1},
       pages={241\ndash 252},
      review={\MR {3549528}},
}

\bib{winter_nuclear_2010}{article}{
      author={Winter, Wilhelm},
      author={Zacharias, Joachim},
       title={The nuclear dimension of {C}*-algebras},
        date={2010},
     journal={Advances in Mathematics},
      volume={224},
      number={2},
       pages={461\ndash 498},
}

\end{biblist}
\end{bibdiv}

\end{document}